\documentclass[11pt]{amsart}
\title{Global existence and convergence of nondimensionalized incompressible Navier-Stokes equations in low Froude number regime}
\author{Stefano Scrobogna}

\usepackage{times}
\usepackage{fullpage}

\usepackage[T1]{fontenc}
\usepackage[utf8]{inputenc}
\usepackage{amssymb,amsmath,amsfonts}
\usepackage{amssymb,amsthm}
\usepackage{mathtools}

\usepackage[colorlinks=true,linkcolor=blue,citecolor=red]{hyperref}

\usepackage{cite}
\usepackage{color}

\usepackage{tikz}
\usetikzlibrary{arrows,automata,backgrounds,calendar,chains,matrix,mindmap,patterns,petri, shapes.geometric,shapes.misc,spy,trees}
\usetikzlibrary{patterns}
\renewcommand{\d}{\textnormal{d}}

\DeclareMathOperator*{\esssup}{ess\,sup}

\newcommand{{\2}}{{L^2\left(\mathbb{R}^3\right)}}

\newcommand{\Linfty}{L^\infty\left(\mathbb{R}^3\right)}
\newcommand{\dx}{\text{d}{x}}
\DeclareMathOperator{\dive}{{div\hspace{0.7mm}}}
\newcommand{\fine}{\hfill$\blacklozenge$}
\newcommand{\nhp}{\nabla_h^\perp}

\newcommand{\nh}{\nabla_h}
\newcommand{\Dh}{\Delta_h}

\newcommand{\Ro}{\textnormal{Ro}}
\newcommand{\Fr}{\textnormal{Fr}}
\newcommand{\PA}{\mathbb{P}\mathcal{A}}
\newcommand{\uh}{\bar{u}^h}
\newcommand{\oh}{\omega^h}
\newcommand{\diveh}{\textnormal{div}_h \ }
\newcommand{\curlh}{\textnormal{curl}_h}
\newcommand{\Es}{ \dot{\mathcal{E}}^s \left( \R^3 \right)}
\newcommand{\Eud}{ \dot{\mathcal{E}}^{1/2} \left( \R^3 \right)}

\newcommand{\drr}{\delta^\varepsilon_{r,R}}

\newcommand{\RN}[1]{%
  \textup{\uppercase\expandafter{\romannumeral#1}}%
}

\newcommand{\R}{\mathbb{R}}

\newcommand{\sumf}{\sum_{\left| q-q' \right|\leqslant4}}
\newcommand{\sumi}{\sum_{q'>q-4}}
\newcommand{\tq}{{\triangle}_q}

\newcommand{\cFFStq}{{\triangle}_q}
\newcommand{\cFFSSq}{{S}_q}

\newcommand{\B}{\mathcal{B}}
\newcommand{\hra}{\hookrightarrow}
\newcommand{\rhu}{\rightharpoonup}
\newcommand{\loc}{\textnormal{loc}}

\newcommand{\Crr}{\mathcal{C}_{r,R}}

\newcommand{\Hs}{{\dot{H}^s\left( \R^3 \right)}}

\newcommand{\set}[1]{\left\lbrace #1 \right\rbrace}
\newcommand{\Wrr}{W^\varepsilon_{r,R}}

\newcommand{\cFFSLtwo}{{L^2\left(\mathbb{R}^3\right)}}
\newcommand{\cFFSLp}{{L^p\left(\mathbb{R}^3\right)}}

\newcommand{\cFFSHud}{\dot{H}^{\frac{1}{2}}\left( \R^3 \right)}

\newcommand{\cFFSHs}{\dot{H}^{s}\left( \R^3 \right)}

\newcommand{\qg}{{quasi-geostrophic }}
\newcommand{\NS}{Navier-Stokes }

\newcommand{\CC}{\mathbb{C}}

\newcommand{\NN}{\mathbb{N}}
\newcommand{\PP}{\mathbb{P}}

\newcommand{\RR}{\mathbb{R}}

\newcommand{\ZZ}{\mathbb{Z}}

\newcommand{\dd}{\partial}

\newcommand{\abs}[1]{\left\vert#1\right\vert}
\newcommand{\psca}[2]{\left\langle\left. #1 \right| #2 \right\rangle}
\newcommand{\pint}[1]{\left[#1\right]}
\newcommand{\pare}[1]{\left(#1\right)}
\newcommand{\norm}[1]{\left\Vert#1\right\Vert}

\newcommand{\spres}{\sum_{\abs{q'-q}\leqslant 4}}
\newcommand{\sloin}{\sum_{q'> q-4}}

\theoremstyle{theorem}
\newtheorem{theorem}{Theorem}[section]
\newtheorem{prop}[theorem]{Proposition}
\newtheorem{lemma}[theorem]{Lemma}
\newtheorem{cor}[theorem]{Corollary}
\newtheorem{defn}[theorem]{Definition}
\theoremstyle{definition}

\newtheorem{rem}[theorem]{Remark}

\numberwithin{equation}{section}

\AtEndDocument{\bigskip{\footnotesize
  \textsc{BCAM - Basque Center for Applied Mathematics, Mazarredo, 14,  E48009 Bilbao, Basque Country -- Spain} \par
  \textit{E-mail address:}  \texttt{\href{mailto:sscrobogna@bcamath.org}{sscrobogna@bcamath.org}}}}

\begin{document}
 
 \date{\today}

 \maketitle
 
 \begin{abstract}
 We prove that the incompressible, density dependent, \NS\ equations are globally well posed in a low Froude number regime. The density profile is supposed to be increasing in depth and linearized around a stable state. Moreover if the Froude number tends to zero we prove that such system converges (strongly) to a two-dimensional, stratified \NS\ equations with full diffusivity.
  No smallness assumption is considered on the initial data.
 \end{abstract}


 \section{Introduction}
 
 In the present article we study the behavior of strong solutions of the following modified Boussinesq system which represents the nondimensionalized incompressible density-dependent \NS\ equations
\begin{equation}\label{perturbed BSSQ} \tag{PBS$_\varepsilon$}
\left\lbrace
\begin{aligned}
&\partial_t u^\varepsilon  + u^\varepsilon \cdot\nabla u^\varepsilon  -\nu \Delta  u^\varepsilon -\displaystyle\frac{1}{\varepsilon} \rho^\varepsilon \overrightarrow{e}_3 &=& -\displaystyle \frac{1}{\varepsilon} \nabla \Phi^\varepsilon,\\
&\partial_t \rho^\varepsilon + u^\varepsilon \cdot\nabla \rho^\varepsilon  -\nu' \Delta \rho^\varepsilon + \displaystyle\frac{1}{\varepsilon} u^{3,\varepsilon} & =& \;0,\\
&\dive u^\varepsilon =\;0,\\
& \left. \left( u^\varepsilon, \rho^\varepsilon \right)\right|_{t=0}= \left( u_0, \rho_0 \right),
\end{aligned}
\right.
\end{equation} 
where the functions $ U^\varepsilon, \rho^\varepsilon $ depend upon the variables  $ \left( t,x \right)\in \R_+\times \R^3 $ in the regime $ \varepsilon \to 0 $. The space variable $ x $ shall be many times considered separately with respect to the horizontal and vertical components, i.e. $ x= \left( x_h, x_3 \right)= \left( x_1, x_2, x_3 \right) $. In the present paper we denote $ \Delta= \partial_1^2 + \partial_2^2 + \partial_3^2 $ the standard laplacian,  $ \Dh= \partial_1^2+\partial_2^2 $ is the laplacian in the horizontal directions, as well as $ \nh=\left( \partial_1,\partial_2 \right)^\intercal , \nhp=\left( -\partial_2,\partial_1 \right)^\intercal $. In the same way the symbol $ \nabla $ represents the gradient in all space directions $ \nabla= \left( \partial_1, \partial_2, \partial_3 \right) $. Considered a vector field $ w $ we denote $ \dive w= \partial_1 w^1+ \partial_2w^2+\partial_3 w^3 $.
Given two three-components vector fields $ w, z $ the notation $ w\cdot\nabla z  $ indicates the operator
$$
w\cdot \nabla z = \sum_{i=1}^3 w^i\partial_i z.
$$
 Generally for any two-components vector field $u=\left( u^1,u^2 \right)$ we shall denote as $ u^\perp = \left( -u^2,u^1 \right) $. The viscosity $ \nu, \nu' $ above  are strictly positive constants $ \nu, \nu'\geqslant c >0 $
We give in what follows a short physical justification of the system \eqref{perturbed BSSQ}.\\
 The system \eqref{perturbed BSSQ} describes how a stratified fluid reacts in a long time-scale $ T $ to big perturbations around a state of dynamical equilibrium.
To understand in what consist such perturbation let us consider a fluid which is perfectly stratified: gravity tends to minimize the gravitational potential and hence to dispose heavier layers under lighter ones. An equilibrium state is hence a configuration in which the fluid density is a function depending on the vertical variable $ x_3 $ only and it is decreasing in $ x_3 $. Let us suppose now to displace a certain volume of fluid with high density in a higher region with lower density (perturbation of equilibrium). Gravity will induce downward motion and Archimede's principle will provide upward buoyancy. This process induces a periodic motion of frequency $ N $ appearing in the third equation of \eqref{eq:physical_syst}.
 The value $ N $ appearing in the equation for $ \rho $ is called Brunt-V\"ais\"al\"a frequency, and describes the oscillatory behavior induced by the buoyancy which is caused by the stratification in decreasing-density stacks. We suppose $ N $ to be constant, and indeed $ N= T_N^{-1} $ where $ T_N $ is the \textit{characteristic time of stratification}. We define the \textit{Froude number} as
$$
\Fr = \frac{T_N}{T} \ll 1.
$$
The Froude number Fr  quantifies the stratification effects on the dynamics of the fluid; the smaller it is the more relevant such effects are. In fact $ \Fr= T_N / T $ is a ratio which involves time-scales only; the characteristic time of stratification $ T_N $ is an intrinsic magnitude of the system which is determinate by the stratification frequency only, while $ T $ can be chosen as large as the observer desires. It is reasonable hence to think that in very large time-scales $ T $ the periodic motion caused by an induced equilibrium disturbance will somehow disperse, and the fluid will once again recover a configuration of equilibrium. For a more detailed physical discussion on the derivation of \eqref{perturbed BSSQ} we refer to Section \ref{sec:physics}. \\

Let us rewrite the system \eqref{perturbed BSSQ} into the following more compact form
\begin{equation}\tag{\ref{perturbed BSSQ}}
\left\lbrace
\begin{aligned}
& {\partial_t U^\varepsilon}+ u^\varepsilon \cdot \nabla U^\varepsilon - \mathbb{D}U^\varepsilon + \frac{1}{\varepsilon}\mathcal{A}U^\varepsilon = -\frac{1}{\varepsilon} \left( \begin{array}{c}
\nabla \Phi^\varepsilon\\
0
\end{array} \right),\\
& U^\varepsilon=\left( U^\varepsilon,\rho^\varepsilon \right),\\
&\dive u^\varepsilon=0,
\end{aligned}
\right.
\end{equation}
 where $ U^\varepsilon= \left( u^\varepsilon, \rho^\varepsilon \right) $ and
\begin{align}\label{matrici}
\mathcal{A}= & \left( \begin{array}{cccc}
0&0&0&0\\
0&0&0&0\\
0&0&0&1\\
0&0&-1&0\\
\end{array} \right),
&
\mathbb{D} =& \left( \begin{array}{cccc}
\nu\Delta&0&0&0\\
0&\nu\Delta&0&0\\
0&0&\nu\Delta&0\\
0&0&0&\nu'\Delta\\
\end{array} \right).
\end{align}
The above form for the system is the one we shall always adopt. We  use as well the following differential operator
\begin{align}\label{eq:definition_proj_Leray}
\PP= \left( 
\begin{array}{c|c}
\delta_{i,j}-\Delta^{-1}\partial_i\partial_j & 0\\
\hline
0 & 1
\end{array}
 \right)_{i,j=1,2,3}.
\end{align}
The operator $ \PP $ acts in the following way: given a four component vector field $ V=V \left( x \right)= \left( V^1, V^2, V^3, V^3 \right)= \left( V', V^4 \right) $ it maps $ V' $ onto a divergence-free vector field  and leaves untouched $ V^4 $, i.e.
\begin{equation*}
\PP V = \left(
\begin{array}{cc}
V' - \Delta^{-1} \nabla \dive V', & V^4
\end{array}
  \right).
\end{equation*}
We underline that $ \PP $ and $ \mathbb{D} $ commute. We shall use this property (even implicitly) repeatedly all along the present work.

\noindent

To the best of our knowledge there are not many works concerning the system \eqref{perturbed BSSQ}. In \cite{embid_majda2} P. Embid and A. Majda study the distributional limit of the primitive equations
\begin{equation}
\label{eq:PE}\tag{PE}
\left\lbrace
\begin{aligned}
& 
\partial_t u^{\Ro,\Fr} + u^{\Ro,\Fr} \cdot \nabla u^{\Ro,\Fr} -\nu \Delta u^{\Ro,\Fr}
+\frac{1}{\Ro} u^{\Ro,\Fr} \wedge \overrightarrow{e_3}
 - \frac{1}{\Fr}\rho^{\Ro,\Fr} \overrightarrow{e_3}= -  \nabla P^{\Ro,\Fr},
\\
 & \partial_t \rho^{\Ro,\Fr} +u^{\Ro,\Fr} \cdot \nabla \rho^{\Ro,\Fr} - \nu' \Delta \rho^{\Ro,\Fr} + \frac{1}{\Fr}\ u^{3,\Fr} =0,\\
& \dive u^{\Ro,\Fr} =0,\\
& \left. \left( u^{\Ro,\Fr}, \rho^{\Ro,\Fr} \right) \right|_{t=0}= \left( u_0, \rho_0 \right).
\end{aligned}
\right.
\end{equation}
In the regimes $ \Ro, \Fr = \mathcal{O} \left( \varepsilon \right) $ and $ \Ro \gg \Fr = \mathcal{O} \left( \varepsilon \right) $ in the case in which the domain is periodic-in-space. The value $ \Ro $ is called the \textit{Rossby number} and quantifies the influence of the rotation on the motion of a fluid in the same way as the Froude number quantifies the stratification effects.   For a formal derivation of \eqref{eq:PE} in the case $ \Ro, \Fr = \mathcal{O} \left( \varepsilon \right) $ we refer to the  introduction of \cite{Charve_thesis}. Concerning always the equations \eqref{eq:PE} in the regime $ \Ro, \Fr = \mathcal{O} \left( \varepsilon \right) $ in the whole space we refer to the pioneering work \cite{chemin_prob_antisym} in which J.-Y. Chemin proves that \eqref{eq:PE} is globally well posed in the case in which $ \Fr= \Ro = \varepsilon $, only a certain part of the initial datum is small and the difference $ |\nu-\nu'| $ is small. Moreover, in \cite{charve1} and \cite{charve2}, F. Charve using dispersive tools (Strichartz estimates) proves that \eqref{eq:PE} in the regime $ \Ro, \Fr = \mathcal{O} \left( \varepsilon \right), \ \Ro\neq \Fr $ are globally well posed and converge to a suitable limit system known as the \qg\ system without any smallness assumption on the initial data.

The system \eqref{perturbed BSSQ} is very close to the system \eqref{eq:PE} in the regime $ \Ro \gg \Fr = \mathcal{O} \left( \varepsilon \right) $, we refer to \cite{embid_majda2} and references therein for a justification of such fact. 
Recently 
 K. Widmayer proved in \cite{Widmayer_Boussisnesq_perturbation}  that the inviscid counterpart of \eqref{perturbed BSSQ} converges locally in $ {\cFFSLtwo} $ to a stratified two-dimensional Euler system. \\
The case in which \eqref{perturbed BSSQ} evolves in a periodic domain is treated in \cite{Scrobo_low_Froude_periodic}.

\subsection{The functional setting}
In this section we introduce the functional setting that we  adopt all along the paper.   We  define the homogeneous Sobolev space $ \dot{H}^s \left( \R^d \right), \ s\in \R $ as the space of tempered distributions $ u $ on $ \R^d $ whose Fourier transform $ \hat{u}\in L^1_{\loc} \left( \R^d \right) $ and such that
\begin{equation*}
\left\| u \right\|_{\dot{H}^s \left( \R^d \right)}^2
= \norm{ \left( -\Delta \right)^{s/2} u }_{L^2 \left( \R^d \right)}^2
 = \int_{\R^d} \left| \xi \right|^{2s} \left| \hat{u} \left( \xi \right) \right|^2 \d \xi < \infty.
\end{equation*}
Since we intend to study the behavior of solutions of nonlinear partial differential equations we are interested to understand the regularity of a product of distributions. Generally a product of distributions is not well defined as it was first proved in \cite{Schwartz54}. This is no longer true if the distributions considered belong to some suitable homogeneous Sobolev space;
\begin{lemma}\label{lem:prod_rules_Sobolev}
Let $ u \in \dot{H}^{s_1} \left( \R^d \right), \ v \in \dot{H}^{s_2} \left( \R^d \right) $ where $ s_1, s_2 <d/2 $ and $ s_1+s_2 >0 $. Then 
\begin{equation*}
\norm{ u \ v}_{ \dot{H}^{s_1 + s_2 -\frac{d}{2}} \left( \R^d \right)}
\leqslant C_{s_1, s_2} \norm{u}_{ \dot{H}^{s_1} \left( \R^d \right)}
\norm{v}_{ \dot{H}^{s_2} \left( \R^d \right)},
\end{equation*}
or, equivalently, the point-wise  multiplication maps continuously $  \dot{H}^{s_1} \left( \R^d \right) \times \dot{H}^{s_2} \left( \R^d \right)   $ to \ $ \dot{H}^{s_1 + s_2 -\frac{d}{2}} \left( \R^d \right) $.
\end{lemma}

Homogeneous Sobolev spaces are Hilbert spaces if and only if $ s<d/2 $, in this case the scalar product of two elements of $ \dot{H}^s \left( \R^d \right) $ is defined as
\begin{equation}
\label{eq:Hs_sp}
\begin{aligned}
\left( \left. u \right| v \right)_{{\cFFSHs}} = & \ \int_{\R^d} \left[ \left( -\Delta \right)^{s/2} u \left( x \right) \right] \cdot \left[ \left( -\Delta \right)^{s/2} v \left( x \right) \right]  \dx,\\
= & \ \int_{\R^d_\xi} \left( \left| \xi \right|^s \hat{u} \left( \xi \right) \right) \cdot \left( \left| \xi \right|^s \hat{v} \left( \xi \right) \right)  \d \xi.
\end{aligned}
\end{equation}
We refer to \cite[Chapter 1]{bahouri_chemin_danchin_book} for a counterexample in the case in which $ s\geqslant d/2 $.\\
The norm of $ \cFFSHs $ does not take in account the behavior of $ u $ in a frequency set close to zero, the non-homogeneous Sobolev space $ H^s \left( \R^d \right) $ defined as
\begin{equation*}
H^s \left( \R^d \right) =\set{ u \in \mathcal{S}' \left( \R^d \right)
\left| \ 
\left\| \left( 1-\Delta \right)^{s/2} u \right\|_{L^2 \left( \R^d \right)} < \infty 
\right. },
\end{equation*}
gives a deeper description of the tempered distribution $ u $ and a mean to control the low-frequencies as well. Nonetheless we shall work constantly with homogeneous Sobolev spaces since the propagation of a critical homogeneous Sobolev regularity suffice to deduce smoothness of solutions, we shall briefly explain such fact. It is moreover interesting to notice that
\begin{equation*}
H^s \left( \R^3 \right) = \cFFSLtwo \cap \cFFSHs, \text{ if } s\geqslant 0.
\end{equation*}

We shall see that the solutions of \eqref{perturbed BSSQ}, if $ \varepsilon $ is sufficiently close to zero, develop a behavior which is radically different along the horizontal direction $ x_h $ and the vertical $ x_3 $. This motivates the introduction of the following \textit{anisotropic Lebesgue spaces}, which are spaces of function whose integrability differs along horizontal and vertical directions. The anisotropic Lebesgue spaces $L_h^p\left( L_v^q \right)$ with $p, q \geq 1$ are defined as
\begin{multline*}
	L_h^p\left( L_v^q(\RR^3) \right) = L^p(\RR^2_h;L^q(\RR_v))
	\\ = \set{ u\in \mathcal{S'}\left| \norm{u}_{L_h^pL_v^q} = \Big[ \int_{\RR^2_h} \Big\vert \int_{\RR_v}\abs{u(x_h,x_3)}^qdx_3 \Big\vert^{\frac{p}q}dx_h \Big]^{\frac{1}p} < +\infty\right. }.
\end{multline*}
Here, the order of integration is important. Indeed, if $1 \leq p \leq q$ and if $u: X_1\times X_2 \rightarrow \RR$ is a
function in $L^p(X_1;L^q(X_2))$, where $(X_1,d\mu_1)$, $(X_2,d\mu_2)$ are measurable spaces, then $u \in L^q(X_2;L^p(X_1))$ and
\begin{equation}\label{eq:LpLqanisotropic ineq}
	\norm{u}_{L^q(X_2;L^p(X_1))} \leq \norm{u}_{L^p(X_1;L^q(X_2))}.
\end{equation}
Obviously we can define in  a symmetric way the space $ L^p_v \left( L^q_h \right)= L^p \left( \R_v; L^q \left( \R^2_h \right) \right) $. We shall be interested to study spaces of the kind $ L^\infty_ v \left( L^q_h \right) $, $ q\in [1, \infty] $, they are indeed defined as the tempered distributions such that 
\begin{equation*}
\left\| u \right\|_{L^\infty_ v \left( L^q_h \right)} = \esssup_{x_3\in \R_v} \left\| u \left( \cdot, x_3 \right) \right\|_{L^q \left( \R^2_h \right)}<\infty.
\end{equation*}
In a similar way we can define $ L^p_h \left( L^\infty_v \right) $ spaces via the norm
\begin{equation*}
\left\| u \right\|_{L^p_h \left( L^\infty_v \right)}=
\left( \int_{\R^2_h} \esssup_{x_x\in \R_v} \left| u \left( x_h, x_3 \right) \right|^p \d x_h \right)^{\frac{1}{p}}.
\end{equation*}
 It is well known (see \cite{Chemin92}, for instance) that as long as \NS\ equations can propagate $ \cFFSHud $ data the solutions are in fact regular, for this reason it makes sense to define the following space for $ s>0 $:
\begin{equation*}
\dot{\mathcal{E}}^s_T \left( \R^d \right) = \mathcal{C} \left( [0,T); \dot{H}^s \left( \R^d \right) \right) \cap L^2 \left( [0,T); \dot{H}^{s+1} \left( \R^d \right) \right),
\end{equation*}
and since we are interested in global-in-time regularity 
we define hence the space 
$$ \dot{\mathcal{E}}^s \left( \R^d \right)= \dot{\mathcal{E}}^s_\infty \left( \R^d \right). $$

\section{Statement of the main result and preliminaries}

Before stating the main result let us mention that, as $ \mathcal{A} $ defined in \eqref{matrici} appearing in \eqref{perturbed BSSQ} is skew-symmetric it does not bring any energy in suitable energy spaces built on $ L^2 $, such as homogeneous and non-homogeneous Sobolev spaces $ \dot{H}^s, H^s, \ s\in \R $ and Besov spaces $ B^s_{2, r}, \ s\in \R, r\geqslant 1 $. We refer to \cite{bahouri_chemin_danchin_book} for a detailed definition and a deep description of Besov spaces in the whole space. This implies in particular that Fujita-Kato and Leray theorem \cite{bahouri_chemin_danchin_book}, \cite{Leray} can be applied on system \eqref{perturbed BSSQ}. At first we state the following theorem \textit{\`a la Leray}:

\begin{theorem}\label{thm:Leray_applied}
Let $ U_0 $ be in $ {\cFFSLtwo} $, for each $ \varepsilon >0 $ 
there exist a sequence $ \left( U^\varepsilon \right)_{\varepsilon > 0} $ such that, for each $ \varepsilon > 0 $, the function $ U^\varepsilon $ is a distributional solution of \eqref{perturbed BSSQ} with initial data $ U_0 $. Moreover the sequence $ \left( U^\varepsilon \right)_{\varepsilon >0} $ is uniformly bounded in $ \dot{\mathcal{E}}^0 \left( \R^3 \right) $.
\end{theorem}

Let us now state a result of existence in the homogeneous Sobolev setting, for a proof we refer to \cite{bahouri_chemin_danchin_book},
\begin{theorem}\label{thm:FK_applied}
Let us suppose $ U_0 \in  {\cFFSHud} $, then
there exists a maximal time $  T_{U_0} $ independent of $ \varepsilon $ such that, for each $ T\in \left[0, T_{U_0}\right) $, there exists a unique solution $ U^\varepsilon $ of \eqref{perturbed BSSQ} in $ L^4 \left( [0,T ]; \dot{H}^1 \left( \R^3 \right) \right) $ which also  belongs to the space $ \dot{\mathcal{E}}^{1/2}_T \left( \R^3 \right) $.
\begin{itemize}
 \item  If the initial data is small in the space $ {\cFFSHud} $, i.e. if $ \left\| U_0 \right\|_{{\cFFSHud}}\leqslant \tilde{c}\min \{ \nu, \nu'\} $, then $ T_{U_0}= \infty $.
 
 \item If $ T_{U_0} $ is finite then
 \begin{equation} \label{eq:BU_condition}
 \int_0^{T_{U_0}} \left\| U^\varepsilon \left( \tau \right) \right\|_{\dot{H}^1 \left( \R^3 \right)}^4 \d \tau =\infty.
 \end{equation}
\end{itemize}
\end{theorem}
Theorem \ref{thm:FK_applied} states that there exist always local, strong solutions for the system \eqref{perturbed BSSQ}, moreover if the $ \dot{H}^{1/2}$ initial data is small with respect to the viscosities characterizing the system (namely if the constant $ \tilde{c} $ in Theorem \ref{thm:FK_applied} is small) the solution is global. \\

The blow-up condition \eqref{eq:BU_condition} gives already an insight on how to connect the critical homogeneous Sobolev regularity $ \cFFSHud $ with its non-homogeneous counterpart. Indeed by interpolation of Sobolev spaces we can argue that
\begin{equation*}
\left\| U^\varepsilon \left( \tau \right) \right\|_{\dot{H}^1 \left( \R^3 \right)}^4 \leqslant 
\left\| U^\varepsilon \left( \tau \right) \right\|_{\cFFSHud}^2
\left\| \nabla U^\varepsilon \left( \tau \right) \right\|_{\cFFSHud}^2,
\end{equation*}
whence to control the $ \Eud $ norm impose a control on the blow-up condition \eqref{eq:BU_condition}. 
 The following result stems in a relatively simple way from Theorem \ref{thm:FK_applied}:
\begin{cor}\label{cor:FK}
Let $ U_0 \in H^s \left( \R^3 \right), \ s\geqslant 1/2 $, then the unique solution $ U^\varepsilon $ identified by Theorem \ref{thm:FK_applied} satisfies the following inequality, for each $ t \in \left( 0, T_{U_0} \right) $:
\begin{equation*}
\left\| U^\varepsilon \left( t \right) \right\|_{H^s \left( \R^3 \right)}^2
+c \int_0^t \left\| \nabla U^\varepsilon \left( \tau \right) \right\|_{H^s \left( \R^3 \right)}^2 \d \tau 
\leqslant C 
\left\| U_0 \right\|_{H^s \left( \R^3 \right)}^2 \exp \set{  \frac{C}{\nu^3} \int_0^{t} \left\| U^\varepsilon \left( \tau \right) \right\|_{\dot{H}^1 \left( \R^3 \right)}^4 \d \tau }.
\end{equation*}
\end{cor}

The proof of Corollary \ref{cor:FK} is rather simple, but it involves the tools of paradifferential calculus and Bony decomposition, for this reason is postponed to Section \ref{paradiffcalcul_FFS}.\\

Whence the blow-up condition \eqref{eq:BU_condition} which is calibrated in the critical homogeneous Sobolev setting suffices to determinate non-homogeneous subcritical regularity, and the maximal lifespan in Corollary \ref{cor:FK} is hence the same one as in Theorem \ref{thm:FK_applied}. In what follows it suffices  hence that we focus on the propagation of homogeneous critical Sobolev regularity. \\

Let us introduce now the following two-dimensional, vertically stratified \NS\ system:

\begin{equation}\label{eq:2DNS_stratified_1}
\left\lbrace
\begin{aligned}
& \partial_t \uh  \left( x_h, x_3 \right) + \uh \left( x_h, x_3 \right) \cdot \nh \uh \left( x_h, x_3 \right) - \nu \Delta \uh \left( x_h, x_3 \right) = -\nh \bar{p} \left( x_h, x_3 \right)\\
& \diveh \uh  \left( x_h, x_3 \right) =0,\\
& \left. \uh \left( x_h, x_3 \right) \right|_{t=0}=\mathbb{P}_0 \ U_0 \left( x_h, x_3 \right)= \uh_0 \left( x_h, x_3 \right).
\end{aligned}
\right.
\end{equation}
Such system will be studied in detail in Section \ref{sec:gwp_limit} and it represents the non-oscillating part of the solutions of \eqref{perturbed BSSQ} in the limit $ \varepsilon\to 0 $.\\

A question of great importance in the study  of hydro-dynamical systems is whether three-dimensional hydro-dynamical flows admit classical solutions which are globally well-defined. For two-dimensional systems the answer is affirmative and it is known since the classical works \cite{Lad58} and \cite{LionsProdi}. In dimension three the question of global solvability for generic large data remain unsolved. Nonetheless there exist many three-dimensional systems which admit global-in-time solution of strong type for arbitrary data, notably geophysical fluids \cite{monographrotating} belong to such category due to the constraining effects of the rotation of the Earth. The system \eqref{perturbed BSSQ} can be studied with the methodologies characterizing such discipline and we prove the following result:

\begin{theorem}\label{thm:main_result}
Let $ U_0 \in H^{\frac{1}{2}}\left( \R^3 \right)$, such that $ \oh=-\partial_2 u^1_0 + \partial_1 u^2_0\in {\cFFSLtwo} $, there exists a $ \varepsilon_0 > 0 $ such that for each $ \varepsilon \in \left( 0, \varepsilon_0 \right) $ the unique local solution $ U^\varepsilon $ of \eqref{perturbed BSSQ} is in fact global and belongs to the space
\begin{equation*}
U^\varepsilon \in L^\infty \left( \R_+; \cFFSHud \right)\cap L^2 \left( \R_+; \dot{H}^{\frac{3}{2}} \left( \R^3 \right) \right)= \Eud.
\end{equation*}
 Moreover as $ \varepsilon \to 0 $ the following convergence takes place,
\begin{align*}
U^\varepsilon - W^\varepsilon - \left( \uh, 0, 0 \right)^\intercal \xrightarrow{\varepsilon\to0} 0, &&
\text{in the space } \Eud,
\end{align*}
where $ U^\varepsilon $ is the strong solution identified by Theorem \ref{thm:FK_applied} and $ W^\varepsilon, \uh $ are respectively the unique global solutions (in the space $ \Eud $) of 
\begin{equation}
\label{eq:free_wave_total_initial_data}
\left\lbrace
\begin{aligned}
& \partial_t W^\varepsilon -\mathbb{D}W^\varepsilon + \frac{1}{\varepsilon} \PA W^\varepsilon =
\left( 
\begin{array}{c}
0\\
0\\
-\partial_3 \left( -\Dh \right)^{-1}\diveh \left( \uh\cdot \nh \uh \right)\\
0
\end{array}
 \right)
, \\
& \left. W^\varepsilon\right|_{t=0}=\left( \PP_{-, \varepsilon}+\mathbb{P}_{+, \varepsilon} \right)U_0,
\end{aligned}
\right.
\end{equation}
 and \eqref{eq:2DNS_stratified_1}. The operators $ \PP_{\pm, \varepsilon} $ are defined in \eqref{eq:projectors}.  
\end{theorem}

\begin{rem}
We point out that Theorem \ref{thm:main_result} is composed of two main statements:
\begin{enumerate}
\item global well-posedness in the energy space $ \Eud $ for positive, small $ \varepsilon $,

\item convergence as $ \varepsilon \to 0 $ in the space $ \Eud $ to the solutions of a suitable limit system,
\end{enumerate}
we prove at first the global well-posedness, and subsequently, thanks to the theory developed in order to prove such result, we prove the convergence result. \fine
\end{rem}

To prove Theorem \ref{thm:main_result} we  proceed as follows:

\begin{itemize}
\item 
In Section \ref{sec:spectral_analysis} we perform a careful spectral analysis of the linear operator $ L_\varepsilon = \PA - \varepsilon \mathbb{D} $, where $ \mathbb{P} $ is the Leray projector onto the first three-components and the identity on the fourth and $ \mathcal{A}, \ \mathbb{D} $ are defined in \eqref{matrici}. Such analysis will be of great relevance in the Sections \ref{sec:dispersive_properties} and \ref{sec:gwp_limit}.

\item
In Section \ref{sec:gwp_limit} we prove that the system \eqref{eq:2DNS_stratified} is globally well posed in $ \Es, \ s\geqslant 0 $. The system \eqref{eq:2DNS_stratified} is the system to whom \eqref{perturbed BSSQ} approaches as $ \varepsilon\to 0 $.

\item
In Section \ref{sec:dispersive_properties} we study the linear  system \eqref{eq:free_wave_CFFS}. The initial data of \eqref{eq:free_wave_CFFS} is considered to be in what we denote as the oscillating subspace of $ L_\varepsilon $, which is introduced at the end of Section \ref{sec:spectral_analysis}, and moreover is localized in a a set $ \Crr $ (see \eqref{eq:CrR})  of the frequency space which makes his evolution to be described by an oscillating integral with no stationary phase. This observation is hence the key observation which allows us to prove some adapted dispersive estimates on the solutions of \eqref{eq:free_wave_CFFS}.

\item
In Section \ref{sec:bootstrap} we prove the global well-posedness part of Theorem \ref{thm:main_result}, i.e. we prove that for $ \varepsilon $ sufficiently small the solution $ U^\varepsilon $ of \eqref{perturbed BSSQ} belongs to the space $ \Eud $. To do so we perform a bootstrap argument on the function $ \drr= U^\varepsilon -\Wrr-\bar{U} $ which requires the use of the dispersive estimates performed in Section \ref{sec:dispersive_properties}.

\item 
Finally in Section \ref{sec:main_result} we prove the convergence part of the statement of Theorem \ref{thm:main_result}.
\end{itemize}

\begin{rem}
All along the paper we shall denote with $ C $ a generic positive constant, independent by any parameter. Such value may differ from line to line. The positive constant $ C_{r, R} $ depends instead from the parameter $ 0<r<R $, and 
\begin{equation*}
C_{r, R}\leqslant C \left( 1+ \frac{R^N}{r^N} \right),
\end{equation*}
for some positive and finite  $ N\in \NN $.\fine
\end{rem}

\subsection{Physical derivation of the system \eqref{perturbed BSSQ} and related results}\label{sec:physics}
 The system describing the motion of a fluid with variable density under the effects of (external) gravitational force is (see \cite{cushman2011introduction})
\begin{equation}\label{eq:physical_syst}
\left\lbrace
\begin{aligned}
&   \partial_t u^1 + {u}\cdot \nabla u^1   &&\;= - \frac{1}{\rho_0}\partial_1 p + \nu \Delta u^1,\\
&   \partial_t  u^2 + {u}\cdot \nabla u^2   &&\;= -\frac{1}{\rho_0} \partial_2 p + \nu \Delta u^2,\\
&   \partial_t  u^3 + {u}\cdot \nabla u^3
 + \frac{ \text{g}\; \rho }{\rho_0}
   &&\;= -\frac{1}{\rho_0} \partial_3 p + \nu \Delta u^3 ,\\
& \partial_t \rho + {u}\cdot \nabla \rho
 && \;=\; \kappa \Delta \rho,\\
& \dive {u}=0.
\end{aligned}
\right.
\end{equation}

The system \eqref{eq:physical_syst} is derived considering  a density-dependent incompressible fluid whose only external force acting on it is the gravity. The fluid density $ \rho $ is linearized around a vertically decreasing density profile
\begin{equation}\label{intro:eq:BH1}
\rho \left( t, x \right)= \rho_0 + \bar{\rho} \left( x_3 \right) + \rho' \left( t, x \right),
\end{equation}
with
\begin{equation}\label{intro:eq:BH2}
\left|  \rho' \right| \ll 1.
\end{equation}

The \NS\ incompressible equations with hypothesis \eqref{intro:eq:BH1}--\eqref{intro:eq:BH2} and Boussinesq approximation read hence as

\begin{equation}\label{intro:eq:strat_fluids1}
\left\lbrace
\begin{aligned}
& \partial_t u + u\cdot \nabla u && = && -\frac{1}{\rho_0} \nabla p +\frac{1}{\rho_0} \ \nu \Delta u - 
\left( 
\begin{array}{c}
0 \\ 0 \\
\frac{\rho}{\rho_0}\ {g} 
\end{array}
 \right)
,\\
& \partial_t \rho' + u \cdot \nabla \rho' + u^3 \partial_3 \bar{\rho} && = && \kappa \Delta \rho' + \kappa \partial_3^2 \bar{\rho},\\
&\dive u=0. 
\end{aligned}
\right.
\end{equation}

Let us define the Brunt-V\"ais\"al\"a stratification frequency as (see \cite{cushman2011introduction}):
\begin{equation*}
\R \ni N^2 = -\frac{g}{\rho_0}\ \partial_3 \bar{\rho} \ \Longrightarrow \ \partial_3^2\bar{\rho}\approx 0,
\end{equation*}
equation \eqref{intro:eq:strat_fluids1} becomes (denoting $ \rho' $ as $ \rho $)
\begin{equation}
\tag{\ref{eq:physical_syst}}
\left\lbrace
\begin{aligned}
&   \partial_t u^1 + {u}\cdot \nabla u^1   &&\;= - \frac{1}{\rho_0}\partial_1 p + \nu \Delta u^1,\\
&   \partial_t  u^2 + {u}\cdot \nabla u^2   &&\;= -\frac{1}{\rho_0} \partial_2 p + \nu \Delta u^2,\\
&   \partial_t  u^3 + {u}\cdot \nabla u^3
 + \frac{ \text{g}\; \rho }{\rho_0}
   &&\;= -\frac{1}{\rho_0} \partial_3 p + \nu \Delta u^3 ,\\
& \partial_t \rho + {u}\cdot \nabla \rho
- \frac{\rho_0 \; N^2}{\text{g}}\; u^3
 && \;=\; \kappa \Delta \rho,\\
& \dive {u}=0.
\end{aligned}
\right.
\end{equation}

\noindent Equation \eqref{eq:physical_syst} describe hence the dynamics of a density-dependent fluid under the sole  hypothesis that the variation of density is small and the density increase with depth.\\

\noindent Let us now nondimensionalize equations \eqref{eq:physical_syst}, defining
\begin{equation*}
\left\lbrace
\begin{aligned}
& L = \text{ standard legth-scale of the system},\\
& U = \text{ standard velocity of the flow},\\
& T_N = N^{-1},\\
& T = L/U, \\
& \Fr = T_N /T,
\end{aligned}
\right.
\end{equation*}
we can define the following adimensional unknowns:
\begin{equation*}
\left\lbrace
\begin{aligned}
& t^\star = t/T,\\
& x^\star = x/L,\\
& u^\star = u/U, \\
& p^\star = \frac{p}{\rho \ U^2},\\
& \rho^\star = \frac{g}{\rho_0 NU} \ \rho.
\end{aligned}
\right.
\end{equation*}

\noindent The dimensionless number Fr is the \textbf{Froude number} as is was defined above. The equations \eqref{eq:physical_syst} in nondimensional form becomes
\begin{equation*}
\left\lbrace
\begin{aligned}
& \partial_t u^\star + u^\star \cdot \nabla u^\star - \nu^\star \Delta u^\star  
+
\left( \begin{array}{c}
0 \\ 0\\ \displaystyle\frac{1}{\Fr}\ \rho^\star
\end{array} \right)=-\nabla p^\star,\\
& \partial_t \rho^\star + u^\star \cdot \nabla \rho^\star -\kappa^\star \Delta \rho^\star - \frac{1}{\Fr} \ u^{3, \star} =0,\\
& \dive u^\star=0,
\end{aligned}
\right.
\end{equation*}
where $ \nu^\star $ and $ \kappa^\star $ are modified kinematic viscosities. Setting $ \Fr=\varepsilon $ we hence derived the system 
\begin{equation} \tag{PBS$_\varepsilon$}
\left\lbrace
\begin{aligned}
&\partial_t u^\varepsilon  + u^\varepsilon \cdot\nabla u^\varepsilon  -\nu \Delta  u^\varepsilon -\displaystyle\frac{1}{\varepsilon} \rho^\varepsilon \overrightarrow{e}_3 &=& -\displaystyle \frac{1}{\varepsilon} \nabla \Phi^\varepsilon,\\
&\partial_t \rho^\varepsilon + u^\varepsilon \cdot\nabla \rho^\varepsilon  -\nu' \Delta \rho^\varepsilon + \displaystyle\frac{1}{\varepsilon} u^{3,\varepsilon} & =& \;0,\\
&\dive u^\varepsilon =\;0.
\end{aligned}
\right.
\end{equation}

The system \eqref{perturbed BSSQ} falls into a wider category of mathematical problems known as \textit{singular perturbation problems}. The idea behind this kind of problems is that, once we have an external, linear, force acting on a system such as in \eqref{perturbed BSSQ}, such force with great magnitude will constraint the motion of the system described. This kind of rigidity can hence be used in order to prove that suitable three-dimensional hydrodynamical flows are globally well-posed without any smallness assumption on the initial data.\\

Another system which falls in the category of singular perturbation problems is the system describing the motion of a flow under the effect of a strong horizontal rotation, namely what is known as the Navier-Stokes-Coriolis equations:
\begin{equation}\label{eq:NSC}
\tag{NSC$_\varepsilon$}
\left\lbrace
\begin{aligned}
& \partial_t v^\varepsilon_{\text{RF}} + v^\varepsilon_{\text{RF}} \cdot \nabla v^\varepsilon_{\text{RF}} - \nu \Delta v^\varepsilon_{\text{RF}} + \frac{e^3 \wedge v^\varepsilon_{\text{RF}}}{\varepsilon }= - \frac{1}{\varepsilon}\nabla p^\varepsilon_{\text{RF}},\\
& \dive v^\varepsilon_{\text{RF}}=0,\\
& \left. v^\varepsilon_{\text{RF}}\right|_{t=0}=v_{\text{RF}, 0}.
\end{aligned}
\right.
\end{equation}
In the case in which the spatial domain is the full three-dimensional space $ \R^3 $, if the initial data is of the form
\begin{equation*}
v_{\text{RF}, 0}= \uh_{2 \text{D}, 0} + \tilde{u}_{3 \text{D}, 0} ,
\end{equation*}
where $ \uh_{2 \text{D}, 0} $ is a two-dimensional vector field it is
 proved in \cite{monographrotating} that
\begin{align*}
v^\varepsilon_{\text{RF}} - w^\varepsilon - \uh_{2\text{D}} \to & \ 0, & \text{in } L^\infty \left( \R_+; {\cFFSHud} \right),\\
\nabla\left( v^\varepsilon_{\text{RF}} - w^\varepsilon - \uh_{2\text{D}} \right) \to & \ 0, & \text{in } L^2 \left( \R_+; {\cFFSHud} \right),
\end{align*}
where $ w^\varepsilon $ is the global solution of the linear homogeneous equation associated to \eqref{eq:NSC} and $ \uh_{2\text{D}} $ is the global solution of the two-dimensional \NS\ equations.
\\
Many more are the works on global existence and convergence for the equation \eqref{eq:NSC}: in the whole space it was proved in \cite{CDGG2} a result of global existence and convergence in Sobolev spaces of anisotropic type in the case in which the vertical diffusivity is null. Such result is physically significant since experimental proof suggests that for fluids at a planetary scale the vertical diffusivity (Ekman number) tends to be very small, see \cite{cushman2011introduction} and \cite{Pedlosky87}. We mention as well \cite{ekmanwell}, \cite{ekmanperiodic} and \cite{ekmanill} for works describing rotating fluids between two parallel rigid layers with Dirichelet boundary conditions, \cite{VSN} for rotating fluids with zero vertical diffusivity and vanishing horizontal diffusivity and \cite{Dutrifoy2} for propagation of tangential regularity in rotating inviscid fluids.

\subsection{Dyadic decomposition}We recall that in $\RR^d$, with $d\in\NN^*$, for $R > 0$, the ball $\mathcal{B}_d(0,R)$ is the set $$\mathcal{B}_d(0,R) = \set{\xi \in \RR^d \;:\; \abs{\xi} \leq R}.$$ For $0 < r_1 < r_2$, we defined the annulus
$$\mathcal{A}_d(r_1,r_2) \stackrel{\text{\tiny def}}{=} \set{\xi \in \RR^d \;:\; r_1 \leq \abs{\xi} \leq r_2}.$$ Next, we recall the following Bernstein-type lemma, which states that Fourier multipliers act almost as homotheties on distributions whose Fourier transforms are supported in a ball or an annulus. We refer the reader to \cite[Lemma 2.1.1]{chemin_book} or \cite[Lemma 2.1]{bahouri_chemin_danchin_book} for a proof of this lemma.
\begin{lemma}
	\label{lemma:Bernstein}
	Let $k\in\NN$, $d \in \NN^*$ and $R, r_1, r_2 \in \RR$ satisfy $0 < r_1 < r_2$ and $R > 0$. There exists a constant $C > 0$ such that, for any $a, b \in \RR$, $1 \leq a \leq b \leq +\infty$, for any $\lambda > 0$ and for any $u \in L^a(\RR^d)$, we have
    \begin{equation}
        \label{eq:Bernstein1}
		\mbox{supp\,}\pare{\widehat{u}} \subset \mathcal{B}_d(0,\lambda R) \quad \Longrightarrow \quad \sup_{\abs{\alpha} = k} \norm{\dd^\alpha u}_{L^b} \leq C^k\lambda^{k+ d \pare{\frac{1}a-\frac{1}b}} \norm{u}_{L^a},
    \end{equation}
    \begin{equation}
        \label{eq:Bernstein2}
		\mbox{supp\,}\pare{\widehat{u}} \subset \mathcal{A}_d(\lambda r_1,\lambda r_2) \quad \Longrightarrow \quad C^{-k} \lambda^k\norm{u}_{L^a} \leq \sup_{\abs{\alpha} = k} \norm{\dd^\alpha u}_{L^a} \leq C^k \lambda^k\norm{u}_{L^a}.
    \end{equation}
\end{lemma}

In order to define the (non-homogeneous) dyadic partition of unity, we also recall the following proposition, the proof of which can be found in \cite[Proposition 2.1.1]{chemin_book} or \cite[Proposition 2.10]{bahouri_chemin_danchin_book}. 
\begin{prop}
	\label{pr:dyadic} Let $d \in \NN^*$. There exist smooth radial function $\chi$ and $\varphi$ from $\RR^d$ to $[0,1]$, such that
	\begin{gather}
		\label{eq:dyadic01} \mbox{supp\,}\chi \in \mathcal{B}_d\pare{0,\frac{4}3}, \quad \mbox{supp\,}\varphi \in \mathcal{A}_d\pare{\frac{3}4,\frac{8}3},\\
		\label{eq:dyadic02} \forall\, \xi \in \RR^3, \quad \chi(\xi) + \sum_{j\geqslant 0} \varphi(2^{-j}\xi) = 1,\\
		\label{eq:dyadic03_FFS} \abs{j-j'} \geqslant 2 \quad \Longrightarrow \quad \mbox{supp\,}\varphi(2^{-j}\cdot) \cap \mbox{supp\,}\varphi(2^{-j'}\cdot) = \varnothing,\\
		\label{eq:dyadic04_FFS} j \geqslant 1 \quad \Longrightarrow \quad \mbox{supp\,}\chi \cap \mbox{supp\,}\varphi(2^{-j}\cdot) = \varnothing.
	\end{gather}
	Moreover, for any $\xi \in \RR^d$, we have
	\begin{equation}
		\label{eq:dyadic05} \frac{1}2 \leqslant \chi^2(\xi) + \sum_{j\geqslant 0} \varphi^2(2^{-j}\xi) \leqslant 1.
	\end{equation}
\end{prop}
The non-homogeneous dyadic blocks are defined as follows
\begin{defn}
	\label{de:dyadic} For any $d\in\NN^*$ and for any tempered distribution $u \in \mathcal{S}'(\RR^d)$, we set
    \begin{align*}
		&\cFFStq u = \mathcal{F}^{-1} \pare{\varphi(2^{-q}\abs{\xi}) \widehat{u}(\xi)}, &&\forall q \geqslant 0,\\
		& \triangle_{-1} u = 
		\mathcal{F}^{-1} \pare{\chi(2^{-1}\abs{\xi}) \widehat{u}(\xi)},\\
		& \triangle_{q} u \equiv 0 , &&\forall q \leqslant -2,\\
		&\cFFSSq u = \sum_{q' \leq q - 1} \Delta_{q'} u, &&\forall q\in \ZZ.
    \end{align*}
\end{defn}

\noindent Using the properties of $\psi$ and $\varphi$, for any tempered distribution $u \in \mathcal{S}'(\RR^d)$, one can formally write
$$u = \sum_{q} \cFFStq u \qquad\mbox{in},$$ and the homogeneous Sobolev spaces ${H}^s(\RR^d)$, with $s\in\RR$, can be characterized as follows
\begin{prop}
    \label{pr:Sobnormiso}
    Let $d\in\NN^*$, $s\in\RR$ and $u\in {H}^s(\RR^d)$. Then,
    \begin{equation*}
        \norm{u}_{{H}^s}  \sim \pare{\sum_{q} 2^{2qs} \norm{\cFFStq u}_{L^2}^2}^{\frac{1}2} = \left\| \left( 2^{qs}\left\| \cFFStq u \right\|_{L^2} \right)_{q\in \ZZ} \right\|_{\ell^2}.
    \end{equation*}
    Moreover, there exists a square-summable sequence of positive numbers $\left( c_q \right)_q$ with $\sum_q c_q^2 = 1$, such that 
	\begin{equation}
		\label{eq:DeltaqHs} \norm{\cFFStq u}_{L^2} \leq c_q(u) 2^{-qs} \norm{u}_{{H}^s}.
	\end{equation}
\end{prop}

\begin{rem}\label{rem:sharper_bound_hom_Sob}
Let us remark that if $ q\geqslant 0 $ the inequality \eqref{eq:DeltaqHs} can be improved to the following sharper bound
\begin{equation}
		\label{eq:DeltaqHshom} \norm{\cFFStq u}_{L^2} \leq c_q(u) 2^{-qs} \norm{u}_{\dot{H}^s}.
	\end{equation}
	We refer to \cite[Chapter 2]{bahouri_chemin_danchin_book} for a thorough study on homogeneous and non-homogeneous dyadic paradifferential calculus. \fine
\end{rem}

\subsection{Paradifferential calculus} \label{paradiffcalcul_FFS}

The decomposition into dyadic blocks allows, at least formally, to write, for any tempered distributions $u$ and $v$,
\begin{align}
	u \ v = & \sum_{\substack{q\in\mathbb{Z} \\ q'\in\mathbb{Z}}}\cFFStq u\, \triangle_{q'} v
\end{align}
The Bony decomposition (see for instance \cite{Bony1981}, \cite{chemin_book} or \cite{bahouri_chemin_danchin_book} for more details) consists in splitting the above sum in three parts. The first corresponds to the low frequencies of $u$ multiplied by the high frequencies of $v$, the second is the symmetric counterpart of the first, and the third part concerns the indices $q$ and $q'$ which are comparable. Then,
\begin{equation*}
	uv = T_u v+ T_v u + R\left(u,v\right),
\end{equation*}
where
\begin{align*}
T_u v=& \sum_q S_{q-1} u \cFFStq v, &
 T_v u= & \sum_{q'} S_{q'-1} v \triangle_{q'} u, &
 R\left( u,v \right) = & \sum_{\abs{q-q'} \leqslant 1} \triangle_q u \triangle_{q'} v.
\end{align*}
Using the quasi-orthogonality given in \eqref{eq:dyadic03_FFS} and \eqref{eq:dyadic04_FFS}, we get the following relations.
\begin{lemma}
	\label{le:orthogonal_FFS}
	For any tempered distributions $u$ and $v$, we have
	\begin{align*}
		&\cFFStq \pare{S_{q'-1} u \triangle_{q'} v} = 0 && \text{if } \left|q-q'\right|\geqslant 5\\
		&\cFFStq \pare{S_{q'+1} u \triangle_{q'} v} = 0 && \text{if } q'\leqslant q-4.
	\end{align*}
\end{lemma}
\noindent Lemma \ref{le:orthogonal_FFS} implies the following decomposition, which we will widely use in this paper
\begin{equation}
	\label{eq:Bonydec} \cFFStq (uv) = \spres \cFFStq \pare{S_{q'-1} v \triangle_{q'} u} + \sloin \cFFStq \pare{S_{q'+2} u \triangle_{q'} v}.
\end{equation}

\textit{Proof of Corollary \ref{cor:FK} :} Let us consider the three dimensional \NS\ equations
\begin{equation}\label{eq:INS}\tag{INS}
\left\lbrace
\begin{aligned}
& \partial_t u + u \cdot \nabla u -\nu \Delta u = - \nabla p,\\
& \dive u =0,\\
& \left. u \right|_{t=0}=u_0.
\end{aligned}
\right.
\end{equation}

Let us consider at first the behavior of \eqref{eq:INS} on the hi-frequencies. To do so let us apply the $ q $--th, for $ q\geqslant 0 $, dyadic block on such equation and lat us multiply the resulting equation for $ \tq u $ and let us integrate in space; we deduce the following differential inequality:
\begin{equation}\label{eq:energy_bound_BU_1}
\frac{1}{2}\ \frac{\d}{\d t} \left\| \tq u \right\|^2_{L^2} + \nu \left\| \tq \nabla u \right\|^2 _{L^2} \leqslant \left| \left(\left. \tq \left( u \otimes u \right) \right| \tq \nabla u  \right)_{L^2} \right|.
\end{equation}
With Bony decomposition we deduce 
\begin{multline*}
\left| \left(\left. \tq \left( u \otimes u \right) \right| \tq \nabla u  \right)_{L^2} \right| \\
\leqslant \sumf \left| \left(\left. S_{q'-1}u \otimes \triangle_{q'} u \right| \tq \nabla u   \right)_{L^2} \right| 
+\sumi \left| \left(\left. S_{q'+2} u \otimes \triangle_{q'} u \right| \tq \nabla u   \right)_{L^2} \right|\\
= I_{1, q} + I_{2, q}.
\end{multline*}
Since $ I_{1, q} $ and  $I_{2, q} $ are symmetric we bound the term $I_{2, q} $ which involves an infinite sum and it is hence more difficult. Applying H\"older inequality  we deduce
\begin{align*}
I_{2, q} = & \ \sumi \left| \left(\left. S_{q'+2} u \otimes \triangle_{q'} u \right| \tq \nabla u   \right)_{L^2} \right|\\
\leqslant & \ \left\| u \right\|_{L^6} \left\| \tq \nabla u \right\|_{L^2} \sumi \left\| \triangle_{q'} u \right\|_{L^3}.
\end{align*}
We use now the embedding $ \dot{H}^1\hra L^6 $ and $ \dot{H}^{1/2} \hra L^3 $ and the property
\begin{equation}\label{eq:prop_dyadic}
\left\| \tq u \right\|_{L^2} \sim 2^{-qs} c_q \left( u \right) \left\| u \right\|_{\dot{H}^s},
\end{equation}
where $ \left( c_q \right)_q\in \ell^2 $. As we pointed out in Remark \ref{rem:sharper_bound_hom_Sob} the estimate \eqref{eq:prop_dyadic} con be applied in this context since we are studying the $ H^s $--energy of the hi-frequencies of \eqref{eq:INS}.  We deduce hence that
\begin{align*}
I_{2, q} \lesssim 2^{-2qs} c_q \left\| u \right\|_{\dot{H}^1} \left\| u \right\|_{{\dot{H}^{s+\frac{1}{2}}}} \left\| u \right\|_{{\dot{H}^{s+1}}} \sumi 2^{\left( q-q' \right) s} c_{q'}.
\end{align*}
We notice that
\begin{equation*}
\left( \sumi 2^{\left( q-q' \right) s} c_{q'} \right)_q =\left( \left(  2^{sp} \ 1_{p<4} \right) \star c_p \right)_q \in \ell^1, 
\end{equation*}
whence the sequence $ \left( b_q \right)_q $ defined as
\begin{equation*}
b_q = c_q \sumi 2^{\left( q-q' \right) s} c_{q'} \in \ell^1.
\end{equation*}
Let us denote the hi-frequencies of $ u $ as
	\begin{equation*}
		u^{\text{H}}=\sum_{q\geqslant 0} \tq u, 
	\end{equation*}
Sobolev interpolation, Young inequality, a multiplication for $ 2^{2qs} $, the use of \eqref{eq:prop_dyadic}, summing on $ q\geqslant 0 $,  a parabolic absorption and Gronwall inequality transform \eqref{eq:energy_bound_BU_1} into
\begin{equation}\label{eq:energy_bound_BU_2}
\left\| u^{\text{H}} \right\|_{\dot{H}^s}^2 + \nu \int_0^t \left\| \nabla u^{\text{H}} \right\|^2_{\dot{H}^s} \d \tau \lesssim \left\| u_0 \right\|^2_{\dot{H}^s} \exp\set{ \frac{C}{\nu^3} \int_0^t \left\| u \right\|^4_{\dot{H}^1} \d \tau }.
\end{equation}
Let us recall that for the low-frequencies
\begin{equation*}
\norm{\triangle_{-1}u}_{H^s}\lesssim_{s} \norm{u}_{L^2}, 
\end{equation*}
whence a standard $ L^2 $ estimate on the \NS\ equations gives us the bound
\begin{equation}\label{eq:energy_bound_BU_3}
\left\| u \right\|_{L^2}^2 + 2\nu \int_0^t \left\| \nabla u \right\|^2_{L^2} \d \tau \leqslant \left\| u_0 \right\|^2_{L^2}.
\end{equation}
Since 
\begin{equation*}
\norm{u}_{H^s}\leqslant \norm{\triangle_{-1}u}_{H^s} + \left\| u^{\text{H}} \right\|_{\dot{H}^s}, 
\end{equation*}
summing \eqref{eq:energy_bound_BU_2} and \eqref{eq:energy_bound_BU_3} and since for $ s>0 $ the non-homogeneous Sobolev space $ H^s $ is continuously embedded in both $ \dot{H}^s$ and $ L^2 $,  we deduce hence the inequality
\begin{equation*}
\left\| u \right\|_{{H}^s}^2 + \nu \int_0^t \left\| \nabla u \right\|^2_{{H}^s} \d \tau \leqslant C \left\| u_0 \right\|^2_{{H}^s}\left[  \exp\set{ \frac{C}{\nu^3} \int_0^t \left\| u \right\|^4_{\dot{H}^1} \d \tau } +1 \right],
\end{equation*}
concluding.\hfill $ \Box $

\section{Spectral analysis of the linear operator} \label{sec:spectral_analysis}

In the context of singular perturbation problems an important role is determined by the dynamics induced by the singular  operator $ \PA $, where $ \PP $ is defined in \eqref{eq:definition_proj_Leray} and $ \mathcal{A} $ in \eqref{matrici}. In particular we are interested to study the effects of the oscillations induced by such operator. This is generally done with tools of Fourier analysis such as dispersive estimates on highly oscillating integrals (\cite{monographrotating}, \cite{Stein93}). To perform such analysis is hence very important to understand the explicit structure of the eigenvalues of the linear operator $ \varepsilon^{-1}\PA -\mathbb{D} $, this is the scope of the present section.\\
We consider the linear operator
\begin{equation}
\label{eq:def_Leps}
L_\varepsilon = \PA -\varepsilon \mathbb{D},
\end{equation}
whose Fourier symbol is
\begin{equation*}
\hat{L}_\varepsilon =
\left( 
\begin{array}{cccc}
\varepsilon \nu \left| \xi \right|^2 & 0 & 0 & -\frac{\xi_3\xi_1}{\left| \xi \right|^2}\\
0 & \varepsilon \nu \left| \xi \right|^2 & 0 & -\frac{\xi_3\xi_2}{\left| \xi \right|^2} \\
0 & 0 & \varepsilon \nu \left| \xi \right|^2 & \frac{\left| \xi_h \right|^2}{\left| \xi \right|^2}\\
0 & 0 & -1 & \varepsilon \nu' \left| \xi \right|^2
\end{array}
 \right).
\end{equation*}
We  study the parabolic operator $ L_\varepsilon $ instead than the hyperbolic $ \PA $ since we want to take in account the regularizing effects induced by the second-order elliptic operator $ -\mathbb{D} $. This choice will become clear in Section \ref{sec:dispersive_properties}.\\
The characteristic polynomial associated to $ \hat{L}_\varepsilon $ is
$$
P_{\hat{L}_\varepsilon} \left( \lambda \right)= \left( \varepsilon \nu \left| \xi \right|^2 -\lambda \right)^2 \left( \lambda^2 - \lambda \varepsilon \left| \xi \right|^2 \left( \nu+ \nu' \right) + \frac{\left| \xi_h \right|^2}{\left| \xi \right|^2} + \varepsilon^2 \nu\nu' \left| \xi \right|^4 \right).
$$
which admits four roots, 
\begin{equation}\label{eq:eigenvalue_0}
\lambda_0^\varepsilon \left( \xi \right)= \varepsilon \nu \left| \xi \right|^2,
\end{equation}
which has multiplicity two and
\begin{equation}\label{eq:eigenvalue_pm}
\lambda_\pm^\varepsilon \left( \xi \right) = \frac{1}{2} \varepsilon \left( \nu+ \nu' \right) \left| \xi \right|^2 \pm i \frac{\left| \xi_h \right|}{\left| \xi \right|}\ S_\varepsilon \left( \xi \right),
\end{equation}
where
$$
S_\varepsilon \left( \xi \right)= \sqrt{1- \varepsilon^2 \ \frac{\left( \nu-\nu' \right) \left| \xi \right|^6}{4 \left| \xi_h \right|^2}}.
$$
Let us restrict ourselves on the localization
\begin{equation}\label{eq:CrR}
\Crr = \set{ \xi \in \R^3_\xi : \left| \xi_h \right|>r, \ \left| \xi \right|<R },
\end{equation}
we choose such localization since in $ \Crr $ the eigenvalues  $ \lambda_0^\varepsilon, \lambda_\pm^\varepsilon $ are well-defined. Moreover $ \left| \lambda_\pm^\varepsilon \left( \xi \right) \right| \geqslant\frac{r}{2R} $ (if $ \varepsilon $ is sufficiently small) for any $ \xi\in \Crr $, hence the oscillating eigenvalues are never null in such set. 
It is clear that, for $ \varepsilon $ sufficiently small, on $ \Crr $
\begin{equation*}
\left| S_\varepsilon \left( \xi \right) - 1 \right|\leqslant C_{r,R} \ \varepsilon,
\end{equation*}
hence from now on we shall consider implicitly $ S_\varepsilon \approx 1 $.\\

Let us evaluate the eigenvectors related to the eigenvalues $ \lambda_i^\varepsilon $, relatively to the eigenvalue $ \lambda_0^\varepsilon $, which has multiplicity two, we have two eigenvectors
\begin{equation*}
\begin{aligned}
e_1 = & \ \left( \begin{array}{cccc}
1&0&0&0
\end{array} \right)^\intercal,
&
e_2 = & \ \left( \begin{array}{cccc}
0&1&0&0
\end{array} \right)^\intercal.
\end{aligned}
\end{equation*}
These eigenvectors are not divergence-free, hence, a priori, they do not describe the evolution of solutions of equation \eqref{perturbed BSSQ}. In any case there is a subspace of the space $ \mathbb{C}e_1\oplus \mathbb{C} e_2 $ which is composed by divergence-free vector fields, namely the space spanned by the vector
\begin{equation}
\label{eq:eigemvectors_0}
E_0 \left( \xi \right)= \frac{1}{\left| \xi_h \right|} \left( \begin{array}{cccc}
-\xi_2 & \xi_1 & 0 & 0
\end{array} \right)^\intercal.
\end{equation}
Relatively to the eigenvalues $ \lambda_\pm^\varepsilon $ the following eigenvectors can be computed
\begin{equation}
\label{eq:eigemvectors_pm}
 E_\pm^\varepsilon \left( \xi \right) =
\left( 
\begin{array}{c}
\pm i \frac{\xi_3 \xi_1}{\left| \xi \right| \left| \xi_h \right|} \ \mathcal{S}_\varepsilon^\pm \left( \xi \right) \\
\pm i \frac{\xi_3 \xi_2}{\left| \xi \right| \left| \xi_h \right|} \ \mathcal{S}_\varepsilon^\pm \left( \xi \right)\\
\mp i \frac{\left| \xi_h \right|}{\left| \xi \right|}  \ \mathcal{S}_\varepsilon^\pm \left( \xi \right)\\
1
\end{array}
 \right),
\end{equation}
where
\begin{equation*}
\mathcal{S}_\varepsilon^\pm \left( \xi \right) = S_\varepsilon \left( \xi \right) \pm \frac{i}{2}\ \varepsilon \ \frac{\left( \nu-\nu' \right)\left| \xi \right|^3}{\left| \xi_h \right|},
\end{equation*}
hence if $ \xi \in \Crr $ and $ \varepsilon $ small $ \mathcal{S}_\varepsilon^\pm \approx 1 $.\\

An important feature of the spectral analysis of the operator $ L_\varepsilon $ is that the eigenvectors \textit{are not orthogonal}. We will in Section \ref{sec:dispersive_properties} require to analyze the regularity of the propagation of some vector field along the eigendirections spanned by $ E_\pm^\varepsilon $. This cannot hence be done by a standard application of the triangular inequality since, as we will see below, the projections onto the eigenspaces are defined by suitable Fourier multiplier, hence a more thorough analysis is required.\\
Let us now consider a solenoidal vector field $ V= \left( V^1, V^2, V^3, V^4 \right) $ which belongs to the space 
\begin{equation*}
\mathcal{X}= \bigoplus_{i=0,\pm} \mathbb{C}\ E_i^\varepsilon,
\end{equation*}
indeed
\begin{equation*}
V= \mathcal{F}^{-1 } \left( \sum_{i=0,\pm} k_{i,\varepsilon} \left( V \right) \ E_i^\varepsilon \right),
\end{equation*}
where the elements $ k_{i,\varepsilon}, \ i=0, \pm $ are suitable forms which act on the space of solenoidal vector fields and they describe the magnitude of the projection of $ V $ onto the eigenspace $ \CC \ E_i^\varepsilon $. 
We can hence define the projections of a divergence-free vector field in $ \mathcal{X} $ onto the eigenspace spanned by $ E_i^\varepsilon $ as
\begin{align}
\label{eq:projectors}
\mathbb{P}_{i, \varepsilon} \left( V \right) = \mathcal{F}^{-1} \left( k_{i,\varepsilon} \left( V \right) \ E_i^\varepsilon \right), && i=0,\pm .
\end{align}
The definition of $ \PP_{i, \varepsilon} $ does not give any insight of the regularity of the element $ \PP_{i, \varepsilon}  \left( V \right) $ w.r.t. the regularity of $ V $. We expect that the form $ k_i $ acts as a Fourier multiplier of a suitable degree. We prove in fact that, as long as we restrict ourselves in the set $ \Crr $, the map $ \hat{V}\mapsto k_{i,\varepsilon} \left( V \right)  E_i^\varepsilon $ acts as a multiplication for a constant in terms of $ L^2 $ regularity:

\begin{lemma}\label{lem:boundedness_projectors}
Let $ V\in \mathcal{X} $ a solenoidal vector field such that $ \textnormal{supp}  \left( \hat{V} \right)\subset \Crr $, then for $ i=0,\pm $
\begin{equation*}
\left\| \mathbb{P}_{i, \varepsilon} \left( V \right) \right\|_{\cFFSLtwo} \leqslant C_{r,R} \left\| V \right\|_{\cFFSLtwo}.
\end{equation*}
\end{lemma}

\begin{proof}
This is a problem of linear algebra. Let us consider the following basis of $ \mathbb{C}^4 $ (in the Fourier space)
$$
\mathcal{B}= \set{ e_1, E_0, E_+^\varepsilon, E_-^\varepsilon },
$$
and the canonical basis
$$
\B_{\text{can}}= \set{ e_j }_{j=1}^4.
$$
The matrix $ \hat{L}_\varepsilon $ is indeed diagonalizable, hence there exists an invertible matrix $ Q $ such that
$$
Q \hat{L}_\varepsilon \left( \xi \right) Q^{-1} = \text{diag} \set{ \lambda_0^\varepsilon \left( \xi \right),\lambda_0^\varepsilon \left( \xi \right), \lambda_+^\varepsilon \left( \xi \right), \lambda_-^\varepsilon \left( \xi \right) },
$$
the matrix $ Q $ is the change of base matrix from the base $ \B_{\text{can}} $ to the base $ \B $ and, given the explicit expression of the eigenvectors in \eqref{eq:eigemvectors_0}, \eqref{eq:eigemvectors_pm} it assumes the form
\begin{equation*}
Q = 
\left( 
\begin{array}{cccc}
1    &  \frac{-\xi_2}{\left| \xi_h \right|} 
&  i\frac{\xi_3 \xi_1}{\left| \xi \right| \left| \xi_h \right|} \ \mathcal{S}_\varepsilon^\pm \left( \xi \right)
&
- i\frac{\xi_3 \xi_1}{\left| \xi \right| \left| \xi_h \right|} \ \mathcal{S}_\varepsilon^\pm \left( \xi \right)\\
0    &  \frac{\xi_1}{\left| \xi_h \right|}
& i \frac{\xi_3 \xi_2}{\left| \xi \right| \left| \xi_h \right|} \ \mathcal{S}_\varepsilon^\pm \left( \xi \right)
&
-i \frac{\xi_3 \xi_2}{\left| \xi \right| \left| \xi_h \right|} \ \mathcal{S}_\varepsilon^\pm \left( \xi \right)\\
0     &  0  &
-i \frac{\left| \xi_h \right|}{\left| \xi \right|}  \ \mathcal{S}_\varepsilon^\pm \left( \xi \right)
&
i \frac{\left| \xi_h \right|}{\left| \xi \right|}  \ \mathcal{S}_\varepsilon^\pm \left( \xi \right)\\
0     &  0  & 1 &1
\end{array}
 \right).
\end{equation*}
Let us note that the first column of $ Q $ is $ \left( 1, 0, 0, 0 \right)^\intercal $, this is motivated by the fact that we completed the basis $ \B $ with the vector $ e_1 $ in order to obatin a complete basis of $ \CC^4 $.\\
The matrix $ Q $ performs the following transformation,
$$
Q
\left( 
\begin{array}{c}
0\\
k_{0, \varepsilon} \\ k_{+, \varepsilon} \\ k_{-, \varepsilon} 
\end{array}
 \right)
 =
 \left( 
\begin{array}{c}
\hat{V}^1 \\ \hat{V}^2 \\ \hat{V}^3 \\ \hat{V}^4
\end{array} 
  \right),
$$
we deduce hence that the element 
\begin{equation}
\label{eq:equation_k}
\left( 
\begin{array}{c}
0\\
k_{0, \varepsilon} \\ k_{+, \varepsilon} \\ k_{-, \varepsilon} 
\end{array}
 \right)
 =
 Q^{-1} 
 \left( 
\begin{array}{c}
\hat{V}^1 \\ \hat{V}^2 \\ \hat{V}^3 \\ \hat{V}^4
\end{array} 
  \right),
\end{equation}
gives the expression of the $ k_i $'s in terms of the variables $ \hat{V}_i $'s multiplied by suitable Fourier multipliers determined by the inverse matrix $ Q^{-1} $. 
Whence it suffice to compute the explicit expression of the matrix $ Q^{-1} $ to solve the linear system above. The matrix $ Q^{-1} $ assumes the form
\begin{equation}
\label{eq:definition_Q-1}
 Q^{-1}=
 \left( 
\begin{array}{cccc}
1 & \frac{\xi_2}{\xi_1} & \frac{\xi_3}{\xi_1} & 0\\
 0 & \frac{\left| \xi_h \right|}{\xi_1} & \frac{\xi_2 \xi_3}{\left| \xi_h \right| \xi_1} & 0\\
 0 & 0 & -i \frac{\left| \xi \right|}{2 \mathcal{S}_\varepsilon^\pm \left( \xi  \right)} & \frac{1}{2}\\
 0 & 0 & +i \frac{\left| \xi \right|}{2 \mathcal{S}_\varepsilon^\pm \left( \xi  \right)} & \frac{1}{2}
\end{array} 
  \right) ,
\end{equation}
whence it is clear that, since $ \hat{V} $ is supported in $ \Crr $:
\begin{align*}
\left| \left( 
\begin{array}{c}
0\\
k_{0, \varepsilon} \\ k_{+, \varepsilon} \\ k_{-, \varepsilon} 
\end{array}
 \right) \right|
 \leqslant & \left| Q^{-1} \right| \left| \hat{V} \right|,\\
 \leqslant & \ C_{r,R}\left| \hat{V} \right|.
\end{align*}
The claim follows applying Plancherel theorem.
\end{proof}

Lemma \ref{lem:boundedness_projectors} gives hence a complete answer regarding the regularity of the projectors $ \PP_{i, \varepsilon} $, nonetheless we did not compute explicitly their form. Regarding the first two equations of the system \eqref{eq:equation_k} we can deduce the following explicit equations thanks to the explicit expression of $ Q^{-1} $ given in \eqref{eq:definition_Q-1}:
\begin{align*}
0 = & \ \xi_1 \hat{V}^1 + \xi_2 \hat{V}^2 + \xi_3 \hat{V}^3,\\
k_0 = & \ \frac{1}{\left| \xi_h \right|} \left( -\xi_2 \hat{V}^1 + \xi_1 \hat{V}^2 \right),
\end{align*}
hence we can compute explicit expression of the projector $ \PP_0 V = \mathcal{F}^{-1} \left( k_0 \left( V \right) E_0 \right) $, which in particular assumes the form (in the Fourier variables):
\begin{align*}
\mathcal{F} \left( \PP_{0, \varepsilon} V \right) = & \ k_{0, \varepsilon} \left( V \right) E_0 ,\\
= & \ \frac{1}{\left| \xi_h \right|^2}
\left( 
\begin{array}{c}
-\xi_2 \\ + \xi_1 \\ 0 \\ 0
\end{array}
 \right)
 \left( -\xi_2 \hat{V}^1 + \xi_1 \hat{V}^2 \right).
\end{align*}
Whence we can define the projector $ \PP_0 $ (which does not depend any more on the parameter $ \varepsilon $) which maps a solenoidal vector field $ V $ onto $ \CC E_0 $ via the following pseudo-differential operator of order zero
\begin{equation}
\label{eq:definition_P0}
\PP_0 V = \left( 
\begin{array}{c}
-\left( -\Dh \right)^{-1}\partial_2 \ \text{curl}_h V\\
+\left( -\Dh \right)^{-1}\partial_1 \ \text{curl}_h V\\
0\\
0
\end{array}
 \right)
=
\left( -\Dh \right)^{-1}
\left( 
\begin{array}{cccc}
\partial_2^2 & -\partial_1\partial_2 & 0 & 0\\
-\partial_1\partial_2 & \partial_1^2& 0 & 0 \\
 0 & 0 & 0 & 0\\
  0 & 0 & 0 & 0
\end{array}
 \right) V
 ,
\end{equation}
where the operator $ \curlh $ is defined as $ \curlh  V = -\partial_2 V^1 + \partial_1 V^2 $.

The space $ \CC E_0 $ shall be denoted as \textit{non-oscillating} subspace, whereas the space $ \CC E_+^\varepsilon \oplus \CC E_-^\varepsilon $ shall be denoted as \textit{oscillating} subspace. This choice of lexicon can easily be justified: let us consider the following linear system,
\begin{equation*}
\left\lbrace
\begin{aligned}
& \partial_t W_{\text{L}} + \frac{1}{\varepsilon} L_\varepsilon \  W_{\text{L}} =0,\\
& \left.  W_{\text{L}} \right|_{t=0}=  W_{\text{L}, 0}.
\end{aligned}
\right.
\end{equation*}
The unique solution of such system can be written as 
\begin{equation*}
 W_{\text{L}}\left( t \right)= e^{\frac{t}{\varepsilon}L_\varepsilon}  W_{\text{L}, 0}.
\end{equation*}
Respectively hence the projection of $  W_{\text{L}} $ onto the subspaces $ \CC E_0, \CC E_\pm^\varepsilon $ is 
\begin{align*}
\mathbb{P}_0  W_{\text{L}} \left( t \right) = & \ \mathcal{F}^{-1} \left( e^{-\nu t \ \left| \xi \right|^2}\widehat{\PP_0 W_{\text{L}, 0}} \left( \xi \right) \right),\\
\PP_{\pm, \varepsilon}  W_{\text{L}} \left( t \right) = & \ \mathcal{F}^{-1} \left( e^{-\ \frac{t}{\varepsilon}\lambda_\varepsilon^\pm \left( \xi \right)}\widehat{\PP_{\pm, \varepsilon} W_{\text{L}, 0}} \left( \xi \right) \right).
\end{align*}
We can immediately see hence that the elements $ \mathbb{P}_0  W_{\text{L}} $ and $ \PP_{\pm, \varepsilon}  W_{\text{L}} $ have two qualitatively very different behaviors: the former has a purely parabolic decay-in-time, while the latter is described by an oscillating integral.

	\section{Global well posedness of the limit system}\label{sec:gwp_limit}
	
	A consistent part of Theorem \ref{thm:main_result} deals with the convergence of solutions of \eqref{perturbed BSSQ} in the regime $ \varepsilon \to 0 $ to a certain limit function. \\
	We expect hence that once we restrict ourselves onto $ \CC E^0 $, no dispersive effect occur due to the absence of the singular perturbation, determining hence a candidate for the limit model we look for.\\

	\subsection{Formal derivation of the limit system}\label{sec:formal_derivation_limit_system}
	
An important step as long as concerns singular perturbation problems is to deduce formally a limit system to whom \eqref{perturbed BSSQ} converges. Several works on geophysical fluids such as \cite{monographrotating}, \cite{charve1} or \cite{FGN} suggest that the solutions of \eqref{perturbed BSSQ} converge (in a sense which we do not specify at the moment) to an element belonging to the nonoscillatory space $ \CC E^0 $. \\

The next result is a direct deduction of Theorem \ref{thm:Leray_applied} (see for instance \cite[Corollary 2.1]{GallagherSaint-Raymondinhomogeneousrotating}):
\begin{lemma}\label{lem:weak_conv_var_to_0}
Let $ U_0 $ be in $ \cFFSLtwo $, and let $ U^\varepsilon $ be a weak solution of \eqref{perturbed BSSQ}, there exists a $ U^\star \in L^\infty \left( \R_+; \cFFSLtwo \right) \cap L^2 \left( \R_+; \dot{H}^1 \left( \R^3 \right) \right) $ and a subsequence $ \varepsilon_j \xrightarrow{j\to \infty} 0 $ such that
\begin{equation*}
U^{\varepsilon_j} \rhu U^\star \ \text{weakly in } \ L^2_{\loc} \left( \R_+; L^2_{\loc} \left( \R^3 \right) \right) \ \text{ as } \ j\to \infty.
\end{equation*}
\end{lemma}

Taking a formal limit for $ \varepsilon \to 0 $ in \eqref{perturbed BSSQ} and supposing that $ \left( U^\varepsilon , \Phi^\varepsilon \right) \to \left( U^\star, \Phi^\star \right) $ the following balance
\begin{equation}\label{eq:qg_modified_balance}
\begin{aligned}
u^{3, \star} = & \ 0,\\
\rho^\star = & \ \partial_3 \Phi^\star,
\end{aligned}
\end{equation}
has to take place by simple comparison of magnitude in \eqref{perturbed BSSQ} in the limit $ \varepsilon \to 0 $.\\
Le us consider now the subsequence $ \left( \varepsilon_j \right)_j $ identified in Lemma \ref{lem:weak_conv_var_to_0}.
With a standard argument of cancellation of the pressure on \eqref{perturbed BSSQ} we can deduce that
\begin{equation*}
-\Delta \Phi^{\varepsilon_j} = -\partial_3 \rho^{\varepsilon_j} + \varepsilon_j \dive \dive \left( u^{\varepsilon_j} \otimes u^{\varepsilon_j} \right).
\end{equation*}
Since 
\begin{align*}
\left\| \dive \dive \left( u^\varepsilon \otimes u^\varepsilon \right) \right\|_{L^2_{\loc}\left(\R_+; H^{-3} \right)}\leqslant \left\| u^\varepsilon \right\|_{L^\infty_{\loc} \left( \R_+; L^2 \right)}
\left\| u^\varepsilon \right\|_{L^2_{\loc} \left( \R_+; H^1 \right)}
<\infty
, \hspace{5mm}\forall \ \varepsilon >0,
\end{align*}
we  deduce that $ \varepsilon \dive \dive \left( u^{\varepsilon} \otimes u^{\varepsilon} \right) $ is an $ \mathcal{O} \left( \varepsilon \right) $ function in the $ L^2_{\loc}\left(\R_+; H^{-3} \right) $ topology, hence since $ \rho^{\varepsilon_j} \rhu \rho^\star $ in $ L^2_{\loc}\left(\R_+; L^2 \right) $ \textit{for the same subsequence} $ \left( \varepsilon_j \right)_j $ we deduce
\begin{equation*}
-\Delta \Phi^{\varepsilon_j} \to -\Delta \Phi^\star = -\partial_3 \rho^\star,
\end{equation*}
in the sense of distributions. The above relation together with \eqref{eq:qg_modified_balance} imply that
\begin{equation*}
-\Delta \rho^\star = -\partial_3^2 \rho^\star \ \Rightarrow \ -\Dh \rho^\star=0.
\end{equation*}
But $ -\Dh \rho ^\star=0 $ in the whole space implies that $ \rho^\star = \rho^\star \left( x_3 \right) $, and hence  Lemma \ref{lem:weak_conv_var_to_0} allows us to state that $ \rho^\star \equiv 0 $ in $ L^2 $.\\
We hence deduced (formally) until now that
\begin{equation*}
\left( u^{h ,\varepsilon}, u^{3, \varepsilon}, \rho^\varepsilon, \Phi^\varepsilon \right) \rhu \left( u^{h, \star}, 0, 0 , \Phi^\star \right),
\end{equation*}
we want to understand (heuristically) which equation is satisfied by the limit function $  u^{h, \star} $.

Next let us consider some very specific test functions $ \phi \in \mathcal{D}\left( \R_+\times \R^3 \right) $ such that $ \phi = \left( \phi_1, \phi_2, 0, 0 \right) $ and 
\begin{align*}
\phi_1 = -\partial_2 \Dh^{-1} \Psi , && 
\phi_2 = -\partial_1 \Dh^{-1} \Psi ,
\end{align*}
for some potential $ \Psi $. This in particular implies that $ \diveh \phi_h =0 $, these hypothesis have been imposed so that
\begin{equation*}
\hat{\phi} \left( t \right)\in \CC E_0, \hspace{1cm} \forall \ t >0.
\end{equation*}
Let us \textit{suppose} moreover that the weak convergence sated in Lemma \ref{lem:weak_conv_var_to_0} is strong enough so that 
\begin{equation}\label{eq:hyp_weak_conv}
u^{ \varepsilon_j}\otimes u^{\varepsilon_j} \rhu u^\star \otimes u^\star.
\end{equation}
Obviously this is not the case, but an educated guess which motivated the  development of the present work.\\

Testing  the equation \eqref{perturbed BSSQ} against functions of such form we deduce that (here we denote as $ u^{h, \varepsilon} $ the horizontal components of $ U^\varepsilon $)
\begin{equation*}
\psca{u^{h, \varepsilon_j}}{\partial_t \phi_h}
-\psca{u^h_0}{\psi \left( 0 \right)}
 + \psca{ u^{h, \varepsilon_j}\otimes u^{h, \varepsilon_j}}{\nh \phi_h} + \psca{u^{3, \varepsilon_j} u^{h, \varepsilon_j}}{\partial_3 \phi_h} + \psca{u^{h, \varepsilon_j}}{\Delta \phi_h}=0,
\end{equation*}
Let us take now formally the limit as $ \varepsilon_j \to 0 $, justified by Lemma \ref{lem:weak_conv_var_to_0}. First of all we remark, thanks to the balance deduced in \eqref{eq:qg_modified_balance}, and the hypothesis \eqref{eq:hyp_weak_conv}:
\begin{equation*}
u^{3, {\varepsilon_j}}\rhu 0 \ \Rightarrow \ \psca{u^{3, {\varepsilon_j}} u^{h, {\varepsilon_j}}}{\partial_3 \phi_h} \to 0 \ \text{as} \ \varepsilon \to 0.
\end{equation*}
Whence we deduce that, at least in this restricted distributional sense, the limit function describing the evolution of the horizontal components shall satisfy the system
\begin{equation*}
\left\lbrace
\begin{aligned}
& \partial_t u^{h, \star} + u^{h, \star} \cdot \nh u^{h, \star} - \nu \Delta u^{h, \star} = -\nh \Phi^\star,\\
& \diveh u^{h, \star} =0.
\end{aligned}
\right.
\end{equation*}

	\subsection{Detailed study of the limit system}

	Section \ref{sec:formal_derivation_limit_system} motivates hence the study of
	 the 2-dimensional, incompressible, stratified \NS\ system 
	\begin{equation}\label{eq:2DNS_stratified}
\left\lbrace
\begin{aligned}
& \partial_t \uh  \left( x_h, x_3 \right) + \uh \left( x_h, x_3 \right) \cdot \nh \uh \left( x_h, x_3 \right) - \nu \Delta \uh \left( x_h, x_3 \right) = -\nh \bar{p} \left( x_h, x_3 \right)\\
& \diveh \uh  \left( x_h, x_3 \right) =0,\\
& \left. \uh \left( x_h, x_3 \right) \right|_{t=0}=\mathbb{P}_0 \ U_0 \left( x_h, x_3 \right)= \uh_0 \left( x_h, x_3 \right).
\end{aligned}
\right.
\end{equation}
The operator $ \PP_0 $ is defined in \eqref{eq:definition_P0}. 
The  velocity field $ \uh $ is endowed with a 2d-like vorticity 
\begin{equation*}
\oh \left( x_h, x_3 \right) = -\partial_2 \bar{u}^{h,1} \left( x_h, x_3 \right) + \partial_1 \bar{u}^{h,2} \left( x_h, x_3 \right),
\end{equation*}
which, as well as for the two-dimensional \NS\ equation satisfies the transport-diffusion equation
\begin{equation}\label{eq:2DNS_vorticity_stratified}
\left\lbrace
\begin{aligned}
& \partial_t \oh \left( x_h, x_3 \right) + \uh \left( x_h, x_3 \right) \cdot \nh \oh  \left( x_h, x_3 \right)  - \nu \Delta \oh \left( x_h, x_3 \right) = 0\\
& \left. \oh  \left( x_h, x_3 \right) \right|_{t=0}=\oh_0 \left( x_h, x_3 \right).
\end{aligned}
\right.
\end{equation}
We can recover $ \uh $ from $ \oh $ via a 2D-like Biot-Savart law
\begin{equation*}
\uh \left( x_h, x_3 \right) =
\left( 
\begin{array}{c}
-\partial_2\\
\partial_1
\end{array}
 \right)
 \Dh^{-1} \oh  \left( x_h, x_3 \right),
\end{equation*}
as it was already outlined and justified in the previous section deducing the explicit expression of the projector  $ \PP_0 $ in \eqref{eq:definition_P0}.\\

Let us make a couple of remarks on the system \eqref{eq:2DNS_stratified}, the unknown $ \uh $ of \eqref{eq:2DNS_stratified} depends on all three space variables and is time-dependent, i.e. $ \uh = \uh \left( t, x \right) = \uh \left( t, x_h, x_3 \right) $. The equations \eqref{eq:2DNS_stratified} represents hence a \NS\ system in the horizontal directions $ x_h $, while it is a diffusive equation along the vertical direction $ x_3 $.\\

The results stated in the following lemmas are classical, hence the proof is omitted.
\begin{lemma} \label{lem:Ler_sol_ubar_omega}
Let $ \uh_0\in {\cFFSLtwo} $ and $  \oh_0 \in {\cFFSLtwo} $. Then there exists respectively a weak solution $ \uh, \oh $ of \eqref{eq:2DNS_stratified} and \eqref{eq:2DNS_vorticity_stratified} such that
$$
\uh, \oh \in L^\infty \left( \R_+; {\cFFSLtwo} \right)\cap L^2 \left( \R_+; \dot{H}^1 \left( \R^3 \right) \right),
$$
and, for each $ t>0 $, the following bounds hold
\begin{align*}
  \left\| \uh \left( t \right) \right\|_{\cFFSLtwo}^2 + 2\nu \int_0^t \left\| \nabla \uh \left( \tau \right)\right\|_{\cFFSLtwo}^2 \d \tau \leqslant \left\| \uh_0 \right\|_{\cFFSLtwo}^2,\\
   \left\| \oh \left( t \right) \right\|_{\cFFSLtwo}^2 + 2\nu \int_0^t \left\| \nabla \oh \left( \tau \right)\right\|_{\cFFSLtwo}^2 \d \tau \leqslant \left\| \oh_0 \right\|_{\cFFSLtwo}^2.
\end{align*}
\end{lemma}
 \begin{lemma}\label{lem:interpolation_L2_H1}
 Let $ U= U \left( x \right)$ be in $ L^2 \left( \R \right)\cap \dot{H}^1 \left( \R \right) $, then $ U\in L^\infty \left( \R \right) $ and 
 $$
 \left\| U \right\|_{L^\infty \left( \R \right)}\leqslant C \left\| U \right\|^{1/2}_{L^2 \left( \R \right)} \left\|  U' \right\|^{1/2}_{L^2 \left( \R \right)}.
 $$
 \end{lemma}
 
 This is all we require in order to prove the following lemma, which is the main result which will allows us subsequently to prove that \eqref{eq:2DNS_stratified} is globally well posed in $ H^s \pare{\R^3}, \ s\geqslant1/2 $:
 
 \begin{lemma}\label{lem:bound_anisotropic_norm}
 Let $ \uh_0, \oh_0 $ satisfy the hypotheses of Lemma \ref{lem:Ler_sol_ubar_omega}, then
 \begin{align*}
  \uh \in L^4 \left( \R_+; L^4 \left( \R^2_h ; L^\infty \left( \R_v \right) \right) \right)= L^4 \left( \R_+; L^2_h \left( L^\infty_v \right) \right),\\
 \uh \in L^4 \left( \R_+; L^\infty \left( \R_v ; L^4 \left( \R^2_h \right) \right) \right)= L^4 \left( \R_+; L^\infty_v \left( L^4_h \right) \right),
 \end{align*}
 and for each $ t>0 $ the following bounds hold
 \begin{equation*}
 \int_0^t \left\| \uh \left( \tau \right) \right\|^4_{ L^\infty_v \left( L^4_h \right)} \d \tau  \leqslant
 \int_0^t \left\| \uh \left( \tau \right) \right\|^4_{ L^4_h \left( L^\infty_v \right)} \d \tau 
  \leqslant \frac{C K^2}{\nu}\left( \left\| \uh_0 \right\|_{\cFFSLtwo}^4 +\left\| \oh_0 \right\|_{\cFFSLtwo}^4 \right).
  \end{equation*}
 \end{lemma}
 
 \begin{proof}
 Let us start considering the value $ \left\| \uh \right\|_{ L^4_h \left( L^\infty_v \right)}^4 $, applying Lemma \ref{lem:interpolation_L2_H1} we deduce
 \begin{equation*}
 \left\| \uh \right\|_{ L^4_h \left( L^\infty_v \right)}^4 \leqslant C \left\| \uh \right\|_{ L^4_h \left( L^2 _v \right)}^2 \left\| \partial_3\uh \right\|_{L^4_h \left( L^2 _v \right)}^2.
 \end{equation*}
 By use of \eqref{eq:LpLqanisotropic ineq} and a Gagliardo-Nirenberg interpolation inequality we deduce
 \begin{align*}
 \left\| \uh \right\|_{L^4_h \left( L^2 _v \right)}^2
 \leqslant & 
 \left\| \uh \right\|_{L^2_v \left( L^4 _h \right)}^2\\
  \leqslant & \ C  \left\| \uh \right\|_{\cFFSLtwo} \left\| \nh \uh \right\|_{\cFFSLtwo},\\
  \left\| \partial_3 \uh \right\|_{L^4_h \left( L^2 _v \right)}^2
\leqslant & 
 \left\| \partial_3 \uh \right\|_{L^2_v \left( L^4 _h \right)}^2\\  
   \leqslant & \ C  \left\| \partial_3 \uh \right\|_{\cFFSLtwo} \left\|  \partial_3\nh \uh \right\|_{\cFFSLtwo},
 \end{align*}
 whence we deduce 
 \begin{equation}\label{eq:bound_anisotropic_norm1}
 \begin{aligned}
 \left\| \uh \right\|_{L^4_h \left( L^\infty_v \right)}^4 \leqslant & \ C \left\| \uh \right\|_{\cFFSLtwo} \left\| \nh \uh \right\|_{\cFFSLtwo}  \left\| \partial_3 \uh \right\|_{\cFFSLtwo} \left\|  \partial_3\nh \uh \right\|_{\cFFSLtwo}, \\
 \leqslant & \ C K^2 \left\| \uh \right\|_{\cFFSLtwo} \left\| \oh \right\|_{\cFFSLtwo} \left\| \partial_3 \uh \right\|_{\cFFSLtwo} \left\|  \partial_3\oh \right\|_{\cFFSLtwo},\\
 \leqslant & \ C K^2 \left\| \uh \right\|_{\cFFSLtwo} \left\| \oh \right\|_{\cFFSLtwo} \left\| \nabla \uh \right\|_{\cFFSLtwo} \left\|  \nabla \oh \right\|_{\cFFSLtwo}
  \end{aligned}
 \end{equation}
 where in the second inequality we used the fact that the map $ \oh \mapsto \nh \uh $ is a Calderon-Zygmund application of norm $ K $. Integrating in time \eqref{eq:bound_anisotropic_norm1} using a Cauchy-Schwarz inequality and the results of Lemma \ref{lem:Ler_sol_ubar_omega} we deduce the inequality
 \begin{equation*}
 \int_0^t \left\| \uh \left( \tau \right) \right\|^4_{ L^4_h \left( L^\infty_v \right)} \d \tau 
  \leqslant \frac{C K^2}{\nu}\left( \left\| \uh_0 \right\|_{\cFFSLtwo}^4 +\left\| \oh_0 \right\|_{\cFFSLtwo}^4 \right).
 \end{equation*}
 To complete the proof it suffice hence to apply \eqref{eq:LpLqanisotropic ineq}.
 \end{proof}
 
 Lemma \ref{lem:bound_anisotropic_norm} is the cornerstone of the proof of the propagation of the isotropic Sobolev regularity, which is formalized in the following proposition
 
 \begin{prop} \label{prop:strong_reg_ubar}
 Let $ \uh_0 \in H^s \left( \R^3 \right), \ s>0 $ and $ \oh_0 \in {\cFFSLtwo} $, then the weak solution $ \uh  $ of \eqref{eq:2DNS_stratified} which exists thanks to Lemma \ref{lem:Ler_sol_ubar_omega} belongs to the space 
 $$
 \uh \in L^\infty\left( \R_+;  H^s \left( \R^3 \right) \right) , \hspace{1cm} 
 \nabla \uh \in L^2 \left( \R_+;  H^s \left( \R^3 \right) \right),
 $$
 and for each $ t>0 $ the following bound holds
 \begin{multline}\label{eq:Hs_bound_uh}
 \left\| \uh \left( t \right) \right\|_{ H^s \left( \R^3 \right)}^2 + \nu \int_0^t \left\| \nabla \uh \left( \tau \right) \right\|_{ H^s \left( \R^3 \right)}^2 \d \tau
 \\
  \leqslant C \left\| \uh_0 \right\|_{ H^s \left( \R^3 \right)}^2 \exp \set{ \frac{C K^2}{\nu}\left( \left\| \uh_0 \right\|_{\cFFSLtwo}^4 +\left\| \oh_0 \right\|_{\cFFSLtwo}^4 \right)}.
 \end{multline}
 \end{prop}
 
 \begin{proof}
  Let us apply the operator $ \cFFStq $ to the equation \eqref{eq:2DNS_stratified} and multiply it for $ \cFFStq \uh $ and integrate in space, we deduce
  \begin{equation}\label{eq:energy_est1}
   \frac{1}{2}\ \frac{\d}{\d t} \left\| \cFFStq \uh \right\|_{{\cFFSLtwo}}^2 + \nu \left\| \cFFStq \nabla \uh \right\|_{{\cFFSLtwo}}^2 \leqslant \left| \left(\left. \cFFStq \left( \uh \cdot\nh \uh \right) \right|   \cFFStq \uh \right)_{\cFFSLtwo} \right|.
  \end{equation}
  Indeed, since $ \diveh \uh =0 $ and integrating by parts,
  \begin{equation*}
   \left| \left(\left. \cFFStq \left( \uh \cdot\nh \uh \right) \right|   \cFFStq \uh \right)_{\cFFSLtwo} \right|=  \left| \left(\left. \cFFStq \left( \uh \otimes \uh \right) \right|   \cFFStq \nh \uh \right)_{\cFFSLtwo} \right|.
  \end{equation*}
  Applying Bony decomposition we deduce
  \begin{multline*}
   \left| \left(\left. \cFFStq \left( \uh \otimes \uh \right) \right|   \cFFStq \nh \uh \right)_{\cFFSLtwo} \right| \\
   \leqslant \sumf 
    \left| \left(\left. \cFFStq \left( S_{q'-1} \uh \otimes \triangle_{q'}\uh \right) \right|   \cFFStq \nh \uh \right)_{\cFFSLtwo} \right| \\
     +
    \sumi  \left| \left(\left. \cFFStq \left( \triangle_{q'}\uh \otimes  S_{q'+2}\uh \right) \right|   \cFFStq \nh \uh \right)_{\cFFSLtwo} \right|=
    I_{1,q} + I_{2,q}.
  \end{multline*}
  Since the operators $ \cFFStq, \cFFSSq $ map continuously any $ {\cFFSLp} $ space to itself and by H\"older inequality  we deduce
  \begin{equation}
  \label{eq:bound_Iq1}
  \begin{aligned}
  I_{1,q} \leqslant & \ C \left\| \uh \right\|_{L^\infty_v \left( L^4_h \right)} \sumf  \left\| \triangle_{q'}\uh \right\|_{L^2_v \left( L^4_h \right)} \left\| \cFFStq \nh \uh \right\|_{\cFFSLtwo},\\
  \leqslant & \ C \left\| \uh \right\|_{L^\infty_v \left( L^4_h \right)} \sumf  \left\| \triangle_{q'}\uh \right\|_{\cFFSLtwo}^{1/2} \left\| \cFFStq \nh \uh \right\|_{\cFFSLtwo}^{3/2},\\
  \leqslant & \ C b_q 2^{-2qs} \left\| \uh \right\|_{L^\infty_v \left( L^4_h \right)}   \left\| \uh \right\|_{ H^s \left( \R^3 \right)}^{1/2} \left\|  \nh \uh \right\|_{ H^s \left( \R^3 \right)}^{3/2}.
  \end{aligned}
    \end{equation}
  In the second inequality we used a Gagliardo-Nirenberg inequality and in the third one the regularity properties of dyadic blocks. The sequence $ \left( b_q \right)_q  \in \ell^1 \left( \mathbb{Z} \right) $. For the term $ I_{2,q} $ we can apply the very same procedure to deduce the same bound
  \begin{equation}\label{eq:bound_Iq2}
  I_{2,q}\leqslant  C b_q 2^{-2qs} \left\| \uh \right\|_{L^\infty_v \left( L^4_h \right)}   \left\| \uh \right\|_{ H^s \left( \R^3 \right)}^{1/2} \left\|  \nh \uh \right\|_{ H^s \left( \R^3 \right)}^{3/2},
  \end{equation}
  but in this case the sequence $ \left( b_q \right)_q $, which is $ \ell^1 $, assumes the convolution form
  $$
  b_q = c_q \sumi 2^{-\left( q'-q \right)s } c_{q'}.
  $$
  Thanks to \eqref{eq:bound_Iq1}, \eqref{eq:bound_Iq2} we hence deduced that
  \begin{equation}\label{eq:bound_bilinear}
  \left| \left(\left. \cFFStq \left( \uh \cdot\nh \uh \right) \right|   \cFFStq \uh \right)_{\cFFSLtwo} \right| \leqslant C b_q 2^{-2qs} \left\| \uh \right\|_{L^\infty_v \left( L^4_h \right)}   \left\| \uh \right\|_{ H^s \left( \R^3 \right)}^{1/2} \left\|  \nh \uh \right\|_{ H^s \left( \R^3 \right)}^{3/2}.
  \end{equation}
  With the bound \eqref{eq:bound_bilinear} applied to \eqref{eq:energy_est1} we deduce
  \begin{equation}\label{eq:energy_est2}
  \frac{1}{2}\ \frac{\d}{\d t} \left\| \cFFStq \uh \right\|_{{\cFFSLtwo}}^2 + \nu \left\| \cFFStq \nabla \uh \right\|_{{\cFFSLtwo}}^2 \leqslant C b_q 2^{-2qs} \left\| \uh \right\|_{L^\infty_v \left( L^4_h \right)}   \left\| \uh \right\|_{ H^s \left( \R^3 \right)}^{1/2} \left\|  \nh \uh \right\|_{ H^s \left( \R^3 \right)}^{3/2}, 
  \end{equation}
  hence, multiplying \eqref{eq:energy_est2} for $ 2^{2qs} $, summing on $ q\in \ZZ $ and using the convexity inequality $ ab \leqslant C a^4 +\frac{\nu}{2} b^{4/3} $ we deduce the bound
  \begin{equation}\label{eq:energy_est3}
  \frac{\d}{\d t} \left\| \uh \right\|_{ H^s \left( \R^3 \right)}^2 + \nu \left\| \nabla \uh \right\|_{ H^s \left( \R^3 \right)}^2 \leqslant \left\| \uh \right\|_{L^\infty_v \left( L^4_h \right)}^4 \left\| \uh \right\|_{ H^s \left( \R^3 \right)}^2.
  \end{equation}
  it suffice hence to apply Gronwall inequality on \eqref{eq:energy_est3} and consider the result of Lemma \ref{lem:bound_anisotropic_norm} to deduce the bound \eqref{eq:Hs_bound_uh}.
 \end{proof}
 
The following result is a direct deduction of the above proposition, the proof is hence omitted. 
 
 \begin{cor}
The solutions of \eqref{eq:2DNS_stratified} are $ \cFFSHud $-stable if the initial data belong to the space $ \cFFSLtwo\cap \cFFSHud $.
 \end{cor}

\section{Dispersive properties}
\label{sec:dispersive_properties}

We recall that, for $0 < r < R$, in \eqref{eq:CrR}, we defined
\begin{equation*}
	\Crr = \set{ \xi \in \R^3_\xi : \left| \xi_h \right|>r, \ \left| \xi \right|<R }. 
\end{equation*}
Let $\psi$ a $\mathcal{C}^{\infty}$-function from $\RR^3$ to $\RR$ such that
\begin{equation*}
	\chi(\xi) = \left\{ \begin{aligned} &1 \qquad\mbox{if } \quad 0 \leqslant \abs{\xi} \leqslant 1\\&0 \qquad\mbox{if} \quad \abs{\xi} \geqslant 2 \end{aligned} \right.
\end{equation*}
and $\Psi_{r, R}: \RR^3 \to \RR$ the following frequency cut-off function
\begin{equation}
	\label{eq:PsirR} \Psi_{r, R}(\xi) = \chi \pare{\frac{\abs{\xi}}{R}} \pint{1- \chi \pare{\frac{\abs{\xi_h}}{r}} } .
\end{equation}
Then, we have $\Psi_{r, R} \in \mathcal{D}(\RR^3)$, $\mbox{supp\,} \Psi_{r, R} \subset \mathcal{C}_{\frac{r}2,2R}$ and $\Psi_{r, R} \equiv 1 \mbox{ on }\mathcal{C}_{r,R}$. 
Indeed the operator $ \Psi_{r, R} $ maps 
any tempered distribution $f$ to
\begin{equation}
	\label{eq:PrR}
	 \Psi_{r, R}(D) f = \mathcal{F}^{-1} \pare{\Psi_{r, R}(\xi) \widehat{f}(\xi)},
	\end{equation}
	with this in mind we want to study the following linear system
	\begin{equation}
	\label{eq:free_wave_CFFS}
	\left\lbrace
	\begin{aligned}
	 & \partial_t \Wrr     + \frac{1}{\varepsilon}\ L_\varepsilon \Wrr =
-\Psi_{r, R} \left( D \right) \left( \PP_{+, \varepsilon} + \PP_{-, \varepsilon} \right) \Lambda \left( \uh \right) 
	 ,\\
	 & \dive w^\varepsilon_{r, R}=0,\\
	 & \left. \Wrr \right|_{t=0} =\Psi_{r, R} \left( D \right) \left( \PP_{+, \varepsilon} + \PP_{-, \varepsilon} \right) U_0,
	\end{aligned}
	\right.
	\end{equation}
	where $ \PP_{\pm, \varepsilon} $ is the projection respectively onto the space $ \mathbb{C} E_\pm^\varepsilon $ defined in \eqref{eq:projectors} and $ L_\varepsilon $ is defined in \eqref{eq:def_Leps}. We stress out the fact that Lemma \ref{lem:boundedness_projectors} implies that the maps $ \PP_{i, \varepsilon} $ are bounded operators onto $ L^2 $ as long as we consider functions localized on the set $ \Crr $.\\
	The forcing term $ \Lambda $ appearing on the right-hand-side of \eqref{eq:free_wave_CFFS} is defined as
\begin{equation}
\label{eq:def_Lambda}
\Lambda \left( \uh \right)=
\left( 
\begin{array}{c}
0\\
0\\
\partial_3 \bar{p} \left( \uh \right)\\
0
\end{array}
 \right),	
\end{equation}	
	where the scalar function $ \bar{p} $ the limit pressure of the limit system \eqref{eq:2DNS_stratified}. We expressed the nonlinearity $ \Lambda $ as depending on the velocity flow $ \uh $, but in the above definition the dependence on $ \bar{p} $ is made explicit. Indeed we can express $ \bar{p} $ it in term of $ \uh $ as
	\begin{equation}
	\label{eq:def_pbar}
	\begin{aligned}
	\bar{p} = & \ \left( -\Dh \right)^{-1} \diveh \left( \uh \cdot \nh \uh \right),\\
	= & \ \left( -\Dh \right)^{-1} \diveh \diveh \left( \uh \otimes \uh \right),
	\end{aligned}
	\end{equation}
	and this justifies the above observation.\\

The forcing term $ \Lambda $	 presents an interesting property

\begin{lemma}\label{lem:P0Lambda=0}
Let $ \PP_0 $ be the projector onto the non-oscillating subspace  defined in \eqref{eq:definition_P0}, then
\begin{equation*}
\PP_0 \Lambda =0.
\end{equation*}
\end{lemma}
\begin{proof}
It suffice to remark that the only non-zero component of $ \Lambda $ is the third one and that the projector $ \PP_0  $ defined in \eqref{eq:definition_P0} maps the third component to zero.
\end{proof}

Lemma \ref{lem:P0Lambda=0} implies in particular that
\begin{equation}\label{eq:Lambda=P+P-Lambda}
\Lambda = \left( \PP_{+, \varepsilon}+\PP_{-, \varepsilon} \right)\Lambda,
\end{equation}
and hence we shall use \eqref{eq:Lambda=P+P-Lambda} repeatedly along this work.

	The presence of the external forcing term $ -\Psi_{r, R} \left( D \right)\left( \PP_{+, \varepsilon}+\PP_{-, \varepsilon} \right) \Lambda $ is motivated by technical needs which will be explained in detail in Section \ref{sec:bootstrap}.

\subsection{Study of the linear system \eqref{eq:free_wave_CFFS}}
	
	In this small section we prove some existence and regularity result concerning the free-wave system \eqref{eq:free_wave_CFFS}. Let us define the space 
	\begin{equation*}
	H^{1/2}_{r, R}
	=
	\set{ g \left| \  g= \Psi_{r, R}\left( D \right) f, \ f \in {\cFFSLtwo} \cap {\cFFSHud} \right. },
	\end{equation*}
	it is indeed trivial to deduce that $ H^{1/2}_{r, R} \subset H^{\frac{1}{2}} \left( \R^3 \right)= {\cFFSLtwo}\cap {\cFFSHud} $. The space $ H^{1/2}_{r, R} $ endowed with the $ H^{\frac{1}{2}} \left( \R^3 \right) $ norm is a Banach space.
	 
	\begin{lemma}
	\label{lem:existence_sol_free_wave}
	Let $ U_0 \in  H^{\frac{1}{2}} \left( \R^3 \right) $ such that $ \oh_0 \in \cFFSLtwo $, for each $ \varepsilon > 0 $ and $ 0<r<R $ there exist a solution $ \Wrr $ of \eqref{eq:free_wave_CFFS} in the space
	\begin{equation*}
	\Wrr \in \mathcal{C}^1 \left( \R_+ ;  H^{1/2}_{r, R} \right).
	\end{equation*}
	The sequence $ \left( \Wrr  \right)_{\substack{\varepsilon>0 }}$ is bounded in the space $ \Eud $  and for each $ t>0 $ and $ \varepsilon>0, \ 0< r <R $ the following bound holds true:
	\begin{multline}
	\label{eq:est_H1/2_Wrr}
	\left\| \Wrr \left( t \right) \right\|_{{\cFFSHud}}^2 + c \int_0^t\left\| \nabla \Wrr \pare{ s } \right\|_{{\cFFSHud}}^2 \d s
	\leqslant
	C_{r, R}
	\left\| U_0 \right\|_{{\cFFSHud}}^2 \\
	+
	\frac{C}{\nu} \left\| U_0 \right\|_{\cFFSHud}^4 \exp \set{ \frac{C K^2}{\nu}\left( \left\| \uh_0 \right\|_{\cFFSLtwo}^4 +\left\| \oh_0 \right\|_{\cFFSLtwo}^4 \right)},
	\end{multline}
	where $ c=\min \set{\nu, \nu'} $.
	\end{lemma}

	In order to prove  Lemma \ref{lem:existence_sol_free_wave} it suffices apply Cauchy-Lipschitz theorem in the following form.  
	
	\begin{lemma}\label{lem:BUcriterion}
	Let us consider  the ordinary differential equation
	\begin{equation}\tag{ODE}\label{eq:ODE}
	\left\lbrace
	\begin{aligned}
	& \dot{u}= F \pare{ u, t }\\
	& \left. u \right|_{t=0} = u_0 \in \omega
	\end{aligned}
	\right. ,
	\end{equation}
	where $\omega$ is an open subset of a Banach space $X$. Let
\begin{equation*}
\begin{aligned}
 &F : && \omega\times \R_+ && \to && X\\
 & && \pare{ u, t } && \mapsto && F \pare{ u, t }
\end{aligned},
\end{equation*}	
 be such that, for each $u_1, u_2 \in \omega$ there exists a function $L\in L^1_{\loc} \pare{ \R_{+} }$ such that 
	\begin{equation} \label{eq:local_Lip_condition}
	\norm{F\pare{ u_1, t }-F\pare{ u_2, t }}_X \leqslant L \pare{ t } \norm{u_1 - u_2}_X.
	\end{equation}

	\noindent
	 Let us suppose moreover that
	\begin{equation*}
	\norm{F \pare{ u, t }}_{X} \leqslant \beta \pare{ t } M \pare{ \norm{u}_{X} },
	\end{equation*}
	where $M\in L^\infty_{\loc} \pare{ \R_+ }, \beta \in L^1_{\loc} \pare{ \R_+ }$. Then  there exists a unique maximal solution $u$ in the space $\mathcal{C}^1 \pare{ [0, t^\star ) ; X }$  of \eqref{eq:ODE},  such that, if $t^\star <\infty$,
	\begin{equation*}
	\limsup_{t\nearrow t^\star} \norm{u\pare{ t }}_{X} =\infty.
	\end{equation*}
	\end{lemma}
	
	\begin{proof}
	See \cite[Proposition 3.11, p. 131]{bahouri_chemin_danchin_book}.
	\end{proof}

\textit{Proof of Lemma \ref{lem:existence_sol_free_wave} :}	It suffices to consider \eqref{eq:free_wave_CFFS} in the form
	\begin{equation*}
	\partial_t \Wrr = F_\varepsilon \left( t, \Wrr \right),
	\end{equation*}
	where (using as well \eqref{eq:Lambda=P+P-Lambda}):
	\begin{equation*}
	F_\varepsilon \left( t, \Wrr \right) = -\frac{1}{\varepsilon}L_\varepsilon W^\varepsilon_{r, R}- \Psi_{r, R} \left( D \right) \Lambda \left( \uh \left( t \right) \right).
	\end{equation*}
It is easy to prove that $F_\varepsilon $ satisfies \eqref{eq:local_Lip_condition} with a locally $L^1$ function which depends on $\varepsilon, r$ and $R$. We aim to prove that, for each $ r, R, \varepsilon>0 $ the function $ \Wrr $ belongs to the space $ \mathcal{C}^1 \left( \R_+; \cFFSHud \right) $: accordingly to Lemma \ref{lem:BUcriterion} it suffices hence to prove that
\begin{equation*}
\sup_{t\geqslant 0}\norm{\Wrr \pare{ t }}_{\cFFSHud} <\infty.
\end{equation*}
 Let us now multiply \eqref{eq:free_wave_CFFS} for $ \Wrr $ and let us take the $ {\cFFSHud}  $ scalar product of it, we deduce hence that
\begin{multline*}
\frac{1}{2}\frac{\d}{\d t} \left\| \Wrr \right\|_{{\cFFSHud}}^2 + c \left\| \nabla \Wrr \right\|_{{\cFFSHud}}^2\\
 \leqslant
\left| \left( \left. \Psi_{r, R} \left( D \right) \partial_3 \left( -\Dh \right)^{-1} \diveh \diveh \left( \uh\otimes \uh \right) \right| \Wrr \right)_{{\cFFSHud}} \right|,
\end{multline*}
where $ c=\min \set{\nu, \nu'} $.
Integration by parts, Cauchy-Schwartz inequality, and the fact that the operator $  \Psi_{r, R} \left( D \right)  \left( -\Dh \right)^{-1} \diveh \diveh $ maps continuously any $ \cFFSHs $ space to itself with norm independent of $ r $ and $ R $ allow us to deduce
\begin{multline*}
\left| \left( \left. \Psi_{r, R} \left( D \right) \partial_3 \left( -\Dh \right)^{-1} \diveh \diveh \left( \uh\otimes \uh \right) \right| \Wrr \right)_{{\cFFSHud}} \right|
\\
\leqslant
C \left\| \uh \otimes \uh \right\|_{{\cFFSHud}} \left\| \nabla \Wrr \right\|_{{\cFFSHud}},
\end{multline*}
and since
\begin{align*}
\left\| \uh \otimes \uh \right\|_{{\cFFSHud}} \leqslant & \ C \left\| \uh \right\|_{\dot{H}^1\left( \R^3 \right)}^2,\\
\leqslant & \ C \left\| \uh \right\|_{{\cFFSHud}} \left\| \nabla \uh \right\|_{{\cFFSHud}},
\end{align*}
whence applying Young inequality we obtain the estimate
\begin{equation*}
\frac{1}{2}\frac{\d}{\d t} \left\| \Wrr \right\|_{{\cFFSHud}}^2 + \frac{c}{2} \left\| \nabla \Wrr \right\|_{{\cFFSHud}}^2
\leqslant
C \left\| \uh \right\|_{{\cFFSHud}}^2 \left\| \nabla \uh \right\|_{{\cFFSHud}}^2.
\end{equation*}
Integrating in-time the above equation and  using the estimate \eqref{eq:Hs_bound_uh} we hence conclude the proof.
\hfill $\Box$

	\subsection{Dispersive properties of \eqref{eq:free_wave_CFFS}}
	
	In the previous section we made sure that \eqref{eq:free_wave_CFFS}  is solvable  in the classical sense and that the solutions of \eqref{eq:free_wave_CFFS} belong to the space $ \Eud $ uniformly w.r.t. the parameters $ \varepsilon, r, R $. In the present section we are hence interested to study the perturbation induced by the operator $ \varepsilon^{-1}L_\varepsilon $, and to prove that such perturbations induce some dispersive effect on $ \Wrr $.

	The result we want to prove in this section is the following one
	\begin{theorem}\label{thm:Strichartz_est}
	Let $ U_0\in {\cFFSLtwo} \cap \cFFSHud $, $ 0<r< R $ and $ \varepsilon >0 $. Then  $ \Wrr $  solution of  \eqref{eq:free_wave_CFFS} belongs to the space $ L^p \left( \R_+; L^\infty \left( \R^3 \right) \right), p\in [1, \infty ) $ and  if $ \varepsilon >0 $ is sufficiently small
	\begin{equation}
	\label{eq:Strichartz_estimates}
	\left\| \Wrr \right\|_{L^p \left( \R_+; L^\infty \left( \R^3 \right) \right)}	
	\leqslant C_{r, R}
	\left( 1+\frac{1}{\nu} \right)
\varepsilon^{\frac{1}{4p}} \	
\max \set{\left\| U_0 \right\|_{{\cFFSLtwo}}, \ 
	 \left\| U_0 \right\|_{{\cFFSLtwo}}^2 },
	\end{equation}
	for $ p\in \left[1, \infty\right) $.
	\end{theorem}
	
	Some preparation is indeed required in order to prove Theorem \ref{thm:Strichartz_est}.\\
	
	We can write the solution of \eqref{eq:free_wave_CFFS} as
	\begin{equation*}
	\Wrr \left( t \right)= e^{-\frac{t}{\varepsilon} \ L_\varepsilon} W^\varepsilon_{r, R, 0}-
	\int_0^t e ^{-\frac{t-s}{\varepsilon} \ L_\varepsilon} \Psi_{r, R} \left( D \right) \left( \PP_{+, \varepsilon}+\PP_{-, \varepsilon} \right) \Lambda \left( \uh \left( s \right) \right) \d s,
	\end{equation*}
whence along the eigendirection $ E_\pm^\varepsilon $ the evolution of \eqref{eq:free_wave_CFFS} is
	\begin{equation}
	\label{eq:evolution_operators}
	\begin{aligned}
	\PP_{\pm, \varepsilon} \left( \Wrr \right) \left( t,x \right) = & \ \mathcal{F}^{-1} \left( k_\pm \left( \Wrr \left( t \right) \right) E_\pm^\varepsilon \right) \left( x \right),\\
	= & \ \int_{\R^3_y\times \R^3_\xi} e^{\pm i \frac{t}{\varepsilon} \lambda^\varepsilon_\pm \left( \xi \right) + i \xi \left( x-y \right)} \Psi_{r, R} \left( \xi \right) \PP_{\pm, \varepsilon} \left( U_0 \right) \left( y \right) \d y \  \d \xi\\
	& \ - \int_0^t \int_{\R^3_y\times \R^3_\xi} e^{\pm i \frac{t-s}{\varepsilon} \lambda^\varepsilon_\pm \left( \xi \right) + i \xi \left( x-y \right)} \Psi_{r, R} \left( \xi \right) \PP_{\pm, \varepsilon} \Lambda \left( \uh \left( s, y \right) \right) \d y \  \d \xi \ \d s,
	\\
	= & \  \mathcal{K}_{\pm, r , R} \left( t, \frac{t}{\varepsilon}, \cdot \right)\star \PP_{\pm, \varepsilon} \left( U_0 \right) \left( x \right)\\
	 & \ - \int_0^t
	  \mathcal{K}_{\pm, r , R} \left( t-s, \frac{t-s}{\varepsilon}, \cdot \right)\star \PP_{\pm, \varepsilon} \Lambda \left( \uh \left( s, \cdot \right)  \right)\d s, \\
	= & \ \mathcal{G}_{\pm, r , R}^\varepsilon \left( \frac{t}{\varepsilon} \right) U_0 \left( x \right)
-
\int_0^t 	\mathcal{G}_{\pm, r , R}^\varepsilon \left( \frac{t-s}{\varepsilon} \right) \Lambda \left( \uh \left( s \right) \right) \ \left( x \right) \d s
	.
\end{aligned}
\end{equation}
	where $ \lambda_\pm^\varepsilon $ is defined in \eqref{eq:eigenvalue_pm}. The convolution kernels $ \mathcal{K}_{\pm, r, R} $ are
	\begin{equation}
	\label{eq:convolution_kernel}
	\begin{aligned}
	\mathcal{K}_{\pm, r , R} \left( t, \tau,  z \right) 
	= & \ \int_{\R^3_\xi} e^{\pm i \tau \frac{\left| \xi_h \right|}{\left| \xi \right|} \ S_\varepsilon \left( \xi \right) - \frac{1}{2} \left( \nu+\nu' \right) \left| \xi \right|^2 t +i \xi \cdot z}\Psi_{r, R} \left( \xi \right) \d \xi.
		\end{aligned}
	\end{equation}
	The convolution kernel $ \mathcal{K}_{\pm, r , R} $ is hence a highly oscillating integral. It is well known that integrals with such a behavior are $ L^\infty \left( \R^3 \right) $ functions whose $ L^\infty \left( \R^3 \right) $ norm decays in time (see \cite{AlinhacGerard}, \cite{bahouri_chemin_danchin_book}, \cite{monographrotating}, \cite{Stein93}...), we shall apply the methodology of \cite{monographrotating} in order to prove the following result
	
	\begin{lemma}\label{lem:dispersive_estimates}
	For any $ r, R $ such that $ 0<r<R $ there exists a constant $ C_{r,R} $ such that for each $ z\in \mathbb{R}^3 $
	\begin{equation}\label{eq:dispersive_estimates}
	\left| \mathcal{K}_{\pm, r , R} \left( t, \tau, z \right) \right|\leqslant C_{r, R} \  \min \set{ 1, \tau^{-1/2} }e^{-\frac{1}{4} \left( \nu+ \nu' \right) r^2 t}.
	\end{equation}
	\end{lemma}
	
	\begin{proof}
	Taking the modulus of both sides of \eqref{eq:convolution_kernel} and integrating, considering that $ \Psi_{r, R} $ is supported in $ \Crr $, it is sufficient to prove that
	$$
	\left| \mathcal{K}_{\pm, r , R} \left( t, \tau, z \right) \right|\leqslant C_{r, R} \ e^{-\frac{1}{2} \left( \nu+ \nu' \right) r^2 t},
	$$
	for each $ t, \tau \in \R_+ $ and $ z\in \R^3 $. This holds hence in particular if $ \tau \in [0,1] $.\\
	
	The rest of the proof is devoted to improve the above estimate in the case $ \tau \geqslant 1 $.\\

	Let us fix some notation first, we denote as $ \phi \left( \xi \right)= \frac{\left| \xi_h \right|}{\left| \xi \right|}\ S_\varepsilon \left( \xi \right)  $ and thanks to Fubini's theorem
	\begin{align*}
	\left| \mathcal{K}_{\pm, r , R} \left( t, \tau, z \right) \right|
	= & \ \int_{\R^3_\xi} e^{\pm i \tau \phi \left( \xi \right)  - \frac{1}{2} \left( \nu+\nu' \right) \left| \xi \right|^2 t +i \xi \cdot z}\Psi_{r, R} \left( \xi \right) \d \xi\\
	= & \ \int_{\R^2_{\xi_h}} e^{i\ \xi_h \cdot z_h}
	\left( 
\int_{\R_{\xi_3}}	 e^{\pm i \tau \phi \left( \xi \right)  - \frac{1}{2} \left( \nu+\nu' \right) t \ \left| \xi \right|^2 +i \xi_3 \cdot z_3}\Psi_{r, R} \left( \xi \right) \d \xi_3
	 \right) \d\xi_h,\\
	 	= & \ \int_{\R^2_{\xi_h}} e^{i\ \xi_h \cdot z_h} \mathcal{I}_{\pm, r , R} \left( t, \tau, \xi_h, z_3 \right) \d \xi_h.
	\end{align*} 
	Indeed since $ \mathcal{I}_\pm $ is supported, relatively to the variable $ \xi_h $, in the set $ \set{\xi_h: r\leqslant \left| \xi_h \right|\leqslant R} $, we deduce
	$$
	\left| \mathcal{K}_{\pm, r , R} \left( t, \tau, z \right) \right| \leqslant C_{r, R} \left|  \mathcal{I}_{\pm, r , R} \left( t, \tau, \xi_h, z_3 \right) \right|, 
	$$
	hence it shall suffice to prove an $ L^\infty $ bound for the function $ \mathcal{I}_\pm $. Let us remark that $ \mathcal{I}_\pm $ are even functions w.r.t. the variable $ z_3 $, hence we can restrict ourselves to the case $ z_3\geqslant 0 $. \\
	We are interested to study the $ L^\infty $ norm of the elements $ \mathcal{I}_\pm $, these norms are invariant under dilation, in particular hence we consider the transformation $ z_3 \mapsto \tau z_3, \ \tau >1 $, with these
	\begin{equation*}
	\mathcal{I}_{\pm, r, R} \left( t, \tau, \xi_h, \tau z_3 \right)
	=
	\int_{\R_{\xi_3}}	 e^{i \tau\left( \pm   \phi \left( \xi \right) + \xi_3  z_3  \right) - \frac{1}{2} \left( \nu+\nu' \right) t \ \left| \xi \right|^2 }\Psi_{r, R} \left( \xi \right) \d \xi_3.
	\end{equation*}
	Let us fix some notation, we define
	\begin{align*}
	\Phi \left( \xi \right)= & \ \partial_{\xi_3}\phi \left( \xi \right) \\
	= & \ \left( \frac{\left| \xi_h \right|}{\left| \xi \right|} S_\varepsilon \left( \xi \right)-\varepsilon^2 \left( \nu-\nu' \right) \frac{ \left| \xi \right| \left| \xi_h \right|}{S_\varepsilon \left( \xi \right)} \right) \xi_3,\\
	\theta_\pm \left( \xi, z_3 \right) = & \ \pm   \phi \left( \xi \right) + \xi_3  z_3, \\
	\Theta_\pm \left( \xi, z_3 \right) = & \ \partial_{\xi_3} \theta_\pm \left( \xi, z_3 \right) ,\\
	= & \ \left( \frac{\left| \xi_h \right|}{\left| \xi \right|} S_\varepsilon \left( \xi \right)-\varepsilon^2 \left( \nu-\nu' \right) \frac{ \left| \xi \right| \left| \xi_h \right|}{S_\varepsilon \left( \xi \right)} \right) \xi_3 + z_3.
	\end{align*}
	With this notation indeed
	\begin{equation*}
	\mathcal{I}_{\pm, r , R} \left( t, \tau, \xi_h, \tau z_3 \right)
	=
	\int_{\R_{\xi_3}}	 e^{i \tau \theta_\pm \left( \xi, z_3 \right) - \frac{1}{2} \left( \nu+\nu' \right) t \ \left| \xi \right|^2 }\Psi_{r, R} \left( \xi \right) \d \xi_3.
	\end{equation*}
	 Let us define the differential operator
	$$
	\mathcal{ L}_\pm :=  \frac{1}{1+\tau \ \Theta_\pm^2 \left( \xi, z_3 \right) } \left( 1+i \ \Theta_\pm \left( \xi, z_3 \right)  \partial_{\xi_3}\right),
	$$
	in particular there exists a positive constant $ C $ independent by any parameter of the problem such that, being $ \xi\in \mathcal{C}_{r, R} $ defined in \eqref{eq:CrR},
	\begin{equation}
	\label{eq:bound_Phi}
	\frac{r^2}{C R} \xi_3 + z_3 \leqslant \left| \Theta_\pm \left( \xi, z_3 \right)  \right| \leqslant \frac{C R^2}{r} \xi_3+z_3.
	\end{equation}
	Indeed $ \mathcal{L}_\pm \left( e^{i \tau \theta_\pm} \right)= e^{i \tau \theta_\pm} $, hence integration by parts yields 
\begin{equation}
\label{eq:I_after_IBP}
	\mathcal{I}_\pm \left( t, \tau, \xi_h, \tau z_3 \right) = \int_{\R^1_{\xi_3}}	 e^{i\tau \theta_\pm \left( \xi, z_3 \right) } \mathcal{L}^\intercal_\pm \left( \Psi_{r, R} \left( \xi \right) e^{-\frac{1}{2} \left( \nu+\nu' \right) t \ \left| \xi \right|^2} \right)\d \xi_3,
	\end{equation}
	where
	\begin{multline*}
	\mathcal{L}^\intercal_\pm \left( \Psi_{r, R} \left( \xi \right) e^{-\frac{1}{2} \left( \nu+\nu' \right) t \ \left| \xi \right|^2} \right) =
	 \left( \frac{1}{1+\tau \Theta^2_\pm} - i \left( \partial_{\xi_3}\Theta_\pm \right) 
\frac{1-\tau \Theta^2_\pm}{\left( 1+\tau \Theta^2_\pm \right)^2}	
	 \right)
	 \Psi_{r, R} \left( \xi \right) e^{-\frac{1}{2} \left( \nu+\nu' \right) t \ \left| \xi \right|^2}\\
	 -
	 \frac{i \Theta}{1+\tau \Theta^2_\pm} \partial_{\xi_3} \left(  \Psi_{r, R} \left( \xi \right) e^{-\frac{1}{2} \left( \nu+\nu' \right) t \ \left| \xi \right|^2} \right).
	\end{multline*}
	Since $ \xi\in \Crr $ and thanks to the estimate \eqref{eq:bound_Phi} we can deduce easily that (here we use the fact that $ \left| \frac{1-\tau \Theta^2_\pm}{\left( 1+\tau \Theta^2_\pm \right)^2} \right| \leqslant \left| \frac{1}{1+\tau \Theta^2_\pm} \right| $)
	\begin{equation*}
	\frac{1}{1+\tau \Theta^2_\pm}  \leqslant \frac{C_{r, R}}{1+\tau \xi_3^2}.
	\end{equation*}
	Moreover
	\begin{align*}
	  \frac{\left|  \Theta_\pm \right|}{1+\tau \left| \Theta_\pm \right|^2} \leqslant & \ C_{r,R}\ \frac{1+z_3}{1+\tau \left| z_3+\xi_3 \right|^2},\\
	  \leqslant & \ C_{r,R}\ \frac{1+z_3}{1+\tau z_3^2 + \tau \xi_3^2},\\
	  \leqslant & \ C_{r,R}\ \frac{1+z_3}{\left( 1+\sqrt{\tau} z_3 \right)^2}\ \frac{1}{1+\tau \xi_3^2} \leqslant \ C_{r,R}\ \frac{1}{1+\tau \xi_3^2}.
	\end{align*}
	The last inequality is true since $ \tau >1 $. 
	Being $ \xi $ localized in $ \Crr $ is a matter of straightforward computations to prove that
	$$
	\left| \partial_{\xi_3}\Theta_\pm \right| \leqslant C_{r, R}, 
	$$
	moreover, being $ \Psi_{r, R}\in \mathcal{D} $,
	\begin{equation*}
	\left| \partial_{\xi_3} \left(  \Psi_{r, R} \left( \xi \right) e^{-\frac{1}{2} \left( \nu+\nu' \right) t \ \left| \xi \right|^2} \right) \right|
	\leqslant
	C_{r, R}
	e^{-\frac{1}{4} \left( \nu+ \nu' \right) r^2 t},
	\end{equation*}
	whence we finally deduced that
	\begin{equation*}
	\left| \mathcal{L}^\intercal_\pm \left( \Psi_{r, R} \left( \xi \right) e^{-\frac{1}{2} \left( \nu+\nu' \right) t \ \left| \xi \right|^2} \right) \right|
	\leqslant \frac{C_{r, R}}{1+\tau \xi_3^2} \ 
	e^{-\frac{1}{4} \left( \nu+ \nu' \right) r^2 t}.
	\end{equation*}
	With the above bound and \eqref{eq:I_after_IBP} we deduce hence
	\begin{align*}
	\left| \mathcal{I}_{\pm, r , R} \left( t, \tau, \xi_h, \tau z_3 \right) \right| \leqslant & \ C_{r, R}
		e^{-\frac{1}{4} \left( \nu+ \nu' \right) r^2 t} \int_{\R^1_{\xi_3}} \frac{\d\xi_3}{1+\tau \xi_3^2},\\
		\leqslant
		& \ C_{r, R} \ \tau^{-1/2} \ 
		e^{-\frac{1}{4} \left( \nu+ \nu' \right) r^2 t}
	\end{align*}
	which concludes the proof.	
	\end{proof}
	
	\begin{prop}
	Let us consider a vector field $ U_0 \in {\cFFSLtwo} $ and the functions $ \mathcal{G}_{\pm, r , R}^\varepsilon \ U_0 $ of the variables $ \left( t, x \right) $ defined in \eqref{eq:evolution_operators}. Then
	\begin{equation}
	\label{eq:Strichartz_estimates2}
	\left\| \mathcal{G}_{\pm, r , R}^\varepsilon U_0 \right\|_{L^p \left( \R_+; L^\infty \left( \R^3 \right) \right)} \leqslant C_{r,R} \  \varepsilon^{\frac{1}{4 p}} \left\| U_0 \right\|_{\cFFSLtwo},
	\end{equation}
	for each $ p \in \left[ 1 , \infty \right) $.
	\end{prop}
	\begin{proof}
	Indeed $ \mathcal{G}_{\pm, r , R}^\varepsilon U_0 $ can be written as a convolution operator as explained in equation \eqref{eq:evolution_operators}, in particular
	\begin{equation*}
	 \ \mathcal{G}_{\pm, r , R}^\varepsilon \left( \frac{t}{\varepsilon} \right) U_0 \left( x \right)
	 =
	 \  \mathcal{K}_{\pm, r , R} \left( t, \frac{t}{\varepsilon}, \cdot \right)\star \PP_{\pm, \varepsilon} \left( U_0 \right) \left( x \right),
	\end{equation*}
	where $ \PP_{\pm, \varepsilon} $ are the projections onto the eigenspaces generated by $ E_\pm^\varepsilon $ defined in \eqref{eq:projectors}, and the convolution kernels $ \mathcal{K}_{\pm, r , R} $ are defined in \eqref{eq:convolution_kernel}. Considering the dispersive estimate \eqref{eq:dispersive_estimates} given in Lemma \ref{lem:dispersive_estimates} we can apply what is known as $ TT^\star $ argument (see \cite[Chapter 8]{bahouri_chemin_danchin_book}) in the same way as it is done in \cite{monographrotating}, \cite{CDGG}, \cite{CDGG2}, \cite{charve1}, \cite{charve_ngo_primitive} to deduce that
	\begin{equation*}
	\left\| \mathcal{K}_{\pm, r , R} \star \PP_{\pm, \varepsilon} \left( U_0 \right) \right\|_{L^1 \left( \R_+; L^\infty \left( \R^3 \right) \right)} \leqslant C_{r, R} \varepsilon^{1/4} \left\| \PP_{\pm, \varepsilon} \left( U_0 \right) \right\|_{\cFFSLtwo}.
	\end{equation*}
	We can hence apply Lemma \ref{lem:boundedness_projectors} obtaining   
	\begin{equation} \label{eq:interp1}
	\left\| \mathcal{G}_{\pm, r , R}^\varepsilon U_0 \right\|_{L^1 \left( \R_+; L^\infty \left( \R^3 \right) \right)} \leqslant C_{r,R} \varepsilon^{1/4} \left\| U_0 \right\|_{\cFFSLtwo}.
	\end{equation}
		The element $ \mathcal{G}_{\pm, r , R}^\varepsilon \left( \frac{t}{\varepsilon} \right) U_0 $ has the following properties:
	\begin{itemize}
	\item $ \mathcal{G}_{\pm, r , R}^\varepsilon \left( \frac{t}{\varepsilon} \right) U_0 $ is localized in the frequency space,
	
	\item $ \left\| \mathcal{G}_{\pm, r , R}^\varepsilon  U_0 \right\|_{ L^\infty \left( \R_+; {\cFFSLtwo} \right)} \leqslant \norm{ U_0}_{{\cFFSLtwo}} $,
	\end{itemize}
	whence an application of Bernstein inequality allows us to deduce that
	\begin{equation} \label{eq:interp2}
		\left\| \mathcal{G}_{\pm, r , R}^\varepsilon  U_0 \right\|_{ L^\infty \left( \R_+; L^\infty \left( \R^3 \right) \right)} \leqslant C_{r, R} \norm{ U_0}_{{\cFFSLtwo}}.
	\end{equation}
	An interpolation between \eqref{eq:interp1} and \eqref{eq:interp2} gives finally \eqref{eq:Strichartz_estimates2}.
	\end{proof}

	The oscillating behavior of the propagator allows us to deduce the following dispersive result on the external forcing $ -\Psi_{r, R} \left( D \right)\Lambda $ as it is done, for instance, in \cite{GiVe}, \cite{GrenierDesjardins} or \cite{charve1}.

	\begin{prop}
	There exists a constant $ C_{r, R} $ depending on the localization \eqref{eq:CrR} such that, for $ \varepsilon $ small
	\begin{equation}
	\label{eq:Strichartz_estimates3}
	\left\| \int_0^t \mathcal{G}^\varepsilon_{\pm, r, R} \left( \frac{\cdot -s}{\varepsilon} \right) \Psi_{r, R}\left( D \right)\Lambda \left( s \right) \d s \right\|_{L^p \left( \R_+; L^\infty \left( \R^3 \right) \right)}
	\leqslant
	C_{r, R} \varepsilon^{\frac{1}{4p}} \left\| \Psi_{r, R}\left( D \right)\Lambda \right\|_{L^1 \left( \R_+; {\cFFSLtwo} \right)},
	\end{equation}
	for each $ p\geqslant 1 $.
	\end{prop}

\begin{proof}
For this proof only we write $ \mathcal{G}^\varepsilon_{\pm, r, R} = \mathcal{G}, \ \Psi_{r, R} = \Psi $ in order to simplify the notation,
\begin{equation*}
	\left\| \int_0^t \mathcal{G} \left( \frac{\cdot -s}{\varepsilon} \right) \Psi\left( D \right)\Lambda \left( s \right) \d s \right\|_{L^1 \left( \R_+; L^\infty \left( \R^3 \right) \right)}
	\leqslant
	\int_0^\infty  \int_0^t \left\| \mathcal{G} \left( \frac{t -s}{\varepsilon} \right) \Psi \left( D \right)\Lambda \left( s \right) \right\|_{L^\infty} \d s \ \d t,
	\end{equation*}
	applying Fubini theorem and performing the change of variable $ \tau= t-s $ we deduce
	\begin{align*}
	\left\| \int_0^t \mathcal{G} \left( \frac{\cdot -s}{\varepsilon} \right) \Psi\left( D \right)\Lambda \left( s \right) \d s \right\|_{L^1 \left( \R_+; L^\infty \left( \R^3 \right) \right)}
	\leqslant
	& \ 
	\int_0^\infty \int_0^\infty \left\| \mathcal{G} \left( \frac{\tau}{\varepsilon} \right)
\Psi \left( D \right)\Lambda \left( s \right)	
	 \right\|_{ L^\infty \left( \R^3 \right) } \d \tau \ \d s,\\
	 = & \ \int_0^\infty 
	 \left\| \mathcal{G} 
\Psi \left( D \right)\Lambda \left( s \right)	
	 \right\|_{ L^1 \left( \R_{+,\tau}  L^\infty \left( \R^3 \right) \right) } \d s,
	\end{align*}
	whence applying \eqref{eq:Strichartz_estimates2} we deduce that 
	\begin{equation*}
	\left\| \mathcal{G} 
\Psi \left( D \right)\Lambda \left( s \right)	
	 \right\|_{ L^1 \left( \R_{+,\tau}  L^\infty \left( \R^3 \right) \right) } \leqslant C_{r, R} \varepsilon^{1/4} \left\| \Psi \left( D \right)\Lambda \left( s \right) \right\|_{{\cFFSLtwo}},
	\end{equation*}
	which in turn implies the claim for $ p=1 $.\\
	
	To lift up the argument to a generic $ p $ it suffice  to notice that, being $ \Psi \left( D \right) \Lambda $ localized in $ \Crr $, there exists a constant $ C_R $ depending on the magnitude of the localization $ \Crr $ such that
	\begin{equation*}
	\left\| \int_0^t \mathcal{G}^\varepsilon_{\pm, r, R} \left( \frac{\cdot -s}{\varepsilon} \right) \Psi_{r, R}\left( D \right)\Lambda \left( s \right) \d s \right\|_{L^\infty \left( \R_+; L^\infty \left( \R^3 \right) \right)}
	\leqslant
	C_R \left\| \Psi \left( D \right) \Lambda \right\|_{L^1 \left( \R_+; L^2 \left( \R^3 \right) \right)},
	\end{equation*}
	hence \eqref{eq:Strichartz_estimates3} follows by interpolation.
\end{proof}

	\textit{Proof of Theorem \ref{thm:Strichartz_est}:} In order to prove Theorem \ref{thm:Strichartz_est} it suffice to collect all the results proved in the present section. By superposition we obviously have that
	\begin{equation*}
	\Wrr = \PP_{-, \varepsilon}\Wrr + \PP_{+, \varepsilon} \Wrr, 
	\end{equation*}
	and applying \eqref{eq:evolution_operators}
	\begin{equation*}
	\PP_{\pm, \varepsilon}\Wrr  = \mathcal{G}^\varepsilon_{\pm, r, R} U_0 - \int _0^\cdot  \mathcal{G}^\varepsilon_{\pm, r, R}\left( \frac{\cdot -s}{\varepsilon} \right)\Psi_{r, R} \left( D \right) \Lambda \left( \uh \left( s \right) \right) \d s,
	\end{equation*}
	whence it suffice to apply \eqref{eq:Strichartz_estimates2} and \eqref{eq:Strichartz_estimates3} to deduce
	\begin{align*}
	\left\| \Wrr \right\|_{L^p \left( \R_+; L^\infty \left( \R^3 \right) \right)}	
	\leqslant & \  C_{r,R} \  \varepsilon^{\frac{1}{4 p}} \left( \left\| U_0 \right\|_{\cFFSLtwo}
	+
	\left\| \Psi_{r, R} \left( D \right) \Lambda \left( \uh \right) \right\|_{L^1 \left( \R_+; {\cFFSLtwo} \right)}
	 \right),
	\end{align*}
	however, since
	\begin{equation*}
	\left\| \Psi_{r, R} \left( D \right) \Lambda \left( \uh \right) \right\|_{L^1 \left( \R_+; {\cFFSLtwo} \right)} \leqslant R^{1/2} \left\| \uh \right\|_{L^2 \left( \R_+; \dot{H}^1 \left( \R^3 \right) \right)}^2,
	\end{equation*}
	and thanks to the results of Lemma \eqref{lem:Ler_sol_ubar_omega} we can hence argue that
	\begin{equation*}
	\left\| \Psi_{r, R} \left( D \right) \Lambda \left( \uh \right) \right\|_{L^1 \left( \R_+; {\cFFSLtwo} \right)} \leqslant \frac{C_{r, R}}{\nu} \left\| U_0 \right\|_{{\cFFSLtwo}}^2,
	\end{equation*}
	which implies in turn that
	\begin{equation*}
	\left\| \Wrr \right\|_{L^p \left( \R_+; L^\infty \left( \R^3 \right) \right)}	
	\leqslant C_{r, R}
	\left( 1+\frac{1}{\nu} \right)
\varepsilon^{\frac{1}{4p}} \	
\max \set{\left\| U_0 \right\|_{{\cFFSLtwo}}, \ 
	 \left\| U_0 \right\|_{{\cFFSLtwo}}^2 },
	\end{equation*}
	concluding.
	 \hfill $ \Box $

 \section{Long time behavior: the bootstrap procedure}\label{sec:bootstrap}
 
 This section is devoted to deduce the maximal lifespan of the function
 \begin{equation}
 \label{eq:definition_delta}
 \drr= U^\varepsilon - \Wrr -\bar{U},
 \end{equation}
 where $ U^\varepsilon $ is the local solution of \eqref{perturbed BSSQ} identified in the Theorem \ref{thm:FK_applied}, $ \Wrr $ is the global solution of the free-wave system \eqref{eq:free_wave_CFFS} and $ \bar{U} $ is the global solution of the limit system identified in Section \ref{sec:formal_derivation_limit_system} , i.e. the system \eqref{eq:2DNS_stratified}. By the definition itself of $ \drr $ we understand that, being $ \bar{U} $ and $ \Wrr $ globally well-posed, $ U^\varepsilon $ and $ \drr $  have the same lifespan. 
 
This first regularity result is a very rough bound on the $ \dot{\mathcal{E}}^0 $ norm of $ \drr $:
 \begin{lemma}
 \label{lem:L2regularity_delta}
 Let $ U_0 \in \cFFSLtwo \cap {\cFFSHud} $ such that $ \oh_0\in \cFFSLtwo $, the function $ \drr $ defined as in \eqref{eq:definition_delta} belongs uniformly in $ \varepsilon, r, R >0 $ to the space $ \dot{\mathcal{E}}^0 \left( \R^3 \right) $ and 
 \begin{multline*}
 \left\| \drr \right\|_{\dot{\mathcal{E}}^0 \left( \R^3 \right)}^2 
 \leqslant C_{r, R} \left( 1+ \frac{1}{c} \right) \left\| U_0 \right\|_{\cFFSLtwo}^2
 \\
 +
 C
\left( 1+ \frac{1}{c^2}  \right)\left\| U_0 \right\|_{\cFFSLtwo}^2 \left\| U_0 \right\|_{\cFFSHud}^2 \exp \set{ \frac{C K^2}{\nu}\left( \left\| U_0 \right\|_{\cFFSLtwo}^4 +\left\| \oh_0 \right\|_{\cFFSLtwo}^4 \right)} 
 ,
 \end{multline*}
 where $ c=\min\set{\nu, \nu'} $.
 \end{lemma}
 \begin{proof}
 Theorem \ref{thm:Leray_applied} implies that $ U^\varepsilon\in \dot{\mathcal{E}}^0 $ as well as Lemma \ref{lem:Ler_sol_ubar_omega} implies that $ \bar{U} \in \dot{\mathcal{E}}^0 $ and moreover 
 \begin{equation*}
 \left\| U^\varepsilon \right\|_{\dot{\mathcal{E}}^0 \left( \R^3 \right)}^2 + \left\| \bar{U} \right\|_{\dot{\mathcal{E}}^0 \left( \R^3 \right)}^2\leqslant C \left( 1+ \frac{1}{c} \right) \left\| U_0 \right\|_{\cFFSLtwo}^2,
 \end{equation*}
 where $ c=\min\set{\nu, \nu'} $.  For $ \Wrr $ the procedure is similar: let us multiply \eqref{eq:free_wave_CFFS} for $ \Wrr $ and let us integrate in space. Recalling that $ \bar{p}= \left( -\Dh \right)^{-1} \dive \dive \left( \uh \otimes \uh \right)= p_0 \left( D \right) \left( \uh \otimes \uh \right) $ it suffice to prove a suitable energy bound on the element 
 \begin{equation*}
 \left| \left(\left. \partial_3  \left( \uh \otimes \uh \right) \right|  \Wrr \right)_{\cFFSLtwo} \right|.
 \end{equation*}
 Integration by parts and Young inequality allow us to deduce that
 \begin{equation*}
 \left| \left(\left. \partial_3  \left( \uh \otimes \uh \right) \right|  \Wrr \right)_{\cFFSLtwo} \right| \leqslant
 \frac{c}{2} \left\| \nabla \Wrr \right\|_{L^2}^2+ C\left\| \uh \otimes \uh \right\|_{\cFFSLtwo}^2.
 \end{equation*}
 Product rules in Sobolev spaces imply
 \begin{equation*}
 \left\| \uh \otimes \uh \right\|_{\cFFSLtwo}^2 \leqslant \left\| \uh \right\|^2_{{\cFFSHud}} \left\| \uh \right\|^2_{\dot{H}^1 \left( \R^3 \right)},
 \end{equation*}
 whence an integration in time
 \begin{multline*}
 \left\| \Wrr \left( t \right) \right\|_{\cFFSLtwo}^2 + c\int_0^t \left\| \nabla \Wrr \left( \tau \right) \right\|_{\cFFSLtwo}^2 \d \tau
 \\
  \leqslant C_{r,R}\left\| U_0 \right\|_{\cFFSLtwo}^2 + \left\| \uh \right\|^2_{L^\infty\left( \R_+;{\cFFSHud} \right)} \left\| \uh \right\|^2_{L^2\left(\R_+; \dot{H}^1 \left( \R^3 \right) \right)}.
 \end{multline*}
 it suffice hence to use the bounds in Lemma \ref{lem:Ler_sol_ubar_omega} and Proposition \ref{prop:strong_reg_ubar} to deduce the claim.
 \end{proof}
 
 Lemma \ref{lem:L2regularity_delta} provides a first rough bound on $ \drr $ under some rather strong regularity assumptions on the initial data ($ U_0 \in H^{1/2} $ and $ \curlh U_0\in \cFFSLtwo $). Nonetheless such bound shall be required in the proof of Lemma \ref{lem:time_reg_nonlinearity} (see function $ g_3^{r. R} $), which is an important step in the proof of Proposition \ref{prop:convergence}, the main result of the present section. Let us remark moreover that the hypothesis on the initial data of Lemma \ref{lem:L2regularity_delta} are the same as the ones of Proposition \ref{prop:convergence}.

 The following procedure is standard in singular perturbation problems (see \cite{CDGG2}, \cite{monographrotating} and \cite{charve1}). In particular, being the diffusion isotropic we shall follow closely the methodology in \cite{monographrotating}, proving that $ \drr $ is globally well posed in $ \Eud $. If we prove this, as mentioned above, we prove as well that $ U^\varepsilon $ is globally well-posed in the space $ \Eud $, and hence we prove the global-well-posedness part in Theorem \ref{thm:main_result}. \\
 
 Let us at first deduce the equation satisfied by the function $ \drr $. This is a matter of careful algebraic computations, which lead us to deduce the following equations
 \begin{equation}
 \label{eq:eq_delta}
 \left\lbrace
 \begin{aligned}
 & \partial_t \drr - \mathbb{D}\drr +\frac{1}{\varepsilon}\PA \drr = -\frac{1}{\varepsilon} \nabla \tilde{p}^\varepsilon -
 \left( 
 F^\varepsilon_{r, R}
+ G^\varepsilon_{r, R} 
  \right)
- \left( 1-\Psi_{r, R} \left( D \right) \right) \Lambda \left( \uh \right) 
  ,\\
  & \dive \drr =0,\\
 & \left. \drr \right|_{t=0} = \left[ 1 - \Psi_{r, R} \left( D \right) \left( \PP_{+, \varepsilon} + \PP_{-, \varepsilon} \right)  - \mathbb{P}_{0}\right] U_0.
 \end{aligned}
 \right.
 \end{equation}
 Where the modified pressure $ \tilde{p}^\varepsilon = \Phi ^\varepsilon - \varepsilon \bar{p} $ and the nonlinearity is defined as 
 \begin{align*}
 F^\varepsilon_{r, R}= & \ \drr \cdot \nabla \drr + \drr \cdot \nabla \uh   + \drr \cdot \nabla \Wrr 
+ \uh \cdot \nh \drr + w^\varepsilon_{r, R}\cdot \nabla \drr,\\
G^\varepsilon_{r, R} = & \ \uh \cdot \nh w^\varepsilon_{r, R} + w^\varepsilon_{r, R}\cdot \nabla \uh + w^\varepsilon_{r, R} \cdot \nabla w^\varepsilon_{r, R}.
 \end{align*}
 We can now explain why in the equation \eqref{eq:free_wave_CFFS} we introduced artificially the external forcing $ -\Psi_{r, R} \left( D \right) \Lambda $ where $ \Lambda $ is defined in \eqref{eq:def_Lambda}. The pressure $ \bar{p} $ appears with an horizontal gradient only in the equation \eqref{eq:2DNS_stratified}, whence the difference 
 \begin{equation*}
 \nabla \Phi^\varepsilon -\nh \bar{p},
 \end{equation*}
 arising when we compute the difference equation of $ U^\varepsilon-\bar{U} $ 
 is \textit{not the gradient of a scalar function}, being $ \bar{p} $ dependent on the variable $ x_3 $ as it is clear from its expression in terms of the velocity flow $ \uh $ given in \eqref{eq:def_pbar}.\\
 The forcing term $ -\Psi_{r, R} \left( D \right)\Lambda= -\Psi_{r, R} \left( D \right) \left( \PP_{+, \varepsilon}+\PP_{-, \varepsilon} \right) \Lambda $ on the right-hand-side of \eqref{eq:free_wave_CFFS} is hence a corrector term: it adds the intermediate frequencies of $ \partial_3 \bar{p} $ in order to later obtain a full gradient function in the system \eqref{eq:eq_delta} describing the evolution of $ \drr $. Obviously we require an additional corrector which covers the very low and very high frequencies of $ \Lambda $, for this reason it is present in equation \eqref{eq:eq_delta} the term $ 
- \left( 1-\Psi_{r, R} \left( D \right) \right) \Lambda $. We had as well to use the property \eqref{eq:Lambda=P+P-Lambda} in such process. \\

Let us now select a positive real value $ \eta $ such that
\begin{equation}
\label{eq:condition_on_eta}
\eta < \min \set{ \frac{c}{4C}, 1},
\end{equation}
where $ c=\min \set{\nu,\nu'} $.

 We prove now the following technical lemma:

 \begin{lemma}\label{lem:integrbility_forcing_term_LHfreq}
 Let $ U_0 \in {\cFFSHud} $ and  $ P_0 $ be a Fourier multiplier of order 0. Let us consider a $ \eta > 0 $ satisfying \eqref{eq:condition_on_eta}, then there exist a $ 0<r_\eta=r \leqslant R_\eta=R <\infty $ such that the following bound holds true
 \begin{equation*}
 \left\| \left( 1-\Psi_{r, R} \left( D \right) \right) P_0 \left( D \right) \left( \uh \otimes \uh \right) \right\|_{L^2 \left( \R_+ ; \dot{H}^{\frac{1}{2}}\left( \R^3 \right) \right)}\leqslant \frac{\eta}{3C}.
 \end{equation*}
 \end{lemma}
 \begin{proof}
 The proof is an application of Lebesgue dominated convergence theorem. Indeed the function
 \begin{equation*}
 \left| 1-\Psi_{r, R} \left( \xi \right) \right|^2 \left| \xi \right|\left|  P_0  \left( \xi \right) \right|^2 \left| \mathcal{F} \left( \uh \otimes \uh \right) \left( \xi \right) \right|^2,
 \end{equation*}
 converges point-wise to zero when $ r\to 0, R\to \infty $, hence it suffice to prove that 
 $$
  \left| \xi \right|\left|  P_0  \left( \xi \right) \right|^2 \left| \mathcal{F} \left( \uh \otimes \uh \right) \left( \xi \right) \right|^2 \in L^1 \left( \R_+ ; {L^1} \right) .
$$  
   By Plancherel theorem and product rules in Sobolev spaces we deduce
 \begin{align*}
 \int_0^t \int_{\R^3} \left| \xi \right| P_0  \left( \xi \right)^2 \mathcal{F} \left( \uh \otimes \uh \right)^2 \left( t, \xi \right) \d \xi \d t \leqslant & \ C \left\| \uh \otimes \uh \right\|_{L^2 \left( \R_+ ; {\cFFSHud} \right)}^2
 \\
 \leqslant & \ C \left\| \left\| \uh \right\|_{\dot{H}^1}^2 \right\|_{L^2\left( \R_+ \right)}^2
 \\
 \leqslant & \
 C
 \left\| \uh \right\|_{L^\infty \left( \R_+ ; {\cFFSHud} \right)}^2
 \left\| \nabla \uh \right\|_{L^2 \left( \R_+ ; {\cFFSHud} \right)}^2 <\infty,
 \end{align*}
 thanks to the results in Proposition  \ref{prop:strong_reg_ubar}, concluding.
 \end{proof}

Let us analyze now the initial data of the system \eqref{eq:eq_delta}, it is defined as 
 \begin{align}
\delta_{r, R, 0}= & \
 \left[ 1 - \Psi_{r, R} \left( D \right) \left( \PP_{+, \varepsilon} + \PP_{-, \varepsilon} \right)  - \mathbb{P}_{0}\right] U_0,\label{eq:deltarR0}
 \end{align}
 where the projectors $ \PP_{i, \varepsilon} $  defined in \eqref{eq:projectors}, are the projections onto the eigendirections $ E_i^\varepsilon, \ i=0,\pm $ defined in \eqref{eq:eigemvectors_0} and \eqref{eq:eigemvectors_pm}.
The initial data is localized onto the very hi and low frequencies along the eigendirections of the eigenvectors $ E^\pm $ defined in \eqref{eq:eigemvectors_pm}. Unfortunately the projectors $ \PP_{\pm, \varepsilon} $ are \textit{not bounded} on such set of frequencies, hence we cannot deduce directly the regularity of $ \delta_{r, R, 0} $ in terms of the regularity of $ U_0 $. Nonetheless we can prove the following result

\begin{lemma}\label{lem:reg_deltarR0}
Let us fix $ 0<r \leqslant R <\infty $ and let $ \delta_{r, R, 0} $ be the initial data of \eqref{eq:eq_delta} be defined as in \eqref{eq:deltarR0}. For any $ s\in \RR $ if $ U_0 \in \cFFSHs $ there exists a constant $ C $ which does not depend on the parameters $ r, R $ of the localization $ \Crr $ such that
\begin{equation*}
\left\| \delta_{r, R, 0} \right\|_{\cFFSHs}\leqslant C \left\| U_0 \right\|_{\cFFSHs}.
\end{equation*}
\end{lemma}

\begin{proof}
 Let us remark that
$ 1= \PP_0 + \PP_{+, \varepsilon} + \PP_{-, \varepsilon}
$, this in turn implies that
\begin{equation*}
1 - \Psi_{r, R} \left( D \right) \left( \PP_{+, \varepsilon} + \PP_{-, \varepsilon} \right)  - \mathbb{P}_{0}
=
\left( 1-\Psi_{r, R} \left( D \right) \right) \left( 1-\PP_0 \right),
\end{equation*}
whence
\begin{equation*}
\delta_{r, R, 0}= \left( 1-\Psi_{r, R} \left( D \right) \right) \left( 1-\PP_0 \right) U_0.
\end{equation*}
The projector $ \PP_0 $ has been evaluated in detail in \eqref{eq:definition_P0}, and in particular it is a Fourier multiplier of order zero. This implies that the operator $ \left( 1-\Psi_{r, R} \left( D \right) \right) \left( 1-\PP_0 \right) $ is as well a Fourier multiplier of order zero, such that
\begin{equation*}
\norm{\left( 1-\Psi_{r, R} \left( D \right) \right) \left( 1-\PP_0 \right)}_{\mathcal{L}\pare{\Hs}} \leqslant \norm{ 1-\PP_0 }_{\mathcal{L}\pare{\Hs}} \leqslant C < \infty.
\end{equation*}
  
\end{proof}

 Selecting hence an $ \eta > 0 $ which satisfies \eqref{eq:condition_on_eta} and an $ U_0 \in {\cFFSHud} $ 
 Lemma \ref{lem:reg_deltarR0} and a dominated convergence argument allow us hence to choose some positive, real $ 0<r<R $ which depend on $ \eta $  such that
 \begin{equation}
 \label{eq:hyp_smallness_initial_data}
 \left\| \delta_{r, R, 0} \right\|_{{\cFFSHud}}< \frac{\eta}{3 C}.
 \end{equation}

The result we prove in this section is the following one:
\begin{prop}
\label{prop:convergence}
Let us consider a positive, real $ \eta $ which satisfies \eqref{eq:condition_on_eta}, and let the initial data $ U_0 \in \cFFSLtwo \cap {\cFFSHud} $ be such that $ \oh_0= -\partial_2 U_0^1 +\partial_1 U_0^2 \in \cFFSLtwo $.
Let us set $ 0< r \ll 1 \ll R $ positive parameters depending on $ \eta $ be such that $ \delta_{r, R, 0} $ defined in \eqref{eq:deltarR0} satisfies \eqref{eq:hyp_smallness_initial_data}.  Let  $ \left( \drr \right)_{\varepsilon > 0} $ be a sequence indexed by $ \varepsilon $  of local solutions of \eqref{eq:eq_delta}, there exists a 
\begin{equation*}
\varepsilon_0=\varepsilon_0\pare{\eta}=\frac{1}{C} \pare{\frac{\eta^2}{3 C_{r, R}}}^8
\end{equation*}
such that for each $ \varepsilon\in \pare{0, \varepsilon_0} $ and $ t\in\left[0, T_{U_0} \right] $
\begin{equation}
\label{eq:bound_norm_E1/2_delta}
\norm{ \drr \left( t \right)}^2_{\cFFSHud}
+
c
\int_0^t \norm{ \nabla \drr \left( \tau \right)}^2_{\cFFSHud}\d \tau
\leqslant \eta^2,
\end{equation}
where $ c=\min\set{\nu, \nu'} $ and $ T_{U_0} $ is the maximal lifespan of $ U^\varepsilon $ given in Theorem \ref{thm:FK_applied}. 
\end{prop}
 
 The proof of the above proposition consists in a bootstrap argument. 
 The main step in order to prove such bootstrap argument is an energy bound on the nonlinearity $ F^\varepsilon_{r, R} + G^\varepsilon_{r, R} $. This is formalized in the following lemma:
 
 \begin{lemma}\label{lem:bilinear bounds}
 The following bounds hold true
 \begin{align*}
 \left| \left(\left. \drr \cdot \nabla \drr \right| \drr  \right)_{{\cFFSHud}} \right| 
 \leqslant & \ C \left\| \drr \right\|_{{\cFFSHud}} \left\| \nabla \drr \right\|_{{\cFFSHud}}^2,\\
 \left| \left(\left. \dive \left( \drr \otimes \left( \uh + \Wrr \right) \right) \right| \drr  \right)_{{\cFFSHud}} \right|\leqslant & \ C 
 \left( \left\| \uh \right\|_{{\cFFSHud}}^{1/2}  \left\| \nabla \uh \right\|_{{\cFFSHud}}^{1/2}
 +
 \left\| \Wrr \right\|_{{\cFFSHud}}^{1/2}  \left\| \nabla \Wrr \right\|_{{\cFFSHud}}^{1/2} \right)\\
 & \ \times
 \left\| \drr \right\|_{{\cFFSHud}}^{1/2}  \left\| \nabla \drr \right\|_{{\cFFSHud}}^{3/2},\\
 \left| \left(\left. \uh\cdot \nh \Wrr \right| \drr  \right)_{{\cFFSHud}} \right| \leqslant & \ 
 C_{r, R} \left\| \nabla \drr \right\|_{{\cFFSLtwo}} \left\| \Wrr \right\|_{\Linfty} \left\| \uh \right\|_{{\cFFSLtwo}},\\
 \left| \left(\left.  w^\varepsilon  _{r,R}\cdot \nabla \uh  \right| \drr  \right)_{{\cFFSHud}}  \right|
 \leqslant 
 & \
 C \ \left\| \Wrr \right\|_{\Linfty} \left\| \uh \right\|_{{\cFFSHud}}^{1/2} \left\| \nabla \uh \right\|_{{\cFFSHud}}^{1/2}\\
& \  \times
 \left\| \drr \right\|_{{\cFFSHud}}^{1/2} \left\| \nabla \drr \right\|_{{\cFFSHud}}^{1/2},
 \\
 \left| \left(\left.  w^\varepsilon  _{r,R}\cdot \nabla \Wrr  \right| \drr  \right)_{{\cFFSHud}}  \right|
 \leqslant 
 & \
 C_{r, R} \left\| \Wrr \right\|_{\Linfty} \left\| \Wrr \right\|_{{\cFFSLtwo}} \left\| \drr \right\|_{{\cFFSHud}}.
 \end{align*}
 \end{lemma}
 
 Thanks to the above bounds we can deduce the following bounds for the nonlinearity $ F^\varepsilon_{r, R} + G^\varepsilon_{r, R} $, which shall be the ones used in the proof of the bootstrap argument
 \begin{lemma} \label{lem:time_reg_nonlinearity}
 The following bounds hold true
 \begin{align}
 \label{bound_F}
  &\begin{multlined}
 \left| \left(\left. F^\varepsilon_{r, R} \right| \drr   \right)_{{\cFFSHud}} \right|  
 \leqslant
  \left( \frac{c}{16} + C \left\| \drr \right\|_{{\cFFSHud}} \right)
  \left\| \nabla \drr \right\|_{{\cFFSHud}}^2 + f_{r, R} \ \left\| \drr \right\|_{{\cFFSHud}}^2,
 \end{multlined}\\
  \label{bound_G}
 & \begin{multlined}
  \left| \left(\left. G^\varepsilon_{r, R} \right| \drr   \right)_{{\cFFSHud}} \right|  
 \leqslant
 \frac{c}{16} \left\| \nabla \drr \right\|_{{\cFFSHud}}^2 
 + g_{1, \varepsilon}^{r, R} \ \left\| \drr \right\|_{{\cFFSHud}}^2
 \\[4mm]
 + \pare{ g_2 ^{r, R} + g_3^{r, R}} \ \left\| \Wrr \right\|_{\Linfty}
 +g_4 \ \left\| \Wrr \right\|_{\Linfty}^2,
  \end{multlined}
 \end{align}
 where
 \begin{equation*}
  \begin{aligned}
 f_{r, R} \left( t \right) = & \ C \left( \left\| \uh \left( t \right) \right\|_{{\cFFSHud}}^{2}  \left\| \nabla \uh \left( t \right) \right\|_{{\cFFSHud}}^{2}\right.\\
 & \
 +
 \left.\left\| \Wrr \left( t \right) \right\|_{{\cFFSHud}}^{2}  \left\| \nabla \Wrr \left( t \right) \right\|_{{\cFFSHud}}^{2}\right), 
 && \in L^1 \left( \R_+ \right) ,\\
 g_{1, \varepsilon}^{r, R} \left( t \right) = & \ C
 \left\| \nabla \uh \left( t \right) \right\|_{{\cFFSHud}}^2+ 
 C_{r, R}  \left\| \Wrr \right\|_{{\cFFSLtwo}}\left\| \Wrr \right\|_{\Linfty},
   &&\in L^1 \left( \R_+ \right) ,\\
 g_2^{r, R}\left( t \right) = & \ 
 C_{r, R}  \left\| \Wrr \right\|_{{\cFFSLtwo}}  , 
 && \in L^\infty \left( \R_+ \right),\\
 g_3^{r, R}\left( t \right) = & \ 
 C_{r, R} \norm{ \uh }_{{\cFFSLtwo}} \norm{ \nabla \drr}_{{\cFFSLtwo}}
 && \in L^2 \pare{ \R_+}, 
\\ 
 g_4 \left( t \right)= & \ C \left\| \uh \left( t \right) \right\|_{{\cFFSHud}}, 
 && \in L^\infty \left( \R_+ \right).
 \end{aligned}
  \end{equation*}
  Moreover if $ U_0\in {\cFFSHud} $ then $ f = f_{r, R} \in L^1 \left( \R_+ \right) $ uniformly with respect to the parameters $ r, R $. For $ 0< \varepsilon< \varepsilon_0 \left( r, R \right) $ the function $ g_1 = g_{1, \varepsilon}^{r, R} $ belongs to $L^1 \left( \R_+ \right) $ uniformly with respect to the parameters $ r, R $.
 \end{lemma}
 The proofs of Lemmas \ref{lem:bilinear bounds} and \ref{lem:time_reg_nonlinearity} are postponed.\\

 We can  prove now the  result stated in Proposition \ref{prop:convergence}.
 \\
 \textit{Proof of Proposition \ref{prop:convergence}:} Let us perform an $ {\cFFSHud} $ energy estimate onto the system \eqref{eq:eq_delta}, we indeed deduce that
 \begin{multline}\label{stima1}
 \frac{1}{2}\frac{\d }{\d t } \left\| \drr \left( t \right) \right\|_{{\cFFSHud}}^2 + 
 c \left\| \nabla \drr \left( t \right) \right\|_{{\cFFSHud}}^2
\\ 
  \leqslant
 \left| \left(\left. F^\varepsilon_{r, R} \left( t \right)  \right| \drr  \left( t \right)  \right)_{{\cFFSHud}} \right| +
 \left| \left(\left. G^\varepsilon_{r, R} \left( t \right)  \right| \drr  \left( t \right)  \right)_{{\cFFSHud}} \right|
 \\
+
\left| \left( \left. \left( 1-\Psi_{r, R} \left( D \right) \right) \Lambda \left( \uh \right) \right| \drr \right)_{{\cFFSHud}} \right| 
 .
 \end{multline}
 Thanks to the explicit definition of $ \Lambda $ given in \eqref{eq:def_Lambda} an integration by parts and young inequality we deduce
 \begin{multline} \label{stima2}
 \left| \left( \left. \left( 1-\Psi_{r, R} \left( D \right) \right) \Lambda \left( \uh \right) \right| \drr \right)_{{\cFFSHud}} \right|
 \\
 \begin{aligned}
 \leqslant & \ \left| \left( \left. \left( 1-\Psi_{r, R} \left( D \right) \right) \left( -\Dh \right)^{-1}\diveh \diveh \left( \uh\otimes\uh \right)  \right| \partial_3 \drr \right)_{{\cFFSHud}} \right|,\\
 \leqslant & \frac{c}{16} \left\| \nabla \drr \right\|_{{\cFFSHud}}^2 + \left\| \left( 1-\Psi_{r, R} \right) P_0 \left( D \right) \left( \uh\otimes\uh \right)  \right\|_{{\cFFSHud}}^2,
  \end{aligned}
 \end{multline}
 where we denoted $ P_0 \left( D \right)= \left( -\Dh \right)^{-1}\diveh \diveh $.\\
 With the bounds \eqref{bound_F}, \eqref{bound_G} and \eqref{stima2} the  equation \eqref{stima1} becomes
 \begin{multline}\label{eq1}
 \frac{1}{2}\frac{\d }{\d t } \left\| \drr \left( t \right) \right\|_{{\cFFSHud}}^2 + 
\left( \frac{3\ c}{4} - C \left\| \drr \left( t \right) \right\|_{{\cFFSHud}} \right) 
  \left\| \nabla \drr \left( t \right) \right\|_{{\cFFSHud}}^2
  \\
  \leqslant
 \left(  f\left( t \right) + g_1 \left( t \right) \right) \ \left\| \drr \left( t \right) \right\|_{{\cFFSHud}}^2
  + \left( g_2 ^{r, R}\left( t \right)  + g_3 ^{r, R}\left( t \right)  \right)\ \left\| \Wrr \left( t \right) \right\|_{\Linfty}\\
 +g_4 \left( t \right) \ \left\| \Wrr \left( t \right) \right\|_{\Linfty}^2
+g_5^{r, R} \left( t \right) 
 ,
 \end{multline}
 where 
\begin{equation}
\label{g5}
g_5^{r, R}= \left\| \left( 1-\Psi_{r, R} \right) P_0 \left( \uh\otimes\uh \right)  \right\|_{{\cFFSHud}}^2.
\end{equation} 
 We omit the dependence of $ f $ and $ g_1 $ on the parameters $ r, R, \varepsilon $ in light of the results of Lemma \ref{lem:time_reg_nonlinearity}.\\
 
 Let us define at this point the time
 \begin{equation*}
 T^\star = \sup \set{  0 < t \leqslant T_{U_0} \ \left| \ 
 \left\| \drr \left( t \right) \right\|_{{\cFFSHud}} < \frac{c}{4 C}
 \right. },
 \end{equation*}
 where $ T_{U_0} $ is the maximal lifespan of $ U^\varepsilon $ defined in Theorem \ref{thm:FK_applied}.\\
 Moreover for each $ t\in \left[ 0 , T^\star\right] $, thanks of the definition of $ T^\star $, we can deduce that
 \begin{equation*}
 \frac{3\ c}{4} - C \left\| \drr \left( t \right) \right\|_{{\cFFSHud}} \geqslant \frac{c}{2} 
  ,
 \end{equation*}
 from which, combined with 
  \eqref{eq1} we can deduce:
  \begin{multline}\label{eq2}
 \frac{1}{2}\frac{\d }{\d t } \left\| \drr \left( t \right) \right\|_{{\cFFSHud}}^2 + 
\frac{c}{2} 
  \left\| \nabla \drr \left( t \right) \right\|_{{\cFFSHud}}^2
  \\
  \leqslant
 \left(  f\left( t \right) + g_1 \left( t \right) \right) \ \left\| \drr \left( t \right) \right\|_{{\cFFSHud}}^2
  + \left( g_2 ^{r, R}\left( t \right) + g_3 ^{r, R}\left( t \right) \right) \ \left\| \Wrr \left( t \right) \right\|_{\Linfty}
  \\
 +g_4 \left( t \right) \ \left\| \Wrr \left( t \right) \right\|_{\Linfty}^2  +g_5^{r, R} \left( t \right) .
 \end{multline}
 Let us set
 \begin{equation*}
 \Xi \left( t \right) = - 2 \int_0^t \left(  f\left( \tau \right) + g_1 \left( \tau \right) \right) \d \tau,
 \end{equation*}
 and let us remark that, since $ f, g_1\in L^1 \left( \R_+ \right) $, then $ \Xi, e^{\pm \Xi}\in L^\infty \left( \R_+ \right) $, and moreover
\begin{align}\label{eq_exp_bounds}
e^{-\Xi \left( t \right)} \geqslant & \ e^{- \left\| \Xi \right\|_{L^\infty \left( \R_+ \right)}}, &
e^{\Xi \left( t \right)} \leqslant & \ e^{ \left\| \Xi \right\|_{L^\infty \left( \R_+ \right)}}.
\end{align} 
Standard calculation on \eqref{eq2} and  integration-in-time imply that
 \begin{multline*}
 \left\| \drr \left( t \right) \right\|_{{\cFFSHud}}^2 
 + c \int_0^t e^{ - \left( \Xi \left( t \right) - \Xi \left( \tau \right) \right)}
 \left\| \nabla \drr \left( \tau \right) \right\|_{{\cFFSHud}}^2 \d \tau
 \\
  \leqslant
   e^{-\Xi \left( t \right)} \left\| \delta_{r, R, 0} \right\|_{{\cFFSHud}}^2
  +
  C \int_0^t e^{ - \left( \Xi \left( t \right) - \Xi \left( \tau \right) \right)}
  \left( \left( g_2 ^{r, R}\left( \tau \right) + g_3 ^{r, R}\left( \tau \right) \right) \ \left\| \Wrr \left( \tau \right) \right\|_{\Linfty}\right.
  \\
  \left. + g_4\left( \tau \right) \ \left\| \Wrr \left( \tau \right) \right\|_{\Linfty}^2
+g_5^{r, R} \left( \tau \right)   
   \right)\d \tau,
 \end{multline*}
 whence by the use of \eqref{eq_exp_bounds} we deduce
 \begin{multline} \label{eq:eq_BS1}
 \left\| \drr \left( t \right) \right\|_{{\cFFSHud}}^2 
 + c \int_0^t 
 \left\| \nabla \drr \left( \tau \right) \right\|_{{\cFFSHud}}^2 \d \tau
 \\
  \leqslant C
  \left\| \delta_{r, R, 0} \right\|_{{\cFFSHud}}^2
  +
  C \int_0^t 
  \left( \left( g_2 ^{r, R}\left( \tau \right) + g_3 ^{r, R}\left( \tau \right) \right) \ \left\| \Wrr \left( \tau \right) \right\|_{\Linfty}\right.
  \\
  \left. + g_4\left( \tau \right) \ \left\| \Wrr \left( \tau \right) \right\|_{\Linfty}^2 
+g_5^{r, R} \left( \tau \right)  
  \right)\d \tau.
 \end{multline}
 Moreover since $ g^{r,R}_2, g_4 \in L^\infty \left( \R_+ \right) $, $ g_3 ^{r, R} \in L^2 \left( \R_+ \right) $ and thanks to the estimates \eqref{eq:Strichartz_estimates} we deduce
 \begin{multline}\label{stima4}
 C \int_0^t 
  \left( \left( g_2 ^{r, R}\left( \tau \right) + g_3 ^{r, R}\left( \tau \right) \right) \ \left\| \Wrr \left( \tau \right) \right\|_{\Linfty} + g_4\left( \tau \right) \ \left\| \Wrr \left( \tau \right) \right\|_{\Linfty}^2 \right)\d \tau 
  \\
\begin{aligned}
  \leqslant& \ C \left( \left\| \Wrr \right\|_{L^1\left( \R_+; L^\infty\left( \R^3 \right) \right)} +
\left\| \Wrr \right\|_{L^2\left( \R_+; L^\infty\left( \R^3 \right) \right)}  
   +
   \left\| \Wrr \right\|_{L^2\left( \R_+; L^\infty\left( \R^3 \right) \right)}^2 \right),\\
  \leqslant & \ C_{r, R} \left( \varepsilon^{1/4} + \varepsilon^{1/8} \right), 
  \end{aligned}  
 \end{multline}
 for $ \varepsilon < \varepsilon_0 $ positive and sufficiently small. 
In light of the definition of $ g_5^{r, R} $ given in \eqref{g5} and Lemma \ref{lem:integrbility_forcing_term_LHfreq} we deduce
\begin{equation}\label{stima3}
C\int_0^\infty g_5^{r, R}\left( \tau \right) \d \tau \leqslant \frac{\eta^2}{3},
\end{equation}
 
 The  bound \eqref{stima4}, \eqref{stima3} and \eqref{eq:hyp_smallness_initial_data} transform \eqref{eq:eq_BS1} into
 \begin{equation*}
 \left\| \drr \left( t \right) \right\|_{{\cFFSHud}}^2 
 + c \int_0^t 
 \left\| \nabla \drr \left( \tau \right) \right\|_{{\cFFSHud}}^2 \d \tau 
 \leqslant \frac{2}{3} \ \eta^2 + C_{r, R} \left( \varepsilon^{1/4} + \varepsilon^{1/8} \right).
 \end{equation*}
 Moreover if $  \varepsilon \lesssim \displaystyle \pare{\frac{\eta^2}{3C_{r, R}}}^8 $ we deduce 
 \begin{equation*}
 \frac{2}{3} \ \eta^2 + C_{r, R} \left( \varepsilon^{1/4} + \varepsilon^{1/8} \right) <\eta^2,
 \end{equation*}
 and hence the bound is independent from the time variable, whence we deduce that $ T^\star = T_{U_0} $ and we prove the claim.
 \hfill $ \Box $

 Proposition \ref{prop:convergence} allows us to prove that, for $ \varepsilon $ sufficiently close to zero, $ U^\varepsilon $ solution of \eqref{perturbed BSSQ} is globally well posed by a standard procedure, which is formalized in the following corollary.
 
 \begin{cor}
 Let $ \eta, r=r_\eta, R=R_\eta, \varepsilon_0=\varepsilon_0\pare{\eta}, U_0 $ be as in the statement of Proposition \ref{prop:convergence}, then for each $ \varepsilon \in \pare{0, \varepsilon_0} $
 \begin{equation*}
 U^\varepsilon \in \Eud\cap L^4 \pare{\RR_+; \dot{H}^1 \pare{\RR^3}}.
 \end{equation*}
 \end{cor}
 
 \begin{proof}
 We proved respectively in Proposition \ref{prop:strong_reg_ubar} and Lemma \ref{lem:existence_sol_free_wave} that $ \bar{U} $ and $ \Wrr $ belong to $ \Eud $. Moreover Proposition  \ref{prop:convergence} asserts that
 \begin{equation*}
 \norm{\drr}_{\dot{\mathcal{E}}^{1/2}_{T_{U_0}}\pare{\RR^3}}\leqslant \eta.
 \end{equation*}
 Whence for each $ T\in \left[0, T_{U_0}\right) $
 \begin{equation*}
 \norm{U^\varepsilon}_{\dot{\mathcal{E}}^{1/2}_{T}\pare{\RR^3}}\leqslant C_{r, R}< \infty.
 \end{equation*}
 Moreover, since for each $ T\in \left[0, T_{U_0}\right) $ the space $ \dot{\mathcal{E}}^{1/2}_{T}\pare{\RR^3} $ is continuously embedded in $ L^4 \pare{[0,T];  \dot{H}^1 \pare{\RR^3}} $ we deduce that
 \begin{equation*}
 \limsup_{T \nearrow T_{U_0}}  \norm{U^\varepsilon}_{L^4 \pare{[0,T];  \dot{H}^1 \pare{\RR^3}} } \leqslant C_{r, R} < \infty,  
 \end{equation*}
 which indeed is a contradiction of the blow-up criterion \eqref{eq:BU_condition}, whence with a continuation argument we deduce that
 \begin{equation*}
 T_{U_0}=\infty.
 \end{equation*}
 \end{proof}

 \subsection{Proof of Lemma \ref{lem:bilinear bounds}} The first bound is a simple application of the definition \eqref{eq:Hs_sp} and of Lemma \ref{lem:prod_rules_Sobolev}
 \begin{align*}
 \left| \left(\left. \drr \cdot \nabla \drr \right| \drr  \right)_{{\cFFSHud}} \right| = & \norm{ \drr \otimes \drr }_{{\cFFSHud}} \norm{ \nabla \drr}_{{\cFFSHud}},\\
 \leqslant & \ C \norm{  \drr}_{{\cFFSHud}}^2 \norm{ \nabla \drr}_{{\cFFSHud}}.
 \end{align*}
 The estimate is derived by interpolation of Sobolev spaces.\\
 For the second estimate 
 \begin{multline*}
 \left| \left(\left. \dive \left( \drr \otimes \left( \uh + \Wrr \right) \right) \right| \drr  \right)_{{\cFFSHud}} \right|\\
 \begin{aligned}
  \leqslant & \  C \norm{ \drr \otimes \left( \uh + \Wrr \right)}_{{\cFFSHud}} \norm{ \nabla \drr}_{{\cFFSHud}},\\
 \leqslant & \  C \norm{ \drr}_{\dot{H}^1 \left( \R^3 \right)} \norm{ \left( \uh + \Wrr \right)}_{\dot{H}^1 \left( \R^3 \right)} \norm{ \nabla \drr}_{{\cFFSHud}},
  \end{aligned}
 \end{multline*}
 an interpolation of Sobolev spaces and triangular inequality conclude the second estimate.\\
 For the next term 
 \begin{align*}
 \left| \left(\left. \uh\cdot \nh \Wrr \right| \drr  \right)_{{\cFFSHud}} \right| \leqslant & \ \norm{ \uh \cdot \nabla \Wrr}_{{\cFFSLtwo}} \norm{\nabla \drr}_{{\cFFSLtwo}},\\
 \leqslant & \ C_{r, R} \norm{ \uh}_{{\cFFSLtwo}} \norm{ \Wrr}_{L^\infty \left( \R^3 \right)} \norm{\nabla \drr}_{{\cFFSLtwo}},
 \end{align*}
 where in the last inequality we applied H\"older inequality and Bernstein inequality.
 For the last term it suffice to remark that the function $  w^\varepsilon  _{r,R}\cdot \nabla \Wrr  $ is well-defined and still localized in the Fourier space, hence apply H\"older and Bernstein inequalities.
 
 \subsection{Proof of Lemma \ref{lem:time_reg_nonlinearity}} To deduce the bound \eqref{bound_F} and \eqref{bound_G} it suffice to apply repeatedly Young inequality to the bounds of Lemma \ref{lem:bilinear bounds}, in detail:\\
 applying the convexity inequality $ \alpha \ \beta \leqslant \frac{c}{16} \alpha^{4/3} + C \ \beta^4 $ we deduce
 \begin{multline*}
 \left| \left(\left. \dive \left( \drr \otimes \left( \uh + \Wrr \right) \right) \right| \drr  \right)_{{\cFFSHud}} \right|\\
\begin{aligned}
 \leqslant & \ C 
 \left( \left\| \uh \right\|_{{\cFFSHud}}^{1/2}  \left\| \nabla \uh \right\|_{{\cFFSHud}}^{1/2}
 +
 \left\| \Wrr \right\|_{{\cFFSHud}}^{1/2}  \left\| \nabla \Wrr \right\|_{{\cFFSHud}}^{1/2} \right)\\
& \hspace{6cm} \times \left\| \drr \right\|_{{\cFFSHud}}^{1/2}  \left\| \nabla \drr \right\|_{{\cFFSHud}}^{3/2}\\
\leqslant & \  \frac{c}{16}\left\| \nabla \drr \right\|_{{\cFFSHud}}^2\\
& +  \ C 
 \left( \left\| \uh \right\|_{{\cFFSHud}}^{2}  \left\| \nabla \uh \right\|_{{\cFFSHud}}^{2}
 +
 \left\| \Wrr \right\|_{{\cFFSHud}}^{2}  \left\| \nabla \Wrr \right\|_{{\cFFSHud}}^{2} \right) \left\| \drr \right\|_{{\cFFSHud}}^2,
\end{aligned} 
 \end{multline*}
 and hence we set 
 \begin{equation*}
 f_{r, R} = \ C 
 \left( \left\| \uh \right\|_{{\cFFSHud}}^{2}  \left\| \nabla \uh \right\|_{{\cFFSHud}}^{2}
 +
 \left\| \Wrr \right\|_{{\cFFSHud}}^{2}  \left\| \nabla \Wrr \right\|_{{\cFFSHud}}^{2} \right),
 \end{equation*}
 obtaining the bound \eqref{bound_F}.\\

 Next we prove \eqref{bound_G}. In the third inequality of Lemma \ref{lem:time_reg_nonlinearity} we proceed as follows
 \begin{equation*}
 \begin{aligned}
\left| \left(\left. \uh\cdot \nh \Wrr \right| \drr  \right)_{{\cFFSHud}} \right|
\leqslant & \ 
 C_{r, R} \left\| \nabla \drr \right\|_{{\cFFSLtwo}}  \left\| \uh \right\|_{{\cFFSLtwo}} \left\| \Wrr \right\|_{\Linfty}\\
 = & \ g_3^{r, R} \left\| \Wrr \right\|_{\Linfty}.
 \end{aligned}
  \end{equation*}

Next, in the fourth inequality of Lemma \ref{lem:time_reg_nonlinearity} we apply the inequality
\begin{equation*}
\alpha \ \beta \ \gamma \leqslant \frac{c}{64} \ \alpha^4 + C \ \beta^4 + C \ \gamma^2
\end{equation*}
in order to deduce the following inequality
\begin{multline*}
\left| \left(\left.  w^\varepsilon  _{r,R}\cdot \nabla \uh  \right| \drr  \right)_{{\cFFSHud}}  \right|\\
\begin{aligned}
\leqslant & \ 
C \ \left\| \Wrr \right\|_{\Linfty} \left\| \uh \right\|_{{\cFFSHud}}^{1/2} \left\| \nabla \uh \right\|_{{\cFFSHud}}^{1/2}
 \left\| \drr \right\|_{{\cFFSHud}}^{1/2} \left\| \nabla \drr \right\|_{{\cFFSHud}}^{1/2}\\
 \leqslant & \ \frac{c}{64} \left\| \nabla \drr \right\|_{{\cFFSHud}}^2   
 + C \left\| \nabla \uh \right\|_{{\cFFSHud}}^2 \left\| \drr \right\|_{{\cFFSHud}}^2\\
& \hspace{4cm} + C \ \left\| \uh \right\|_{{\cFFSHud}}  \left\| \Wrr \right\|_{\Linfty} ^2, 
\end{aligned}
\end{multline*}
hence we set
\begin{align*}
g_4 = & \ C \ \left\| \uh \right\|_{{\cFFSHud}} ,\\
g_{1, \RN{1}}^{r, R} = & \  C \left\| \nabla \uh \right\|_{{\cFFSHud}}^2.
\end{align*}

For the last inequality it suffice to remark that
\begin{multline*}
\left| \left(\left.  w^\varepsilon  _{r,R}\cdot \nabla \Wrr  \right| \drr  \right)_{{\cFFSHud}}  \right|
 \leqslant 
  \
 C_{r, R} \left\| \Wrr \right\|_{\Linfty} \left\| \Wrr \right\|_{{\cFFSLtwo}} \left\| \drr \right\|_{{\cFFSHud}}^2\\
 +C_{r, R} \left\| \Wrr \right\|_{\Linfty} \left\| \Wrr \right\|_{{\cFFSLtwo}},
\end{multline*}
whence we set 
\begin{align*}
g_{1, \RN{2}, \varepsilon}^{r, R}= & \ C_{r, R} \left\| \Wrr \right\|_{\Linfty} \left\| \Wrr \right\|_{{\cFFSLtwo}},\\
g_2^{r, R} = & \ C_{r, R}  \left\| \Wrr \right\|_{{\cFFSLtwo}}.
\end{align*}
Lastly we finally define
\begin{equation*}
g_{1, \varepsilon}^{r, R}= g_{1, \RN{1}}^{r, R} +
g_{1, \RN{2}, \varepsilon}^{r, R},
\end{equation*}
and we deduce the bound \eqref{bound_G}.\\

 The function $ f_{r, R} $ belongs indeed to $ L^1 \left( \R_+ \right) $ uniformly with respect to $ r, R $ thanks to the result in Proposition \ref{prop:strong_reg_ubar} and Lemma \ref{lem:existence_sol_free_wave}.\\
 For the function $ g^{r, R}_{1, \varepsilon} $ it suffice to integrate in time and to use the result in Proposition \ref{prop:strong_reg_ubar} and \eqref{eq:Strichartz_estimates} to obtain
 \begin{align*}
 \norm{ g^{r, R}_{1, \varepsilon}}_{L^1 \left( \R_+ \right)}
 \leqslant  & \
 C \norm{ \nabla \uh }_{L^2 \left( \R_+; {\cFFSHud} \right)}^2 
 +
 C_{r, R}
  \left\| \Wrr \right\|_{L^\infty \left( \R_+;{\cFFSLtwo}\right)}\left\| \Wrr \right\|_{L^1\left(\R_+;  \Linfty \right)}\\
  \leqslant & \ C + C_{r, R}\ \varepsilon^{1/4},\\
  < & \ \infty,
 \end{align*}
 if $ \varepsilon $ is sufficiently small.

\section{Proof of the main result} \label{sec:main_result}

Section \ref{sec:bootstrap} gives us all the ingredients required in order to prove the main result of the present paper, namely Theorem \ref{thm:main_result}. Remarkably the statement in Theormem \ref{thm:main_result} and Proposition \ref{prop:convergence} are very similar: the difference is that $ W^\varepsilon $ solution of \eqref{eq:free_wave_total_initial_data} does not depend on the parameters $ r, R $ as $ \Wrr $ solution of \eqref{eq:free_wave_CFFS}. Let us hence define 
\begin{equation*}
\delta^\varepsilon = U^\varepsilon - W^\varepsilon - \bar{U}. 
\end{equation*}

In Section \ref{sec:dispersive_properties} we focused on  existence, regularity and dispersive results for $ \Wrr $, but no result was proved for $ W^\varepsilon $. Namely the initial data of the system \eqref{eq:free_wave_total_initial_data}, which is solved by $ W^\varepsilon $, is not any more localized in the frequency space. The estimate \eqref{eq:est_H1/2_Wrr} hence does not hold true any more, in particular the bound 
\begin{equation*}
\left\| \left. \Wrr\right|_{t=0} \right\|_{{\cFFSHud}}^2 \leqslant C_{r, R} \left\| U_0 \right\|_{{\cFFSHud}},
\end{equation*}
is false for initial data which are not localized as for $ W^\varepsilon $.
 Fortunately we can extend the result of Lemma \ref{lem:existence_sol_free_wave} to the system \eqref{eq:free_wave_total_initial_data} with an argument very similar to the one given in the proof of Lemma \ref{lem:reg_deltarR0}. We omit a detailed proof here, but it suffice to remark that the operator $ \PP_{+, \varepsilon}+\PP_{-, \varepsilon}= 1-\PP_0 $, and that the operator $ 1-\PP_0 $ is continuous in any $ \cFFSHs $ space. We hence showed that, if $ U_0\in {\cFFSHud} $,
 \begin{equation*}
 W^\varepsilon \in \Eud,
 \end{equation*}
for each $ \varepsilon > 0 $. We shall use this property continuously in what follows.

Let us fix now an $ \eta > 0 $ such that satisfies \eqref{eq:condition_on_eta} and Proposition \ref{prop:convergence} holds true. Let us moreover choose a $ \varepsilon \in \left[0, \varepsilon_0\pare{\eta} \right) $ where
\begin{equation*}
\varepsilon_0
= \frac{1}{C} \ \frac{\eta^{16}}{C_{r, R}}.
\end{equation*}
Accordingly to the statement of Proposition \ref{prop:convergence} there exist some positive $ 0<r_\eta \leqslant R_{\eta}< \infty $ so that, fixed $ r\in\pare{0, r_\eta} $ and $ R\in\pare{R_\eta, \infty} $ the bound
\begin{equation}\label{c1}
\norm{\drr}_{\Eud}\leqslant \eta, 
\end{equation}
holds uniformly in $ \pare{0, \varepsilon_0} $.

With such setting indeed we have that
\begin{equation}\label{c3}
\delta^{\varepsilon} = \delta^{\varepsilon}_{r, R} - \left( W^\varepsilon -W^{\varepsilon}_{r, R} \right).
\end{equation}
We can hence exploit a dominated convergence argument to argue that, fixed a $ \eta $ as above,   there exists some positive $ 0<r_1\leqslant R_1 <\infty $ so that fixed $ r\in\pare{0, r_1} $ and $ R\in\pare{R_1, \infty} $
\begin{equation}\label{c2}
\norm{W^\varepsilon -W^{\varepsilon}_{r, R}}_{\Eud}\leqslant \eta,
\end{equation} 
uniformly in $ \varepsilon \in \pare{0, \varepsilon_0}$. \\
Let us hence fix some 
\begin{equation*}
\begin{aligned}
r\in\pare{0, \min\set{r_\eta, r_1}},  && 
R \in \pare{\max \set{R_\eta, R_1}, \infty},
\end{aligned}
\end{equation*}
so that for such values the conditions \eqref{c1} and \eqref{c2} are satisfied simultaneously. We hence deduce from \eqref{c3} that
\begin{equation*}
\limsup_{\varepsilon\to 0}\norm{\delta^\varepsilon}_{\Eud}\leqslant 2\eta, 
\end{equation*}
uniformly for each $ \eta $ satisfying \eqref{eq:condition_on_eta}, i.e. for each 
\begin{equation*}
\eta\in \left[ 0, \min\set{\frac{c}{4C}, 1}\right).
\end{equation*}
Whence we let $ \eta\to 0 $ in order to conclude the proof.

\footnotesize{
\providecommand{\bysame}{\leavevmode\hbox to3em{\hrulefill}\thinspace}
\providecommand{\MR}{\relax\ifhmode\unskip\space\fi MR }
\providecommand{\MRhref}[2]{%
  \href{http://www.ams.org/mathscinet-getitem?mr=#1}{#2}
}
\providecommand{\href}[2]{#2}

\bibliographystyle{amsplain}}

\begin{thebibliography}{10}

\bibitem{AlinhacGerard}
Serge Alinhac and Patrick G{\'e}rard, \emph{Op\'erateurs pseudo-diff\'erentiels
  et th\'eor\`eme de {N}ash-{M}oser}, Savoirs Actuels. [Current Scholarship],
  InterEditions, Paris; \'Editions du Centre National de la Recherche
  Scientifique (CNRS), Meudon, 1991.

\bibitem{bahouri_chemin_danchin_book}
Hajer Bahouri, Jean-Yves Chemin, and Rapha{\"e}l Danchin, \emph{Fourier
  analysis and nonlinear partial differential equations}, Grundlehren der
  Mathematischen Wissenschaften [Fundamental Principles of Mathematical
  Sciences], vol. 343, Springer, Heidelberg, 2011.

\bibitem{Bony1981}
Jean-Michel Bony, \emph{Calcul symbolique et propagation des singularit\'es
  pour les \'equations aux d\'eriv\'ees partielles non lin\'eaires}, Annales
  scientifiques de l'\'Ecole Normale Sup\'erieure \textbf{14} (1981), no.~2,
  209--246 (fre).

\bibitem{Charve_thesis}
Fr\'ed\'eric Charve, \emph{{\'Etude de ph\'enom\`enes dispersifs en m\'ecanique
  des fluides g\'eophysiques}}, Ph.D. thesis, \'Ecole Polytechnique, 2004.

\bibitem{charve1}
Fr{\'e}d{\'e}ric Charve, \emph{Global well-posedness and asymptotics for a
  geophysical fluid system}, Comm. Partial Differential Equations \textbf{29}
  (2004), no.~11-12, 1919--1940.

\bibitem{charve2}
\bysame, \emph{Convergence of weak solutions for the primitive system of the
  quasigeostrophic equations}, Asymptot. Anal. \textbf{42} (2005), no.~3-4,
  173--209.

\bibitem{charve_ngo_primitive}
Fr{\'e}d{\'e}ric Charve and Van-Sang Ngo, \emph{Global existence for the
  primitive equations with small anisotropic viscosity}, Rev. Mat. Iberoam.
  \textbf{27} (2011), no.~1, 1--38.

\bibitem{Chemin92}
J.-Y. Chemin, \emph{Remarques sur l'existence globale pour le syst\`eme de
  {N}avier-{S}tokes incompressible}, SIAM J. Math. Anal. \textbf{23} (1992),
  no.~1, 20--28.

\bibitem{chemin_prob_antisym}
Jean-Yves Chemin, \emph{\`{A} propos d'un probl\`eme de p\'enalisation de type
  antisym\'etrique}, J. Math. Pures Appl. (9) \textbf{76} (1997), no.~9,
  739--755.

\bibitem{chemin_book}
\bysame, \emph{Perfect incompressible fluids}, Oxford Lecture Series in
  Mathematics and its Applications, vol.~14, The Clarendon Press, Oxford
  University Press, New York, 1998, Translated from the 1995 French original by
  Isabelle Gallagher and Dragos Iftimie.

\bibitem{CDGG2}
Jean-Yves Chemin, Beno{\^{\i}}t Desjardins, Isabelle Gallagher, and Emmanuel
  Grenier, \emph{Fluids with anisotropic viscosity}, M2AN Math. Model. Numer.
  Anal. \textbf{34} (2000), no.~2, 315--335, Special issue for R. Temam's 60th
  birthday.

\bibitem{CDGG}
\bysame, \emph{Anisotropy and dispersion in rotating fluids}, Nonlinear partial
  differential equations and their applications. {C}oll\`ege de {F}rance
  {S}eminar, {V}ol. {XIV} ({P}aris, 1997/1998), Stud. Math. Appl., vol.~31,
  North-Holland, Amsterdam, 2002, pp.~171--192.

\bibitem{ekmanill}
\bysame, \emph{Ekman boundary layers in rotating fluids}, ESAIM Control Optim.
  Calc. Var. \textbf{8} (2002), 441--466 (electronic), A tribute to J. L.
  Lions.

\bibitem{monographrotating}
\bysame, \emph{Mathematical geophysics}, Oxford Lecture Series in Mathematics
  and its Applications, vol.~32, The Clarendon Press, Oxford University Press,
  Oxford, 2006, An introduction to rotating fluids and the Navier-Stokes
  equations.

\bibitem{cushman2011introduction}
Benoit Cushman-Roisin and Jean-Marie Beckers, \emph{Introduction to geophysical
  fluid dynamics: physical and numerical aspects}, vol. 101, Academic Press,
  2011.

\bibitem{GrenierDesjardins}
Benoit Desjardins and E.~Grenier, \emph{Low {M}ach number limit of viscous
  compressible flows in the whole space}, R. Soc. Lond. Proc. Ser. A Math.
  Phys. Eng. Sci. \textbf{455} (1999), no.~1986, 2271--2279.

\bibitem{Dutrifoy2}
Alexandre Dutrifoy, \emph{Examples of dispersive effects in non-viscous
  rotating fluids}, J. Math. Pures Appl. (9) \textbf{84} (2005), no.~3,
  331--356.

\bibitem{embid_majda2}
Pedro~F. Embid and Andrew~J. Majda, \emph{Low {F}roude number limiting dynamics
  for stably stratified flow with small or finite {R}ossby numbers}, Geophys.
  Astrophys. Fluid Dynam. \textbf{87} (1998), no.~1-2, 1--50. \MR{1626186}

\bibitem{FGN}
Eduard Feireisl, Isabelle Gallagher, and Anton{\'{\i}}n Novotn{\'y}, \emph{A
  singular limit for compressible rotating fluids}, SIAM J. Math. Anal.
  \textbf{44} (2012), no.~1, 192--205.

\bibitem{GallagherSaint-Raymondinhomogeneousrotating}
Isabelle Gallagher and Laure Saint-Raymond, \emph{Weak convergence results for
  inhomogeneous rotating fluid equations}, J. Anal. Math. \textbf{99} (2006),
  1--34.

\bibitem{GiVe}
J.~Ginibre and G.~Velo, \emph{Generalized {S}trichartz inequalities for the
  wave equation}, J. Funct. Anal. \textbf{133} (1995), no.~1, 50--68.

\bibitem{ekmanwell}
E.~Grenier and N.~Masmoudi, \emph{Ekman layers of rotating fluids, the case of
  well prepared initial data}, Comm. Partial Differential Equations \textbf{22}
  (1997), no.~5-6, 953--975.

\bibitem{Lad58}
O.~A. Lady{\v{z}}enskaja, \emph{Solution ``in the large'' to the boundary-value
  problem for the {N}avier-{S}tokes equations in two space variables}, Soviet
  Physics. Dokl. \textbf{123 (3)} (1958), 1128--1131 (427--429 Dokl. Akad. Nauk
  SSSR).

\bibitem{Leray}
Jean Leray, \emph{\'{E}tude de diverses \'equations int\'egrales non
  lin\'eaires et de quelques probl\`emes que pose l'hydrodynamique}, NUMDAM,
  [place of publication not identified], 1933.

\bibitem{LionsProdi}
Jacques-Louis Lions and Giovanni Prodi, \emph{Un th\'eor\`eme d'existence et
  unicit\'e dans les \'equations de {N}avier-{S}tokes en dimension 2}, C. R.
  Acad. Sci. Paris \textbf{248} (1959), 3519--3521.

\bibitem{ekmanperiodic}
Nader Masmoudi, \emph{Ekman layers of rotating fluids: the case of general
  initial data}, Comm. Pure Appl. Math. \textbf{53} (2000), no.~4, 432--483.

\bibitem{VSN}
Van-Sang Ngo, \emph{Rotating fluids with small viscosity}, Int. Math. Res. Not.
  IMRN (2009), no.~10, 1860--1890.

\bibitem{Pedlosky87}
Joseph Pedlosky, \emph{Geophysical {F}luid {D}ynamics}, Springer-Verlag, 1987.

\bibitem{Schwartz54}
Laurent Schwartz, \emph{Sur l'impossibilit\'e de la multiplication des
  distributions}, C. R. Acad. Sci. Paris \textbf{239} (1954), 847--848.

\bibitem{Scrobo_low_Froude_periodic}
Stefano Scrobogna, \emph{Derivation of limit equation for a perturbed 3d
  periodic {B}oussinesq system.}, to appear in DCDS-A.

\bibitem{Stein93}
Elias~M. Stein, \emph{Harmonic analysis: real-variable methods, orthogonality,
  and oscillatory integrals}, Princeton Mathematical Series, vol.~43, Princeton
  University Press, Princeton, NJ, 1993, With the assistance of Timothy S.
  Murphy, Monographs in Harmonic Analysis, III.

\bibitem{Widmayer_Boussisnesq_perturbation}
Klaus Widmayer, \emph{Convergence to stratified flow for an inviscid 3d
  boussinesq system}, \url{http://arxiv.org/abs/1509.09216}.

\end{thebibliography}
 
 \end{document}